\documentclass[a4paper,11pt]{article}
\usepackage{amsmath}
\usepackage{amsthm}

\usepackage[top=3.5cm, bottom=3.5cm, left=2.5cm, right=2.5cm]{geometry}

\usepackage[utf8]{inputenc}
\usepackage[T1]{fontenc}
\usepackage[french,english]{babel} 
\usepackage{enumerate}
\usepackage[shortlabels]{enumitem}
\usepackage{color}
\usepackage[]{pdfpages}

\usepackage{titling}
\usepackage[colorlinks=true, linkcolor=red, linktoc=page,citecolor=blue]{hyperref} 

\usepackage{csquotes}
\usepackage[nottoc, notlof, notlot, numbib]{tocbibind}
\usepackage[style=alphabetic,giveninits,doi=false,isbn=false,url=false,eprint=false,clearlang=true,sorting=nyt,maxbibnames=99]{biblatex}
\usepackage{mathtools}
\usepackage{thmtools}
\usepackage{bbm}
\usepackage{amssymb}
\usepackage{mathrsfs}
\usepackage{float}
\usepackage{stmaryrd}

\usepackage{setspace}
\usepackage{array}
\usepackage{tabularx}
\usepackage{ltablex} 

\DeclareUnicodeCharacter{221E}{$\infty$}
\DeclareUnicodeCharacter{00B4}{\'}
\DeclareUnicodeCharacter{03BB}{$\lambda$}
\DeclareUnicodeCharacter{0393}{$\Gamma$}
\DeclareUnicodeCharacter{211D}{$\mathbb{R}$}
\DeclareUnicodeCharacter{2212}{-} 

\newcommand{\A}{\forall}
\newcommand{\R}{\mathbb{R}}

\newcommand{\E}{\mathbb{E}}
\newcommand{\ps}{\mathcal{P}}

\renewcommand{\d}{\mathrm{d}}
\renewcommand{\P}{\mathbb{P}} 
\renewcommand{\L}{\mathcal{L}} 

\newcommand{\F}{\mathcal{F}}

\newcommand{\M}{\mathcal{M}}

\DeclareMathOperator{\argmin}{argmin}

\newtheorem{theorem}{Theorem}[section]
\newtheorem{proposition}[theorem]{Proposition}
\newtheorem{lemma}[theorem]{Lemma}
\newtheorem{corollary}[theorem]{Corollary}

\theoremstyle{definition}
\newtheorem{definition}[theorem]{Definition}

\theoremstyle{remark}
\newtheorem{rem}[theorem]{Remark}
\newtheorem{example}[theorem]{Example}

\theoremstyle{plain}
\newtheorem{assumptionV}{Assumption}

\usepackage{authblk}

\title{\bf Regularity and stability for the Gibbs conditioning principle on path space via McKean-Vlasov control}

\author[1,2,3]{Louis-Pierre \textsc{Chaintron}}
\author[4]{Giovanni \textsc{Conforti}}
\affil[1]{\small
DMA, École normale supérieure, Université PSL, CNRS, 75005 Paris, France}
\affil[2]{\small
CERMICS, École des Ponts, IP Paris, Marne-La-Vall\'ee, France
}
\affil[3]{\small
Inria, Team M${\sf \Xi}$DISIM, Inria Saclay, 91128 Palaiseau, France
}
\affil[4]{\small
Università degli Studi di Padova, Via Trieste, 63, 35131 Padova, Italia
}

\date{\today}

\begin{document}

\maketitle

\abstract{
We consider a system of diffusion processes interacting through their empirical distribution. Assuming that the empirical average of a given observable can be observed at any time, we derive regularity and quantitative stability results for the optimal solutions in the associated version of the Gibbs conditioning principle.
The proofs rely on the analysis of a McKean-Vlasov control problem with distributional constraints. 
Some new estimates are derived for Hamilton-Jacobi-Bellman equations and the Hessian of the log-density of diffusion processes, which are of independent interest.
}

\tableofcontents

\section{Introduction}

Consider a system of identically distributed random variables 
\begin{equation*}
    X^1,X^2,\ldots,X^N,
\end{equation*}
taking values in a topological space $E$. For example, as it is the case in this article, $E$ may be the space of continuous trajectories over $\R^d$ and the $X^i$ are sample paths from a system of interacting diffusion processes. Furthermore, assume that the variables $X^i$ are either independent or interacting through their empirical distribution (or configuration)
\begin{equation*}
    \mu^N := \frac1N \sum_{i=1}^N  \delta_{X^i} \in \ps ( E ).
\end{equation*}
Next, imagine that some measurement is made on the system configuration revealing that $\mu^N$ belongs to a subset $\mathcal{O}$. 
Then, one is naturally led to the question of computing either the most likely configuration of the system or the law of $X^1$ conditionally on $\{\mu^N \in \mathcal{O}\}$. 
A  detailed answer to both questions can be given when the empirical distribution satisfies a large deviation principle  (henceforth, LDP) with rate function $I$.
This answer is part of a meta-principle, called \emph{Gibbs conditioning principle}, which we now briefly describe.
To this aim, recall that the LDP holds if for any closed set $A\subseteq \mathcal{P}(E)$, 
\begin{equation*}
\limsup_{N\rightarrow+\infty} N^{-1} \log \P ( \mu^N\in A ) \leq  -\inf_A I,
\end{equation*}
and the converse inequality holds for open sets. Then, an informal statement of the Gibbs conditional principle is
\begin{equation*}
\lim_{N\rightarrow +\infty } \mathcal{L}( X^1 | \mu^N\in\mathcal{O} ) =\overline{\mu},\quad \text{with} \quad \overline{\mu} := \argmin_{\mathcal{O}} I.
\end{equation*}
In words, the conditional distribution converges to the minimiser of the rate function among all distributions compatible with the observation.

Equivalently, the distribution $\overline{\mu}$ may be regarded as the most likely configuration of the system conditionally on the observation: this viewpoint is the one adopted by E. Schr\"odinger in a celebrated thought experiment \cite{Schr32} motivating the formulation of the Schr\"odinger problem, which is the prototype of a stochastic mass transport problem. 
In all physical experiments, measurements are affected by uncertainties of different nature. 
Thus, it is of paramount importance to quantify how much the sought most likely evolution or conditional law $\overline{\mu}$ is stable with respect to variations in the observed values. 
The scope of this article is to address this problem and the tightly related question of quantifying the regularity properties of $\overline{\mu}$ when the $X^i$ are $\mathbb{R}^d$-valued interacting diffusion processes such as, e.g.
\begin{equation*}
\d X^i_t = \big[ b(X^i_t)-\frac1N \sum_{j\neq i} \nabla W(X^i_t-X^j_t) \big ] \d t + \sigma(X^i_t) \d B^i_t, \quad t\in[0,T],\,\,i=1,\ldots,N,
\end{equation*}
and what is observed is whether or not at any time $t\in[0,T]$ the value of a certain observable $\Psi(\mu^N_t)$ is below a threshold $\varepsilon$, where $\mu^N_t :=\frac1N\sum_i \delta_{X^i_t}$ is the particle configuration at time $t$. To fix ideas, the observation may be the empirical average of a given observable $\psi$,
\begin{equation*}
\Psi(\mu^N_t) = \frac1N\sum_{i=1}^N\psi(X^i_t),
\end{equation*}
which can be thought of as a mean energy in a physical setting.
However, our framework is more general as we shall see below, encompassing non-linear convex constraints. 
In this setting, the LDP is known to hold with rate function in relative entropy form. 

For clarity of exposition,
let us now describe our results in the simplified setting $W \equiv 0$. 
The LDP is now directly given by Sanov's Theorem.
The most likely evolution $\overline{\mu}$ then reduces to the problem of computing the entropic projection in the sense of Cszisár \cite{csiszar1975divergence} of the path measure $\nu_{[0,T]}$ describing the law of $X^1$ onto the admissible set (clearly, $\nu_{[0,T]}$ does not depend on $N$ when $W \equiv 0$). 
We are therefore looking at the problem
\begin{equation}\label{eq:ent_proj_intro}
 \inf_{\substack{\mu_{[0,T]} \in \ps ( C ( [0,T], \R^d ) ) \\ \forall t \in [0,T], \; \Psi ( \mu_t ) \leq \varepsilon}} H(\mu_{[0,T]} \vert \nu_{[0,T]} ),
\end{equation}
where $H$ is the relative entropy on path space, defined in Section \ref{subsec:Notations} below.
Our approach takes benefit from classical duality results for entropic projections. 
In particular, these results suggest that there exists a non-negative measure $\overline\lambda^{\varepsilon}$ with finite mass on $[0,T]$ playing the role of a Lagrange multiplier and such that the density of the minimiser $\overline\mu^\varepsilon_{[0,T]}$ for \eqref{eq:ent_proj_intro} reads
\begin{equation*}
\frac{\d \overline{\mu}^{\varepsilon}_{[0,T]}}{\d \nu_{[0,T]}} (x_{[0,T]}) = Z_\varepsilon^{-1} \exp \bigg[ -\int_0^T \psi(x_t) \overline\lambda^{\varepsilon}(dt) \bigg],
\end{equation*}
where $Z_\varepsilon$ is a normalising constant.
We then rely on the classical connection between entropy minimisation and stochastic control \cite{follmer1988random,dawsont1987large,dai1991stochastic,budhiraja2012large,leonard2012girsanov}.
This connection allows us to view the optimal measure $\overline\mu^{\varepsilon}_{[0,T]}$ as an optimal state in a stochastic control problem (of McKean-Vlasov type) with distributional constraints.
The optimal control policy $(t,x) \mapsto \nabla \overline\varphi^{\varepsilon}_t(x)$ can be computed from the path-density through the Feynman-Kac formula, see e.g. \cite{leonard2022feynman}. 
Leaving all general and rigorous statements to the next section, let us informally showcase our contributions in the simple though relevant setting we have just described:
\begin{itemize}
\item \emph{\underline{Regularity:}} we show at Theorem \ref{thm:opti} and Theorem \ref{thm:density} that under suitable regularity assumptions on the coefficients:
\begin{itemize}
    \item The optimal marginal flow $(\overline\mu^{\varepsilon}_t)_{t\in[0,T]}$, together with the Lagrange multiplier $\overline\lambda^{\varepsilon}$ and the optimal control policy $(\nabla\overline\varphi^{\varepsilon}_t)_{t\in [0,T]}$ form the unique solution of a coupled forward-backward PDE system, whose precise form is given at \eqref{eq:coupled_FP_HJB} below.
    As it is customary in mean-field control, the forward equation is a Fokker-Planck equation and the backward equation is a Hamilton-Jacobi-Bellman equation (henceforth, HJB). 
    The non-standard nature of the equations we study here lies in the presence of the Lagrange multiplier $\overline\lambda^{\varepsilon}$, accounting for the distributional constraints in the underlying control problem. 
    \item The restriction of $\overline\lambda^{\varepsilon}$ to the open interval $(0,T)$ is an absolutely continuous measure with globally bounded density. Atoms may exist at the endpoints of the interval.
\end{itemize}

\item \underline{\emph{Stability:}}  strengthening the regularity assumptions of Theorem \ref{thm:density} and adding a convexity assumption on $\Psi$ we show at Theorem \ref{thm:stability}: 
\begin{itemize}
\item \emph{Entropic stability:} the relative entropy $H( \overline\mu^{0}_{[0,T]} \vert \overline\mu^{\varepsilon}_{[0,T]} )$ is $O(\varepsilon)$. 
As a corollary of Pinsker's inequality, the total variation distance is of order $O(\varepsilon^{1/2})$. Furthermore, as a consequence of our assumptions, we prove that the $1$-Wasserstein distance is also $O(\varepsilon^{1/2})$.
\item \emph{Multiplier stability:} we show that the  total variation distance between the non-negative measures $\overline\lambda^{\varepsilon}$ and $\overline\lambda^0$ is $O(\varepsilon^{1/4})$.
\item \emph{Control stability:} we establish that, uniformly in $t\in[0,T]$, the $L^{\infty}$-difference between the optimal Markov policies $\nabla\overline\varphi^{\varepsilon}_t(\cdot)$ and $\nabla\overline\varphi^0_t(\cdot)$ is $O(\varepsilon^{1/4})$.
\end{itemize}
\end{itemize}
Both the Schr\"odinger problem and \eqref{eq:ent_proj_intro} aim at computing some kind of entropic projection and originate from the Gibbs conditioning principle. 
However, they fundamentally differ in the type of available observations. 
Whereas the former problem provides the full configuration of the particle system but only at initial and final times, \eqref{eq:ent_proj_intro} only has very partial information on the particle configuration, namely  the sign of the empirical average of $\psi$, but at all times $t\in[0,T]$. 
Thus, passing from one problem to the other, there is a trade-off between how rich is the observation we have 
and the set of times at which observations are made. 
This difference has deep implications in terms of how to approach the stability problem. For example, most stability results for the Schrödinger problem (see Section \ref{subsec:litt} below for references) rely on a static equivalent formulation in terms of an optimization problem over couplings on $\mathbb{R}^d \times \mathbb{R}^d$. 
Because of the dynamic nature of the constraints we deal with in \eqref{eq:ent_proj_intro}, such reduction to a static problem is impossible and the proof approach has to be different.

The backbone of our proof strategy consists of two main ingredients. 
The first one is a careful analysis of the regularity properties of the HJB equations that arise in connection with a relaxed version of \eqref{eq:ent_proj_intro}. 
This relaxation employs a Lagrange multiplier to get rid of the distributional constraints. 
The analysis of these PDEs is rather delicate due the low regularity of the involved cost functionals, which are typically only $L^1$ in time. 
Moreover, in order to obtain quantitative stability estimates, we have to push the analysis of these HJB equations up to the third derivative in Section \ref{subsec:smoothlambda}, thus obtaining novel bounds (up to our knowledge) that are of independent interest. 
We combine these regularity bounds with new estimates (up to our knowledge) for the Hessian of the log-density of diffusion processes.
These estimates are also of independent interest; they are proved in Appendix \ref{app:timerev} using time-reversal in the spirit of \cite{haussmann1986time}, extending results from \cite{fontbona2016trajectorial}.

The second key-ingredient in Section \ref{subsec:smoothlambda} is a procedure to construct good competitors for the regularised stochastic control problems, by slightly pushing optimal processes along the gradient flow generated by the observable $\psi$ and reinterpreting the result as an admissible control. 
This construction is crucial for showing that the optimal multiplier has a bounded density. 

We conclude this introduction by highlighting that the results of this paper open the door for obtaining quantitative versions of the Gibbs conditioning principle in which one computes the rate of convergence of the conditional distribution of the $N$-particle system towards the law of the optimal solution in \eqref{eq:ent_proj_intro}.

\subsection{Literature review} \label{subsec:litt}

Some reference textbooks about the Gibbs conditioning principle are \cite{lanford1973entropy,ruelle1965correlation,dembo2009large,dupuis2011weak,ellis2006entropy}.
Some prominent contributions are \cite{borel1906principes,diaconis1987dozen,stroock1991microcanonical}, which are summarised and deepened in the aforementioned textbooks. 
The related notion of entropic projection was introduced and studied in \cite{csiszar1975divergence,csiszar1984sanov}.
For quantitative versions of the Gibbs principle in convex settings, we refer to \cite{dembo1996refinements,dembo1998re,cattiaux2007deviations}.
All theses settings study the case of a finite number of constraints.
Existence of a Lagrange multiplier for abstract equality constraints is proved in \cite{leonard2000minimizers}.
The case of an infinite number of inequality constraints is treated in \cite{ConstrainedSchrodinger}.

There has been an abundant literature on Schr\"odinger bridges, some seminal works being \cite{csiszar1975divergence,cattiaux1995large,cattiaux1996minimization}.
We also refer to the survey article \cite{leonard2013survey}, the lecture notes \cite{nutz2021introduction} and references therein. 
The analogous problem when replacing the relative entropy by the rate function of an interacting particle system (case $W \neq 0$ of the introduction) is known as the mean-field Schr\"odinger problem \cite{backhoff2020mean}.
Up to our knowledge, the first result on Schrödinger bridges with distributional constraints is \cite{ConstrainedSchrodinger}.

Recently, important progresses on the stability problem for entropic projections have been achieved \cite{eckstein2022quantitative,nutz2023stability,chiarini2023gradient,divol2024tight}, mostly if not exclusively driven by the surge of interest around the Schr\"odinger problem and its applications in machine learning. 
However, the proofs in our setting are quite different as explained in the introduction.

Stochastic control problems under distributional constraints have received some attention, mostly motivated by mathematical finance, see e.g. \cite{follmer1999quantile,guo2022calibration} and references therein.
For problems with expectation constraints, we refer to \cite{chow2020dynamic,guo2022calibration,pfeiffer2021duality} and related works.
{ In what concerns the regularity results,} 
the closest works to ours are { the recent works} \cite{daudin2020optimal,daudin2023optimal}, which study the problem of controlling the drift of a Brownian motion under convex distributional constraints {, and sparked renewed interest around constrained stochastic control problems.}  
Their approach is purely PDE-oriented:
existence { and regularity are obtained for a Lagrange multiplier using a penalisation method, 
following some constructions put forward in the deterministic setting in \cite{cannarsa2018bf}.}
The main challenge of our framework compared to these works is the addition of a non-constant diffusion matrix and the McKean-Vlasov setting. 
In particular, our results rely on novel estimates for HJB equations, whose starting points are our recent works \cite{chaintron2023existence} and \cite{ConstrainedSchrodinger}. 
The related (non-quantitative) mean-field limit is proved in \cite{daudin2023meanCV}, adapting the method of \cite{lacker2017limit} to the constrained setting. { When it comes to the stability results we obtained, they appear to be the first of this kind that apply to dynamically constrained McKean-Vlasov control problems, at least to the best of our knowledge and understanding. }

\subsection{Frequently used notations} \label{subsec:Notations}

\begin{itemize}
    \item $\ps (E)$ denotes the set of probability measures over a probability space $E$.
    \item $\mathcal{L}(X)$ denotes the law in $\ps(E)$ of a $E$-valued random variable $X$.
    \item $\delta_x$ denotes the Dirac measure at some point $x$ in $E$.  
    \item $H(\mu \vert \nu)$ is the relative entropy of $\mu, \nu \in \ps (E)$, defined by $H(\mu \vert \nu) := \int_{E} \log \frac{\d \mu}{\d \nu} \d \mu$ if $\mu \ll \nu$, and $H(\mu \vert \nu) := +\infty$ otherwise. 
    \item $\ps_1(E)$ denotes the set of measures $\mu \in \ps_1 ( E)$ such that $\int_{E} d (x,x_0) \mu (\d x) < +\infty$ for some distance $d$ on $E$ and $x_0 \in E$.
    \item $W_1$ denotes the Wasserstein distance on $\ps_1(E)$, defined by
    \[ W_1(\mu,\mu') := \inf_{X \sim \mu, \, Y \sim \mu'} \E [ d ( X , Y ) ].  \]
    \item $\lVert \cdot \rVert_{\mathrm{TV}}$ denotes the total variation distance on $\ps(E)$, defined by
    \[ \lVert \mu - \mu' \rVert_{\mathrm{TV}} := \inf_{X \sim \mu, \, Y \sim \mu'} \P ( X \neq Y ).  \]
    \item $f \ast \mu$ is the convolution of a measurable $f : E \rightarrow E'$ and $\mu \in \ps (E)$, defined by $f \ast \mu ( x ) := \int_{E} f(x-y) \mu (\d y)$.
    \item $\tfrac{\delta F}{\delta\mu}(\mu) : x \mapsto \tfrac{\delta F}{\delta\mu}(\mu,x)$ denotes the linear functional derivative at $\mu$ of a function $F : \ps(E) \rightarrow \R$, see Definition \ref{def:PAP2DiFF} below.  
    The convention is adopted that $\int_{E} \tfrac{\delta F}{\delta\mu}(\mu) \d \mu=0$. 
    \item $T > 0$ is a given real number, and $d \geq 1$ is an integer.
    \item $x_{[0,T]} \in C([0,T],\R^d)$ denotes a continuous function $x_{[0,T]} : [0,T] \rightarrow \R^d$.
    \item $\mu_{[0,T]}$ denotes a path measure in $\ps( C([0,T], \R^d)$. The related marginal measure at time $t$ will be denoted by $\mu_t$.
    \item ${\sf{X}}_t$ denotes the coordinate map $x_{[0,T]} \mapsto x_t$. It can be seen as a random variable on the canonical space $\Omega = C([0,T],\R^d)$.
    \item  $\Sigma = \big( \Omega,(\mathcal{F}_t)_{0\leq t\leq T},\P,(B_t)_{0\leq t\leq T} \big)$ denotes a reference probability system in the terminology of \cite[Chapter 4]{fleming2006controlled}: $(\Omega,\mathcal{F}_T,\P)$ is a probability space, $(\mathcal{F}_t)_{0\leq t\leq T}$ is a filtration satisfying the usual conditions, and $(B_t)_{0\leq t\leq T}$ is a $(\mathcal{F}_t)_{0\leq t\leq T}$-Brownian motion. 
    \item $\M_+([0,T])$ denotes the convex cone of positive Radon measures over $[0,T]$. In this setting, a Radon measure is a signed finite measure that is both inner and outer regular as defined in \cite[Definition 2.15]{rudin1970real}.
    \item $\cdot^\top$ and $\mathrm{Tr}[\cdot]$ respectively denote the transpose and the trace of matrices.
    \item $\vert a \vert$ denotes the Frobenius norm $\sqrt{\mathrm{Tr}[a a^\top]}$ of a matrix $a$.
    \item $\nabla \cdot a$ for a matrix field $a  = (a^{i,j})_{1 \leq i,j \leq d}$ is the vector field whose $i$ entry is the divergence of the vector field $( a^{i,j} )_{1 \leq j \leq d}$. 
    We similarly define differential operators on matrices. 
    \item When using coordinates, we will always use the summing convention of repeated indices.
\end{itemize}

\begin{definition}[Linear functional derivative] \label{def:PAP2DiFF}
Let $\mathcal{C}$ be a convex subset of $\ps(E)$, and let $\mu \in \mathcal{C}$.
A map $F : \mathcal{C} \rightarrow \R$ is differentiable at $\mu$ if there exists a measurable map 
\[
\frac{\delta F}{\delta\mu}(\mu):
\begin{cases}
E \rightarrow \R, \\
x \mapsto \frac{\delta F}{\delta\mu}(\mu,x),
\end{cases}
\]
such that for every $\mu'$ in $\mathcal{C}$, $\tfrac{\delta F}{\delta\mu}(\mu)$ is $\mu'$-integrable and satisfies
\begin{equation*} 
\varepsilon^{-1} \bigl[ F( (1-\varepsilon) \mu + \varepsilon \mu' ) - F ( \mu ) \bigr] \xrightarrow[\varepsilon \rightarrow 0^+]{} \int_{E} \frac{\delta F}{\delta \mu}(\mu) \d [ \mu' - \mu ]. 
\end{equation*} 
The map $\tfrac{\delta F}{\delta \mu}(\mu)$ being defined up to an additive constant, we adopt the usual convention that $\int_{E} \tfrac{\delta F}{\delta \mu}(\mu) \d\mu = 0$.
This map is called the \emph{linear functional derivative of $F$ at $\mu$} (w.r.t. the set of directions $\mathcal{C}$).
We notice that this definition does not depend on the behaviour of $F$ outside of an arbitrary small neighbourhood of $\mu$.
\end{definition}

By direct integration, the definition implies
\begin{equation*} 
\forall \mu, \mu' \in \mathcal{C}, \quad F(\mu)-F(\mu') = \int_0^1 \int_{E} \frac{\delta F}{\delta \mu}((1-r)\mu' + r \mu) \d [ \mu - \mu' ] \d r, 
\end{equation*}
provided that the integral on the r.h.s. is well-defined.

\begin{example}[Linear case]
In the particular case $F(\mu) = \int_{E} f \d \mu$ for some measurable $f: E \rightarrow \R$ that is $\mu$-integrable for every $\mu$ in $\mathcal{C}$, we have $\frac{\delta F}{\delta \mu}(\mu,x) = f(x) - \int_{E} f \d\mu$.
\end{example}

\section{Statement of the main results} \label{sec:PAP2results}

For $\nu_0 \in \ps_1( \R^d )$ and $\mu_{[0,T]} \in \ps_1 (C([0,T],\R^d))$, let  $\Gamma(\mu_{[0,T]})$ denote the path-law of the pathwise unique strong solution to
\[ \d X_t = b_t (X_t,\mu_t) \d t + \sigma_t (X_t) \d B_t, \quad X_0 \sim \nu_0, \]
where $(B_t)_{0 \leq t \leq T}$ is a Brownian motion and the coefficients are continuous functions $b : [0,T] \times \R^d \times \ps_1 (\R^d) \rightarrow \R^d$, $\sigma : [0,T] \times \R^d \rightarrow \R^{d \times d}$ satisfying \ref{ass:coefReg1} below.
The generalisation of \eqref{eq:ent_proj_intro} we are interested in is
\begin{equation} \label{eq:mfMinProb}
\overline{H}_\Psi := \inf_{\substack{\mu_{[0,T]} \in \ps_1( C ( [0,T], \R^d ) ) \\ \forall t \in [0,T], \; \Psi ( \mu_t ) \leq 0}} H(\mu_{[0,T]} \vert \Gamma(\mu_{[0,T]} )), 
\end{equation}
where $\Psi : \ps_1 (\R^d) \rightarrow \R$ is lower semi-continuous. 
The map $\mu_{[0,T]} \mapsto H(\mu_{[0,T]} \vert \Gamma(\mu_{[0,T]} ))$ is a large deviation rate function, which is studied in e.g. \cite{fischer2014form,backhoff2020mean}.
Under \ref{ass:ini1}-\ref{ass:coefReg1} below, this map has compact level sets in $\ps_1 ( C([0,T],\R^d)$ from \cite[Remark 5.2]{fischer2014form} or \cite[Corollary B.4]{chaintronLDP}
(it is a good rate function in the large deviation terminology).
Since $\Psi$ is lower semi-continuous, existence of a minimiser always holds for \eqref{eq:mfMinProb} if $\overline{H}_\Psi$ is finite.
We now detail the regularity assumptions that are needed for writing optimality conditions for \eqref{eq:mfMinProb}.

\begin{assumptionV}[Initial condition] \label{ass:ini1} 
For every $\alpha >0$, $x \mapsto e^{\alpha \vert x \vert}$ is $\nu_0$-integrable.
\end{assumptionV}

\begin{assumptionV}[Coefficients] \label{ass:coefReg1}
There exists $C > 0$ such that:
\begin{enumerate}[label=(\roman*),ref=(\roman*)]
\item\label{item:coefLip} $\forall (t,x,y,\mu,\mu') \in [0,T] \times \R^d \times \R^d \times \ps_1(\R^d) \times \ps_1(\R^d),$ \\
$\phantom{\forall (t,x,y,\mu,\mu') \in .}\lvert b_t (x,\mu) - b_t (y,\mu') \rvert + \lvert \sigma_t (x) - \sigma_t (y) \rvert \leq C [ \vert x - y \vert + W_1 (\mu,\mu') ]$.
\item\label{item:sigmabound} $\forall (t,x) \in [0,T] \times \R^d, \; \lvert \sigma_t (x) \rvert \leq C$, and $t \mapsto \sigma_t (x)$ is locally Hölder-continuous. 
\item\label{item:sigmaellip} $\forall (t,x,\xi) \in [0,T] \times \R^d \times \R^d, \; \xi^\top \sigma_t \sigma_t ^\top (x) \xi \geq C^{-1} \vert \xi \vert^2$.
\end{enumerate}
\end{assumptionV}

\begin{assumptionV}[Linear derivatives] \label{ass:linDiff}
The functions $\mu \mapsto \Psi(\mu)$ and $\mu \mapsto b_t(x,\mu)$ have a linear functional derivative in the sense of Definition \ref{def:PAP2DiFF}.
Moreover, $(x,\mu) \mapsto \tfrac{\delta\Psi}{\delta\mu}(\mu,x)$ and $(t,x,\mu,y) \mapsto \tfrac{\delta b_t}{\delta\mu}(x,\mu,y)$ are jointly continuous, and for every compact set $K \subset \ps_1 (\R^d)$, 
\[ \forall (t,\mu,x,y,z) \in [0,T] \times K \times (\R^d)^3, \; \big\lvert \tfrac{\delta b_t}{\delta\mu} (x,\mu,y) - \tfrac{\delta b_t}{\delta\mu} (x,\mu,z) \big\vert + \big\lvert \tfrac{\delta\Psi}{\delta\mu} (\mu,y) - \tfrac{\delta\Psi}{\delta\mu} (\mu,z) \big\vert \leq C_K \vert y - z \vert, \]
for $C_K > 0$ that only depends on $K$.
Finally, $x \mapsto \nabla \tfrac{\delta\Psi}{\delta\mu}(\mu,x)$ is $C_K$-Lipschitz for any $\mu \in K$.
\end{assumptionV}

In the statement of our results, we shall require the following assumption, called \emph{constraint qualification}, to be satisfied at a given $\overline\mu_{[0,T]} \in \ps_1(C([0,T],\R^d))$ such that $\Psi (\overline\mu_t) \leq 0$ for every $t \in [0,T]$.

\begin{assumptionV}[Qualification] \label{ass:consquali}
There exists $\tilde\varepsilon >0$, $\tilde{\mu}_{[0,T]} \in \ps_1(C([0,T],\R^d))$ such that $H(\tilde{\mu}_{[0,T]} \vert \Gamma(\overline\mu_{[0,T]})) < +\infty$ and  
\[ \forall t \in [0,T], \quad \Psi (\overline\mu_t) + \tilde\varepsilon \int_{\R^d} \frac{\delta \Psi}{\delta\mu} (\overline\mu_t) \d \tilde\mu_t < 0. \]
\end{assumptionV}

An analytical equivalent of this assumption is given by Remark \ref{rem:Suffqualif} below. 

\begin{rem} \label{rem:Suffconv}
When $\Psi$ is convex, a condition equivalent to \ref{ass:consquali} is the existence of $\tilde{\mu}_{[0,T]} \in \ps_1(C([0,T],\R^d))$ such that $H(\tilde{\mu}_{[0,T]} \vert \Gamma(\overline\mu_{[0,T]})) < +\infty$ and  
\[ \forall t \in [0,T], \quad \Psi (\tilde\mu_t) < 0, \]
see \cite[Lemma 2.17]{ConstrainedSchrodinger}.
Another sufficient condition is given by
\cite[Remark 2.15]{ConstrainedSchrodinger}.
\end{rem}

We now provide a control interpretation of \eqref{eq:mfMinProb} for which we write optimality conditions.

\begin{theorem}[Structure of optimisers] \label{thm:opti}
Under \ref{ass:ini1}-\ref{ass:coefReg1}-\ref{ass:linDiff}, let us assume that $\overline\mu_{[0,T]}$ is an optimal measure for \eqref{eq:mfMinProb} that satisfies \ref{ass:consquali}. 
\begin{enumerate}[label=(\roman*),ref=(\roman*)]
\item\label{item:Control} \underline{McKean-Vlasov control}: consider the control problem
with law constraints
\[ \overline{V}_\Psi := \inf_{\substack{(X^\alpha_t)_{0 \leq t \leq T}, \alpha \\ \forall t \in[0,T], \, \Psi( \mathcal{L}(X^\alpha_t) ) \leq 0}} H( \mathcal{L}(X^\alpha_0) \vert \nu_0 ) + \E \, \int_0^T \frac{1}{2} \vert \alpha_t \vert^2 \d t, \]
on a filtered probability space $(\Omega,\F_T,(\F_t)_{0 \leq t \leq T},\P)$ satisfying the usual conditions, where $\alpha$ is progressively measurable, $(X^\alpha_t)_{0 \leq t \leq T}$ is adapted, and $\P$-a.s. 
\[ \d X^{\alpha}_t = b_t (  X^{\alpha}_t, \mathcal{L}( X^{\alpha}_t )  ) \d t+ \sigma_t (X^{\alpha}_t) \alpha_t \d t+ \sigma_t (X^{\alpha}_t) \d B_t, \quad 0\leq t\leq T, \]
for a given $(\F_t)_{0 \leq t \leq T}$-Brownian motion $(B_t)_{0 \leq t \leq T}$ in $\R^d$.
We have the correspondence:  
\begin{equation*} 
\overline{H}_\Psi = \overline{V}_\Psi, 
\end{equation*} 
and $\overline{\mu}_{[0,T]}$ is the path-law of an optimally controlled process. 
\item\label{item:OptCond} \underline{Optimality conditions}: there exists $\overline{\lambda} \in \M_+([0,T])$ and $\overline{\varphi}:[0,T] \times \R^d \rightarrow \R$ such that $\overline{\mu}_0(dx)= \overline{Z}^{-1} e^{-\overline\varphi_0(x)}\nu_0(\d x)$ and
\begin{equation}\label{eq:coupled_FP_HJB}
\begin{cases}
\partial_t \overline{\mu}_t - \nabla\cdot [ -\overline\mu_t b_t (\cdot,\overline{\mu}_t) +  \overline{\mu}_t \sigma_t \sigma_t^\top \nabla \overline\varphi_t + \tfrac{1}{2} \nabla \cdot [ \overline\mu_t \sigma_t \sigma_t^\top ]  ] =0,  \\[5pt]
- \overline\varphi_t + \int_t^T b_s (\cdot,\overline{\mu}_s) \cdot \nabla \overline\varphi_s - \tfrac{1}{2} \big\lvert \sigma^\top_s \nabla\overline\varphi_s \big\rvert^2 + \tfrac{1}{2} \mathrm{Tr} [ \sigma_s \sigma_s^\top \nabla^2 \overline\varphi_s ] + \overline{c}_s \d s \\[5pt]
\qquad\qquad\qquad\qquad\qquad\qquad\qquad\qquad + \int_{[t,T]} \frac{\delta\Psi}{\delta\mu} (\overline{\mu}_s) \overline{\lambda}( \d s) = 0,
\end{cases}
\end{equation}
where
\[ \overline{c}_s(x) := \int_{\R^d} \frac{\delta b}{\delta\mu}(y,\overline{\mu}_s,x) \cdot \nabla \overline\varphi_s (y) \overline{\mu}_s(\d y). \]
Moreover, $\overline{V}_\Psi$ is realised by the optimal control in feed-back form $\overline{\alpha}_t = - \sigma_t \sigma_t^\top \nabla \overline\varphi_t ( X^{\overline\alpha}_t )$.
In the above, solutions to the Fokker-Planck equation have to be understood in the weak sense, and solutions to the Hamilton-Jacobi-Bellman (HJB) equation have to be understood in the sense of Definition \ref{def:mildForm}. Furthermore, the complementary slackness condition holds, 
\begin{equation} \label{eq:CompSlackmf}
\Psi (\overline{\mu}_t) = 0, \quad \text{for } \overline{\lambda}\text{-a.e. } t\in [0,T],
\end{equation}
and 
we have the pathwise representation
\begin{equation*} 
\frac{\d \overline{\mu}_{[0,T]}}{\d \Gamma(\overline{\mu}_{[0,T]})}(x_{[0,T]})=\overline{Z}^{-1}\exp \bigg[ -\int_0^T  \overline{c}_t (x_t) \d t - \int_{0}^T \frac{\delta\Psi}{\delta\mu}(\overline\mu_t, x_t) \overline\lambda (\d t) \bigg]. 
\end{equation*} 
\end{enumerate}
\end{theorem}

Let us further study the Lagrange multiplier $\overline{\lambda}$.
To do so, we need to strengthen our regularity assumptions.

\begin{assumptionV}[Strengthening of \ref{ass:ini1}] 
\label{ass:ini2}
Assumption \ref{ass:ini1} holds and $\nu_0$ has a positive density w.r.t. the Lebesgue measure, still denoted by $\nu_0$.
Moreover, $C, \varepsilon_0 > 0$ exist such that
\begin{enumerate}[label=(\roman*),ref=(\roman*)]
\item\label{item:inih1} $\log \nu_0$ is (a.e.) differentiable and $\int_{\R^d} \big[ ( \log \nu_0 )^{1+\varepsilon_0}  + \vert \nabla \log \nu_0 \vert^{2+\varepsilon_0} \, \big] \d \nu_0 \leq C$.
\item\label{item:inih2} $\nabla \log \nu_0$ is differentiable and $\int_{\R^d} \big[ ( \log \nu_0 )^{1+\varepsilon_0} + \vert \nabla \log \nu_0 \vert^{4 + \varepsilon_0} + \vert \nabla^2 \log \nu_0 \vert^{2 +\varepsilon_0} \big] \d \nu_0 \leq C$.
\item\label{item:iniT1} We further have $\int_{\R^d} e^{\varepsilon_0 \vert x \vert^2} \nu_0 (\d x) \leq C$.
\end{enumerate}
\end{assumptionV}

\begin{assumptionV}[Strengthening of \ref{ass:coefReg1}-\ref{ass:linDiff}] 
\label{ass:coefReg2}
Assumptions \ref{ass:coefReg1}-\ref{ass:linDiff} hold, and
$\mu \mapsto \tfrac{\delta\Psi}{\delta\mu}(\mu,x)$ has a linear derivative $y \mapsto \tfrac{\delta^2\Psi}{\delta\mu^2}(\mu,x,y)$. 
Moreover, the functions 
\[ t \mapsto b_t (x,\mu), \; t \mapsto \sigma_t (x), \; x \mapsto \nabla \sigma_t(x), \]
\[ (\mu,x,y) \mapsto \tfrac{\delta b_t}{\delta\mu} (x,\mu,y), \; (\mu,x) \mapsto \tfrac{\delta\Psi}{\delta\mu}(\mu,x), \; x \mapsto \nabla \tfrac{\delta\Psi}{\delta\mu}(\mu,x), \; y \mapsto \tfrac{\delta^2\Psi}{\delta\mu^2}(\mu,x,y), \] 
are well-defined and Lipschitz-continuous uniformly in $(t,x,\mu)$.
\end{assumptionV}

\begin{assumptionV}[Strengthening of \ref{ass:coefReg2}] 
\label{ass:coefReg3}
Assumption \ref{ass:coefReg2} holds and there exists $C > 0$ such that for every $(t,x,\mu) \in [0,T] \times \R^d \times \ps_1 (\R^d)$,
\begin{enumerate}[label=(\roman*),ref=(\roman*)]
\item\label{item:bC2} $x \mapsto \nabla_x b_t (x,\mu)$, $x \mapsto \nabla^2 \sigma_t (x)$ and $y \mapsto \nabla_y \tfrac{\delta b_t}{\delta \mu} (x,\mu,y)$ are well-defined and $C$-Lipschitz.
\item\label{item:PsiC3} $x \mapsto \tfrac{\delta\Psi}{\delta\mu}(x,\mu)$ is $C^2$ with $C$-Lipschiz derivatives which are continuous in $(x,\mu)$. \item\label{item:PsiC4} $x \mapsto \tfrac{\delta\Psi}{\delta\mu}(\mu,x)$ is $C^4$ and $y \mapsto \tfrac{\delta^2\Psi}{\delta\mu^2}(\mu,x,y)$ is $C^2$, both with derivatives bounded by $C$ and jointly Lipschitz-continuous in $(\mu,x,y)$.
\end{enumerate}
\end{assumptionV}

We now state sufficient conditions for $\overline{\lambda}$ having a density.

\begin{theorem}[Density for the multiplier] \label{thm:density}
Under \ref{ass:consquali} at $\overline\mu_{[0,T]}$, there exists $(\overline\lambda,\overline\varphi)$ satisfying Theorem \ref{thm:opti}-$(ii)$ such that $x \mapsto \nabla \overline\varphi_t(x)$ is Lipschitz-continuous and
\[ \overline{\lambda}(\d t) = \overline{\lambda}_0 \delta_0 ( 
\d t) + \overline{\lambda}_t \d t + \overline{\lambda}_T \delta_T(\d t), \]
for $t \mapsto \overline{\lambda}_t$ in $L^\infty(0,T)$, if we assume that either \ref{ass:ini2}-\ref{item:inih1} and \ref{ass:coefReg3}-\ref{item:PsiC3} hold, or that \ref{ass:ini2}-\ref{item:inih2} and \ref{ass:coefReg3}-\ref{item:bC2} hold. 
In the second case, $x \mapsto \nabla^2 \overline{\varphi}_t(x)$ is Lipschitz-continuous and the densities of the time-marginals of $\overline{\mu}_{[0,T]}$ satisfy
\begin{equation} \label{eq:gamm3}
\bigg[ \sup_{t \in [0,T]} \int_{\R^d} \big[ \log \overline\mu_t + \vert \nabla \log \overline\mu_t \vert^4 + \vert \nabla^2 \log \overline\mu_t \vert^2 \big] \d \overline\mu_t \bigg] + \int_0^T \int_{\R^d} \lvert \nabla^3 \log \overline\mu_t \rvert^2 \d \overline\mu_t \d t < +\infty. 
\end{equation} 
\end{theorem}

{
\begin{rem}
As highlighted in the introduction, constrained mean field control problems have recently been studied by Daudin in a series of articles. 
More specifically, optimality conditions of the form \eqref{eq:coupled_FP_HJB} and the existence of a bounded Radon-Nikodym density for Lagrange multipliers are established in \cite[Theorem 2.2]{daudin2023optimal}. These results are very similar in spirit to Theorem \ref{thm:opti} and Theorem \ref{thm:density} and have partially inspired our work. 
One should note that the control problems considered here are slightly different: in \cite{daudin2023optimal}, the objective function includes an additional mean field term in both the running and terminal costs, {the reference dynamics has zero drift and a constant diffusion matrix,} and the initial distribution is fixed. 
Here, the initial distribution is a variable over which we optimize, and the controlled dynamics corresponds to a  stochastic differential equation  of McKean-Vlasov-type with non constant diffusions coefficient.
Regarding the constraint function $\Psi$, we are able to work under weaker regularity assumptions and do not require convexity.  
We also rephrase the transversality conditions from Daudin’s work in terms of a constraint qualification condition, see Remark \ref{rem:Suffqualif}.
{Our proof approaches are rather different too.
We study regularity properties by pushing the optimal process along the gradient flow of the constraint to build suitable competitors, whereas \cite{daudin2023optimal} 
approximates the constrained problem by relaxing the constraint and penalising the flows that violate it. Then, he shows that if the penalisation strength is sufficiently large, the penalised problem and the constrained problem are equivalent.
}
\end{rem}
}

The results on the density of $\overline\lambda$ can be seen as a regularity trade-off between the constraint and the coefficients.
They rely on the new technical estimate \eqref{eq:gamm3} for diffusions, which is proved in Proposition \ref{pro:Gamma_3}.
The proof of Theorem \ref{thm:density} and Corollary \ref{cor:timeReg1} further show that $t \in (0,T] \mapsto \nabla \varphi_t (x)$ is continuous and $t \in [0,T] \mapsto \Psi ( \overline\mu_t )$ is $C^1$ under \ref{ass:coefReg3}-\ref{item:PsiC3}, whereas $t \in (0,T] \mapsto \nabla^2 \varphi_t (x)$ is continuous under \ref{ass:coefReg3}-\ref{item:bC2}.
Depending on where the regularity is assumed (constraints or coefficients), we thus get more smoothness on either the constrained curve or the optimal control.   
Requiring more derivatives on coefficients would still improve the regularity in $(t,x)$.

\begin{rem}[Equivalent qualification condition]\label{rem:Suffqualif}
Theorem \ref{thm:opti} implies that $\overline\mu_{[0,T]}$ and $\Gamma ( \overline\mu_{[0,T]} )$ are equivalent.
A consequence of \ref{ass:consquali} is then
\[ \forall t \in [0,T], \quad \Psi ( \overline\mu_t ) = 0 \Rightarrow \int_{\R^d} \bigg\lvert \nabla \frac{\delta \Psi}{\delta\mu} ( \overline\mu_t ) \bigg\rvert^2 \d \overline\mu_t \neq 0. \]
In the setting of Theorem \ref{thm:density}, this condition is equivalent to \ref{ass:consquali} from Corollary \ref{cor:Qualif}.      
\end{rem} 

We now prove quantitative stability with respect to the constraint $\Psi$.
To guarantee that uniqueness holds for \eqref{eq:mfMinProb}, we restrict ourselves to $b_t(x,\mu) = b_t(x)$ and $\Psi$ convex.
This setting still includes the entropic projection \eqref{eq:ent_proj_intro} of the introduction.
Let $\nu_{[0,T]}$ denote the path-law of the solution to
\[ d X_t = b_t(X_t) \d t + \sigma_t ( X_t) \d B_t, \quad X_0 \sim \nu_0, \]
where $(B_t)_{0 \leq t \leq T}$ is a Brownian motion. 
To alleviate notations, optimal measures will be denoted by $\mu^\varepsilon_{[0,T]}$ in the following statement, instead of $\overline\mu_{[0,T]}$.
The same goes for $\varphi^\varepsilon$ and $\lambda^\varepsilon$.

\begin{theorem}[Quantitative stability] \label{thm:stability}
We assume \ref{ass:consquali}-\ref{ass:ini2}-\ref{ass:coefReg3} and that $\Psi$ is convex. 
For $\varepsilon \geq 0$, let $\mu^\varepsilon_{[0,T]}$ denote the unique minimiser for
\[ \inf_{\substack{\mu_{[0,T]} \in \ps_1( C ( [0,T], \R^d ) ) \\ \forall t \in [0,T], \; \Psi ( \mu_t ) \leq \varepsilon}} H(\mu_{[0,T]} \vert \nu_{[0,T]} ). \]
Let $(\varphi^\varepsilon_t,\lambda^\varepsilon_t,\lambda^\varepsilon_0,\lambda^\varepsilon_T)$ be related to $\mu^\varepsilon_{[0,T]}$ by Theorems \ref{thm:opti}-\ref{thm:density}. 
We assume that the number of non-trivial intervals where $\Psi(\mu^0_t)$ is identically zero is finite.
Then, there exists $C > 0$ independent of $\varepsilon$ such that uniformly in $\varepsilon \in [0,1]$,
\begin{enumerate}[label=(\roman*),ref=(\roman*)]
\item \underline{Entropic stability:} $H( \mu^0_{[0,T]} \vert \mu^\varepsilon_{[0,T]} ) \leq C \varepsilon$.
\item\label{itempap2:stab}  \underline{Multiplier stability:} $\vert \lambda^\varepsilon_0 - \lambda^0_0 \vert + \lVert \lambda^\varepsilon - \lambda^0 \rVert_{L^1 (0,T)} + \vert \lambda^\varepsilon_T - \lambda^0_T \vert \leq C {\varepsilon^{1/4}}$.
\item \underline{Control stability:} $\sup_{(t,x) \in [0,T] \times \R^d} \lvert \nabla {\varphi}^\varepsilon_t (x) - \nabla {\varphi}^0_t (x) \rvert \leq C \varepsilon^{1/4}$.
\end{enumerate}
\end{theorem}

We notice that the l.h.s. of \ref{itempap2:stab} corresponds to the total variation distance between the multipliers, seen as measures over $[0,T]$. 
The Gaussian bound \ref{ass:ini2}-\ref{item:iniT1} is only used to simplify the convergence rate (otherwise, additional $\log \varepsilon$ factors would have been needed).
Similarly, the assumption on the finite number of intervals can be removed if we deteriorate the rate, see Remark \ref{rem:InfiniteNumber} below.

\begin{rem}[$L^\infty$-stability]
From the proof of Proposition \ref{pro:L1estiamte}, the $L^1$-estimate on $\lambda^\varepsilon$ is actually a $L^\infty$-one on intervals where $\Psi ( \mu^0_t) = 0$ and $\Psi ( \mu^\varepsilon_t) = \varepsilon$ simultaneously.
We cannot expect a global $L^\infty$-estimate, since moving the constraint can produce strong discontinuities for $\lambda^\varepsilon_t$ at the boundary of $\{ t \in [0,T] \, , \, \Psi (\mu^\varepsilon_t) = \varepsilon \}$.
However, under stronger assumptions on coefficients, Lemma \ref{lem:timeReg2} can provide continuity for $\lambda^\varepsilon_t$ within intervals where $\Psi ( \mu^\varepsilon_t) = \varepsilon$ identically. 
\end{rem}

The proofs are organised as follows. 
Section \ref{sec:Bounds} proves bounds on the multiplier in a linearised setting with smooth coefficients.
However, we precisely keep track of regularity constants along the estimates.
The main results are then proved in Section \ref{sec:proofs}.
The proof of Theorem \ref{thm:opti} essentially relies on results from \cite{ConstrainedSchrodinger}. 
To cope with the delicate regularity assumptions of Theorem \ref{thm:density}, a careful smoothing procedure is detailed in Section \ref{subsec:density}. 
The non-linear problem \eqref{eq:mfMinProb} is first approximated by smooth non-linear problems. 
These smooth problems are then linearised to enter the scope of Section \ref{sec:Bounds}.
Theorem \ref{thm:density} is obtained by taking the limit in the bounds therein.
Starting from Theorems \ref{thm:opti}-\ref{thm:density}, Theorem \ref{thm:stability} is proved in an independent way in Section \ref{subsec:stability}.

\section{Regularity bounds on the multiplier} \label{sec:Bounds}

This section is concerned with the strictly convex minimisation problem
\begin{equation} \label{eq:SmoothProb}
\inf_{\substack{\mu_{[0,T]} \in \ps_1 (C([0,T],\R^d)) \\ \forall t \in [0,T], \; \int_{\R^d} \psi_t \d \mu_t \leq 0}} H( \mu_{[0,T]} \vert \nu_{[0,T]} ) + \int_0^T \int_{\R^d} c_t \d \mu_t \d t,  
\end{equation}
where $\nu_{[0,T]}$ is the path-law of the pathwise unique strong solution to
\[ \d X_t = b_t (X_t ) \d t + \sigma_t (X_t) \d B_t, \quad  X_0 \sim \nu_0, \] 
under the following assumption.

\begin{assumptionV}[Smooth setting] \label{ass:SmoothSett} $\phantom{a}$
\begin{enumerate}[label=(\roman*),ref=(\roman*)]
\item\label{it:SmoothEllip} $\nu_0$ satisfies \ref{ass:ini1}, and $\sigma$ satisfies \ref{ass:coefReg1}-\ref{item:sigmabound}-\ref{item:sigmaellip}.
\item\label{it:SmoothDx} $b_t$, $\sigma_t$, $c_t$, $\psi_t$ are $C^\infty$ function of $x$, with bounded derivatives of all orders independently of $t$.
Moreover, all  derivatives of $b_t$, $\sigma_t$, $\psi_t$ are locally Hölder-continuous in $(t,x)$. 
\item\label{it:SmoothDt} $\partial_t b_t$, $\partial_t \sigma_t$ are well-defined, $C^\infty$ and bounded independently of $t$, as well as of all their derivatives. 
\item\label{it:SmoothPsi} There exists a minimiser $\overline{\mu}_{[0,T]}$ for \eqref{eq:SmoothProb}, which is necessarily unique by strict convexity of the relative entropy. 
Moreover, $\psi$ satisfies
\begin{equation} \label{eq:asspartialpsi}
\forall t \in [0,T], \quad \int_{\R^d} \partial_t \psi_t \d \overline\mu_t = 0. 
\end{equation}
and constraint qualification holds at $\overline{\mu}_{[0,T]}$: there exists $\tilde\varepsilon >0, \tilde\mu_{[0,T]} \in \ps_1(C([0,T],\R^d))$ with $H(\tilde\mu_{[0,T]} \vert \nu_{[0,T]}) < +\infty$ such that 
\begin{equation} \label{eq:smoothQualif}
\forall t \in [0,T] \quad \int_{\R^d} \psi_t \d \overline\mu_t + \tilde\varepsilon \int_{\R^d} \psi_t \d ( \tilde\mu_t - \overline\mu_t) < 0. 
\end{equation}
\end{enumerate}
\end{assumptionV}

This setting satisfies the assumptions of Theorem \cite[Theorem 2.12]{ConstrainedSchrodinger} with $E = C([0,T],\R^d)$, $\phi(x_{[0,T]})=\sup_{t \in [0,T]} \vert x_t \vert$, $\mathcal{F}(\mu_{[0,T]}) = \int_0^T \int_{\R^d} c_t \d \mu_t \d t$, the inequality constraints $\Psi_t : \mu_{[0,T]} \mapsto \int_{\R^d} \psi_t \d \mu_t$ with $\mathcal{T} = [0,T]$, and no equality constraint.
Indeed, ${c}_t$ has linear growth uniformly in $t$, implying that $x_{[0,T]} \mapsto \int_0^T {c}_t (x_t)$ has linear growth, and thus 
\cite[Assumptions (A6)-(A8)]{ConstrainedSchrodinger} are satisfied.
Moreover, \cite[Assumptions (A7)]{ConstrainedSchrodinger} is satisfied using \ref{ass:SmoothSett} and \cite[Remark 2.14]{ConstrainedSchrodinger}.
Since \eqref{eq:smoothQualif} provides the qualification condition, we obtain a Lagrange multiplier $\overline{\lambda} \in \M_+([0,T])$ such that 
\begin{equation} \label{eq:LinDensity}
\frac{\d \overline\mu_{[0,T]}}{\d \nu_{[0,T]}} ( x_{[0,T]} ) = \overline{Z}^{-1}\exp \bigg[ -\int_0^T c_t (x_t) \d t - \int_{0}^T \psi_t (x_t) \overline\lambda (\d t) \bigg], 
\end{equation} 
with the complementary slackness condition
\begin{equation} \label{eq:PAP2Compsl}
\int_{\R^d} \psi_t \d \mu_t = 0, \quad \text{for } \overline{\lambda}\text{-a.e. } t \in [0,T].
\end{equation}
This section aims to show that the restriction of $\overline\lambda$ to $(0,T)$ has a bounded density w.r.t. the Lebesgue measure.
Up to changing $\nu_0 ( \d x)$ into $Z^{-1}_0 e^{-\overline{\lambda}(\{ 0 \}) \psi_0 (x)} \nu_0 ( \d x)$, we can assume that $\overline{\lambda}(\{ 0 \}) = 0$.
Indeed, under \ref{ass:coefReg2}, $\nu_0$ still satisfies \ref{ass:ini2} after this change.
In the following, we set $\overline\lambda_T := \overline{\lambda}(\{ T \})$ and we will use the convenient notation $a := \sigma \sigma^\top$.

\subsection{Smoothed multiplier and stochastic control} \label{subsec:smoothlambda}

Let $\Sigma = (\Omega, (\F_t)_{0 \leq t \leq T},\P,(B_t)_{0 \leq t \leq T})$ be a reference probability system as defined in Section \ref{subsec:Notations}. 
For $(t,x) \in [0,T] \times \R^d$, we consider the controlled dynamics 
\begin{equation} \label{eq:relaxcontrolled}
\begin{cases}
\d X^{t,x,\alpha}_s = b_s (X^{t,x,\alpha}_s) \d s + \sigma_s (X^{t,x,\alpha}_s) \alpha_s \d s + \sigma_s ( X^{t,x,\alpha}_s ) \d B_s, \quad t \leq s \leq T,\\
X^{t,x,\alpha}_t = x,
\end{cases}
\end{equation}
where $\alpha = ( \alpha_s )_{t \leq s \leq T}$ is any progressively measurable square-integrable process on $\Sigma$.
Such a process will be called an \emph{open-loop control}.
Following \cite[Section 3, page 9]{budhiraja2012large}, strong existence and pathwise uniqueness for \eqref{eq:relaxcontrolled} is given by the Girsanov transform if $\int_0^T \vert \alpha_t \rvert^2 \d t \leq M$ a.s. for some $M >0$, and still holds otherwise using a localisation argument.
In the following, we will often consider consider \emph{feed-back controls} $\alpha_s = \alpha ( s, X^\alpha_s )$, as it is customary in control theory.
In this case, regularity properties will be needed on the function $\alpha$ to prove strong existence and pathwise uniqueness for the SDE
\begin{equation} \label{eq:feedbackSDE}
\d Y^{t,x,\alpha}_s = b_s (Y^{t,x,\alpha}_s) \d s + \sigma_s (Y^{t,x,\alpha}_s) \alpha (s, Y^{t,x,\alpha}_s) \d s + \sigma_s ( Y^{t,x,\alpha}_s ) \d B_s, \quad t \leq s \leq T,
\end{equation}
with $Y^{t,x,\alpha}_t = x$. 
If well-posedness is granted for this SDE, we will still denote its solution by $X^{t,x,\alpha}$, with a slight abuse of notations, and we may write $\alpha_s ( X^{t,x,\alpha}_s)$ instead of $\alpha (s, X^{t,x,\alpha}_s)$ when it is non-ambiguous.
Furthermore, we will use the infinitesimal generator $L^\alpha_t$ defined on $C^2$ functions $\phi : \R^d \rightarrow \R$ by
\[ L^\alpha_t \phi := b_t \cdot \nabla \phi + \sigma_t \alpha_t \cdot \nabla \phi + \frac{1}{2} \mathrm{Tr}[a_t \nabla^2 \phi ]. \]
Given any $(\mu_0,\lambda) \in \ps_1 ( \R^d) \times \M_+([0,T])$, we define the value functions:
\[ V^\lambda_\Sigma ( t,x ) := \inf_{\alpha} \E \int_t^T \frac{1}{2} \lvert \alpha_s \rvert^2 + c_s (X^{t,x,\alpha}_s) \d s + \int_{[t,T]} \psi_s ( X^{t,x,\alpha}_s) \lambda ( \d s ), \]
\[ V^\lambda_\Sigma ( \mu_0 ) := \inf_{\substack{X_0,\alpha \\ X_0 \sim \mu_0}} \E \int_0^T \frac{1}{2} \lvert \alpha_s \rvert^2 + c_s (X^{0,X_0,\alpha}_s) \d s + \int_{[0,T]} \psi_s ( X^{0,X_0,\alpha}_s) \lambda ( \d s ), \]
where we minimise over open-loop controls $\alpha$, and $\F_0$-adapted $X_0$ with law $\mu_0$ for $V^\lambda_\Sigma$.
We have to keep track of $\Sigma$ here, because the proof of Proposition \ref{pro:perturbound} below will require a change of reference system. 
Throughout this subsection, we fix $(\mu_0,\lambda)$ and we make the following assumption in addition to \ref{ass:SmoothSett}.

\begin{assumptionV} \label{ass:SmoothBound} $\phantom{a}$
\begin{enumerate}[label=(\roman*),ref=(\roman*)]
\item There exists a $C^1$ non-negative function $(\lambda_s)_{0 \leq s \leq T}$ such that $\lambda ( \d s) = \lambda_s \d s + \overline\lambda_T \delta_T (\d s)$.
\item $(t,x) \mapsto c_t(x)$ is locally Hölder-continuous, as well as all its derivatives w.r.t. $x$. 
\end{enumerate}
\end{assumptionV}

Under \ref{ass:SmoothBound} the above control problem is standard.
To solve it, we introduce the HJB equation in $[0,T] \times \R^d$:
\begin{equation} \label{eq:approxHJB}
\begin{cases}
\partial_t \varphi_t + L^0_{t} \varphi_t - \frac{1}{2} \lvert \sigma^\top_t \nabla \varphi_t \rvert^2 + c_t = -\lambda_t \psi_t, \quad 0 \leq t < T, \\
\varphi_T = \overline\lambda_T \psi_T.
\end{cases}
\end{equation}
From \cite[Theorem 2.2]{chaintron2023existence},
\eqref{eq:approxHJB} has a unique $C^{1,2}$ solution $\varphi$, and
\begin{equation} \label{eq:Approxnabla}
\sup_{t \in [0,T] } \lVert \nabla \varphi_t \rVert_\infty \leq C, 
\end{equation} 
for $C > 0$ that only depends on $\lambda([0,T])$ and the uniform norms of $\sigma$, $\sigma^{-1}$, $\nabla b$, $\nabla \sigma$, $\nabla c$, and $\nabla \psi$.
Moreover, $\varphi$ and its derivatives are Hölder-continuous, and
\begin{equation} \label{eq:LinGrowth} 
\forall (t,x) \in [0,T] \times \R^d, \quad \vert \varphi_t ( x ) \vert \leq C [ 1 + \vert x \vert ], 
\end{equation}
using \cite[Lemma 3.1]{chaintron2023existence}.
Since the coefficients are $C^\infty$, $\varphi_t$ is in fact $C^\infty$ from \cite[Corollary 5.6]{ConstrainedSchrodinger}.
We can now introduce the feed-back control $\alpha^0_s := -\sigma^\top_s \nabla \varphi_s$, because strong existence and pathwise-uniqueness holds for the related SDE \eqref{eq:feedbackSDE} using \ref{ass:SmoothSett}, \eqref{eq:Approxnabla} and the fact that $\varphi$ is $C^2$. 

\begin{lemma}[Verification] \label{lem:VerifApprox}
For every $(t,x)$ in $[0,T] \times \R^d$, $V^\lambda_\Sigma (t,x)$ is uniquely realised by the feed-back control $\alpha^0$.
As a consequence, $V^\lambda_\Sigma ( \mu_0 ) = \int_{\R^d} V_\Sigma^\lambda (0,x) \d \mu_0 (x)$ does not depend on $\Sigma$.
\end{lemma}

This result is classical in control theory, see e.g. \cite[Chapter III.8, Theorem 8.1]{fleming2006controlled}, but we still provide the proof for the sake of completeness.

\begin{proof}
Let $\alpha = ( \alpha_s )_{t \leq s \leq T}$ be an open-loop control.
The regularity of $\varphi$ allows us to apply Ito's formula:
\begin{multline*}
\varphi_T (X^{t,x,\alpha}_{T}) - \varphi_t (x) = \int_t^{T} (\partial_s + L^0_{s}) \varphi_s (X^{t,x,\alpha}_s) + \sigma_s ( X^{t,x,\alpha}_s ) \alpha_s \cdot \nabla \varphi_s ( X^{t,x,\alpha}_s ) \d s
\\
+ \int_t^{T} \nabla \varphi_s\cdot \sigma_s (X^{t,x,\alpha}_s) \d B_s . 
\end{multline*} 
The stochastic integral is a true martingale because of the uniform bounds on $\sigma$ and $\nabla \varphi$.
The HJB equation  \eqref{eq:approxHJB} yields 
\begin{multline*} 
\varphi_t (x) = \int_t^{T} -\alpha_s \cdot \sigma^\top_s\nabla \varphi_s (X^{t,x,\alpha}_s) - \frac{1}{2} \lvert \sigma^\top_s \nabla \varphi_s (X^{t,x,\alpha}_s) \rvert^2 + c_s (X^{t,x,\alpha}_s) + \psi_s (X^{t,x,\alpha}_s) \lambda_s \d s \\
- \int_t^T \nabla \varphi_s \cdot \sigma_s (X^{t,x,\alpha}_s) \d B_s + \overline\lambda_T \psi_T (X^{t,x,\alpha}_T). 
\end{multline*} 
Moreover, 
\begin{equation*}
\int_t^T - \alpha_s \cdot \sigma^\top_s\nabla \varphi_s ( X^{t,x,\alpha}_s ) - \frac{1}{2} \lvert \sigma^\top_s \nabla \varphi_s (X^{t,x,\alpha}_s) \rvert^2 \d s
\leq \int_t^{T} \frac{1}{2} \lvert \alpha_s \rvert^2 \d s \quad \text{a.s},
\end{equation*} 
with equality if and only if $\alpha_s = - a_s \nabla \varphi_s (X^{t,x,\alpha}_s)$ for a.e. $s \in [t,T]$, so that taking expectations, 
\[ \varphi_t (x) \leq \E \, \int_t^T \frac{1}{2} \lvert \alpha_s \rvert^2 + c_s (X^{t,x,\alpha}_s ) + \psi_s (X^{t,x,\alpha}_s) \lambda_s \d s + \overline\lambda_T \psi_T (X^{t,x,\alpha}_T), \]
with equality if and only if $\alpha_s = - a_s \nabla \varphi_s (X^{t,x,\alpha}_s), \; \d s \otimes \P$-a.s.
\end{proof}

We now prove further regularity bounds.
For $R \geq 0$, let $B(0,R)$ denote the centred ball of $\R^d$ with radius $R$.
For technical reasons, we need to introduce the matrix
\[ a^R_t (x) := \rho^R ( x ) a_t (x) + [ 1 - \rho^R (x) ] \mathrm{Id}, \]
where $(\rho^R)_{R \geq 0}$ is any family of smooth functions $\R^d \rightarrow [0,1]$ such that $\rho^R \vert _{B(0,R)} \equiv 1$ and $\rho^R \vert_{B^c (0,R+1)} \equiv 0$, the derivatives of $\rho^R$ being globally bounded uniformly in $R$.  
Let $\sigma^R_t$ denote the positive definite square-root of $a^R_t$.
Let $\varphi^R_t$ denote the solution of \eqref{eq:approxHJB} when replacing $\sigma$ by $\sigma^R$. 
We notice that $\sigma^R$ and $(\sigma^{R})^{-1}$ satisfy the same bounds as $\sigma$ and $\sigma^{-1}$, for regularity constants that are independent of $R$.
As a consequence, $\nabla \varphi^R$ satisfies the bound \eqref{eq:Approxnabla}-\eqref{eq:LinGrowth} independently of $R$.
Let us consider the feed-back control $\alpha^R_t := - ( \sigma^R_t )^\top \nabla \varphi^R_t$, the related SDE \eqref{eq:feedbackSDE} having a pathwise unique strong solution because $\varphi^R$ is $C^2$ with bounded $\nabla \varphi^R$. 
Since $\varphi$ is the classical solution of \eqref{eq:approxHJB}, it is also the mild solution of \eqref{eq:approxHJB} in the sense of Definition \ref{def:mildForm}, which is given by \cite[Theorem 3.6]{ConstrainedSchrodinger}.
As a consequence of \cite[Lemma 5.3]{ConstrainedSchrodinger}, for Lebesgue-a.e. $t \in [0,T]$,
\begin{equation} \label{eq:CVapproxSig}
\varphi^R_t \xrightarrow[R \rightarrow +\infty]{} \varphi_t, \qquad \nabla \varphi^R_t \xrightarrow[R \rightarrow +\infty]{} \nabla \varphi_t, 
\end{equation}
uniformly on every compact set of $\R^d$, if we show that $( \nabla \varphi^R_t )_{R \geq 0}$ is pre-compact on any compact set.
From the uniform bound on $( \nabla \varphi^R_t )_{R \geq 0}$ and the Arzelà–Ascoli theorem, a sufficient condition for this is $( \nabla^2 \varphi^R_t )_{R \geq 0}$ being uniformly bounded.

To the best of our knowledge the following estimate is new in the stochastic control literature.
Based on \cite[Proposition 3.2]{conforti2022coupling}, its key element is a gradient estimate that is proved in \cite{priola2006gradient} using reflection coupling.  

\begin{proposition}[Hessian bound] \label{pro:hessbound}
The following estimate holds 
\[ \sup_{t \in [0,T]} \lVert \nabla \varphi_t \rVert_{\infty} + \sup_{t \in [0,T]} \lVert \nabla^2 \varphi_t \rVert_{\infty} \leq C, \]
for a constant $C$ that only depends on coefficients through $\lambda([0,T])$ and the uniform norms of $\sigma$, $\sigma^{-1}$, $\nabla b$, $\nabla \sigma$, $\nabla c$, $\nabla \psi_s$ and $\nabla^2 \psi_s$.
\end{proposition}

\begin{proof} 
Starting from \eqref{eq:Approxnabla}, we have to prove the Hessian bound. 
Since we do not know that $\nabla^2 \varphi_t$ is bounded beforehand, we first reason on $\varphi^R_t$.
Differentiating the HJB equation, for $1 \leq i \leq d$, 
\begin{equation} \label{eq:DiffHJB}
(\partial_t + L^{\alpha^R}_t) \partial_i \varphi^R_t = - \partial_i b_t \cdot \nabla \varphi^R_t + \frac{1}{2} \nabla \varphi^R_t \cdot \partial_i a^R_t \nabla \varphi^R_t - \frac{1}{2}\mathrm{Tr}[ \partial_i a^R_t \cdot \nabla^2 \varphi^R_t] - \partial_i c_t - \lambda_t \partial_i \psi_t. 
\end{equation} 
With a slight abuse, we still write $L^{\alpha^R}_t$ for the operator where $\sigma_t$ has been replaced by $\sigma^R_t$.
Similarly, we write $X^{t,x,\alpha^R}$ for the process \eqref{eq:feedbackSDE} controlled by the feed-back $\alpha^R$, where the diffusion matrix has been replaced by $\sigma^R_t$.
From \cite[Theorem 2.2]{chaintron2023existence}, $\varphi^R$ is $C^{1,2}$ with Hölder-continuous derivatives.
Since $\nabla b$, $\nabla c$, $\nabla \sigma^R$ and $\nabla \psi$ are Hölder-continuous, \cite[Chapter 3, Theorem 13]{friedman2008partial} guarantees that $\nabla \varphi^R$ is $C^{1,2}$ with Hölder-continuous derivatives.
Ito's formula applied to $\partial_i \varphi^R_s ( X^{t,x,\alpha^R}_s )$ then yields
\begin{multline} \label{eq:ItoGrad}
\partial_i \varphi^R_t ( x ) = \overline\lambda_T \partial_i \psi_T (X^{t,x,\alpha^R}_T) + \int_t^T \big\{ f_s + \frac{1}{2}\mathrm{Tr}[\partial_i a^R_s \nabla^2 \varphi^R_s] + \lambda_s \partial_i \psi_s \big\} (X^{t,x,\alpha^R}_s) \d s \\
+ \int_t^T ( \sigma^R_s)^\top \nabla \partial_i \varphi^R_s (X^{t,x,\alpha^R}_s) \d B_s,
\end{multline} 
where
\[ f_s := \partial_i b_s \cdot \nabla \varphi^R_s - \frac{1}{2} \nabla \varphi^R_s \cdot \partial_i a^R_s \nabla \varphi^R_s + \partial_i c_s.  \]
Using the uniform bound on $\nabla \varphi^R_s$, $f_s$ is bounded by a constant $C > 0$ independent of $(s,R)$, which satisfies the requirements of Proposition \ref{pro:hessbound}. 
In the following, $C$ may change from line to line, but still satisfying these requirements.
For $y \in \R^d$, we similarly decompose $\partial_i \varphi^R_t ( y )$ and we subtract this expression to  \eqref{eq:ItoGrad}, before taking
expectations to get rid of the martingale.
Since $\alpha^R$ is bounded continuous, we can use the estimate from \cite[Theorem 3.4 (a)]{priola2006gradient}.
Strictly speaking, \cite[Theorem 3.4 (a)]{priola2006gradient} only allows for time-homogeneous coefficients, but
the adaptation of these arguments to the time-dependent setting is straight-forward.
This yields
\begin{equation} \label{eq:PriolaWang}
\forall (x,y) \in \R^d \times \R^d, \quad \lvert \E [ f_s ( X^{t,x,\alpha^R}_s ) ] - \E [ f_s ( X^{t,y,\alpha^R}_s ) ] \rvert \leq \frac{C \vert x - y \vert}{1 \wedge \sqrt{t-s}}. 
\end{equation} 
We similarly handle the other terms to obtain that
\[ \lvert \partial_i \varphi^R_t (x) - \partial_i \varphi^R_t (y) \rvert \leq C \overline{\lambda}_T \vert x- y \rvert + \int_t^T \frac{C \lvert x- y \rvert}{1 \wedge \sqrt{s-t}} [ 1 + \lVert \partial_i a^R_s \nabla^2 \varphi^R_s \rVert_\infty + \lambda_s ] \d s, \]
where $\lVert \partial_i a^R_s \nabla^2 \varphi^R_s \rVert_\infty$ is finite because $\nabla^2 \varphi^R$ is continuous and $\partial_i a^R_s \equiv 0$ out of $B(0,R+1)$.
Dividing by $\lvert x - y \rvert$ and taking the supremum over $(x,y)$ and $i$, this proves that $\sup_{0 \leq t \leq T} \lVert \nabla^2 \varphi^R_t \rVert_\infty$ is finite.

We now perform the same reasoning, but applying Ito's formula to $\partial_i \varphi^R_s ( X^{t,x,0}_s )$  instead of $\partial_i \varphi^R_s ( X^{t,x,\alpha^R}_s )$, still with $\sigma^R$ instead of $\sigma$ in the definition of $X^{t,x,0}$. 
Using that $L^{\alpha^R}_t = L^0_t + \sigma^R_t \alpha^R_t \cdot \nabla$ in \eqref{eq:DiffHJB}, this yields
\begin{multline*} 
\partial_i \varphi^R_t ( x ) = \int_t^T \big\{ f_s + \sigma^R_s \alpha^R_s \cdot \nabla \partial_i \varphi^R_s + \frac{1}{2}\mathrm{Tr}[\partial_i a^R_s \nabla^2 \varphi^R_s] + \lambda_s \partial_i \psi_s \big\} (X^{t,x,0}_s) \d s \\
+ \overline\lambda_T \partial_i \psi_T (X^{t,x,0}_T) + \int_t^T ( \sigma^R_s )^\top \nabla \partial_i \varphi^R_s (X^{t,x,0}_s) \d B_s,
\end{multline*} 
Using \eqref{eq:PriolaWang} as previously, with $X^{t,x,0}$ instead of $X^{t,x,\alpha^R}$, we obtain
\begin{multline} \label{eq:ItoGrad2}
\lvert \partial_i \varphi^R_t (x) - \partial_i \varphi^R_t (y) \rvert \leq C \overline{\lambda}_T \vert x- y \rvert + \int_t^T \frac{C \lvert x- y \rvert}{1 \wedge \sqrt{s-t}} [ 1 + \lVert \nabla^2 \varphi^R_s \rVert_\infty ] \d s \\
+ \int_t^T \lambda_s 
\vert \E [ \partial_i \psi_s ( X^{t,x,0}_s ) ]  -  \E [ \partial_i \psi_s ( X^{t,y,0}_s ) ] \vert \d s. 
\end{multline}
Using e.g. synchronous coupling, it is standard that    
\[ \vert \E [ \partial_i \psi_s ( X^{t,x,0}_s ) ]  -  \E [ \partial_i \psi_s ( X^{t,y,0}_s ) ] \vert \leq C \lVert \nabla \partial_i \psi_s \rVert_\infty \lvert x -y \vert. \]
The last term in \eqref{eq:ItoGrad2} can then be bounded using that $\lambda ([0,T]) \leq C$. 
We now divide by $\lvert x - y \rvert$, and we take the supremum over $(x,y)$ and $i$ to get that 
\[ \forall t \in [0,T], \quad \lVert \nabla^2 \varphi^R_t \big\rVert_\infty \leq C + \int_t^T \frac{C}{1 \wedge \sqrt{s-t}} [ 1 + \lVert \nabla^2 \varphi^R_s \rVert_\infty ] \d s. \]
Using the Gronwall lemma, we eventually obtain
\[ \sup_{0 \leq t \leq T} \lVert \nabla^2 \varphi^R_t \big\rVert_\infty \leq C \bigg[ 1 + \int_0^T \frac{\d s}{1 \wedge \sqrt{s}} \bigg] \exp \bigg[ \int_0^T \frac{C}{1 \wedge \sqrt{s}} \d s \bigg], \]
where the constant $C$ is independent of $R$ and satisfies the requirements of Proposition \ref{pro:hessbound}.
This uniform
bound on $( \nabla^2 \varphi^R_t )_{R \geq 0}$ implies the convergence \eqref{eq:CVapproxSig}, for Lebesgue-a.e. $t \in [0,T]$.
From the Banach-Alaoglu theorem, $( \nabla^2 \varphi^R_t )_{R \geq 0}$ has a weakly-$\star$ converging sub-sequence as $R \rightarrow + \infty$ for the $\sigma(L^\infty,L^1)$-topology, whose limit still satisfies the above bound.
Since $\nabla^2 \varphi_t$ was already well-defined, this proves that $\nabla^2 \varphi^R_t \rightarrow \nabla^2 \varphi_t$ weakly-$\star$.
In particular, $\nabla^2 \varphi_t$ satisfies the above bound, for Lebesgue-a.e. $t \in [0,T]$. 
Since $t \mapsto \nabla^2 \varphi_t (x)$ is continuous, this concludes the proof.
\end{proof}

In the above proof, the dependence on
$\nabla^2 \psi_s$ for $C$ is needed because we want an estimate that only depends on $\lambda$ through $\lambda([0,T])$.
Otherwise, instead of synchronous coupling, we would use reflection coupling to handle the last term in \eqref{eq:ItoGrad2} as the other ones.
Still relying on the estimate from \cite{priola2006gradient}, the next result is a new third-order bound for HJB equations. 
In contrast with the previous bound, which only depends on $\lambda([0,T])$, the next bound depends on $\sup_{0 \leq t \leq T} \vert \lambda_t \rvert$. 
However, it is crucial for the sequel that this dependence comes through a multiplicative constant that can be taken as small as desired.

\begin{lemma}[Third-order bound] \label{lem:thirdorder}
For every $\varepsilon >0$, there exists $C_\varepsilon >0$ such that
\[ \sup_{t \in [0,T]} \lVert \nabla^3 \varphi_t \rVert_{\infty} \leq C_\varepsilon + \varepsilon \sup_{0 \leq t \leq T} \vert \lambda_t \rvert, \]
for a constant $C_\varepsilon$ that only depends on coefficients through $\lambda ( [0,T] )$ and the uniform norms of $\sigma$, $\sigma^{-1}$, $\nabla b$, $\nabla \sigma$, $\nabla c$, $\nabla \psi_s$, $\nabla^2 b$, $\nabla^2 \sigma$, $\nabla^2 c$ and $\nabla^2 \psi_s$.
\end{lemma}

\begin{proof}
As in the proof of Proposition \ref{pro:hessbound}, we first reason on $\varphi^R$.
For $1 \leq i,j \leq d$, we differentiate \eqref{eq:DiffHJB} and we use the bounds that were proved on $\varphi^R$ to get as previously that 
\begin{equation*} 
(\partial_t + L^{\alpha^R}_t) \partial_{i,j} \varphi^R_t = - f_t - \frac{1}{2}\mathrm{Tr}[ \partial_i a^R_t \partial_j \nabla^2 \varphi^R_t] - \lambda_t \partial_{i,j} \psi_t, 
\end{equation*} 
for some measurable $f : [0,T] \times \R^d \rightarrow \R$ that is bounded by a constant $C > 0$ independent of $R$, which satisfies the requirements of Lemma \ref{lem:thirdorder}.
In the following, $C$ may change from line to line, but still satisfying these requirements.
Since $\nabla^2 b$, $\nabla^2 c$, $\nabla^2 \sigma^R$ and $\nabla^2 \psi$ are Hölder-continuous, \cite[Chapter 3, Theorem 13]{friedman2008partial} guarantees as previously that $\nabla^2 \varphi^R$ is $C^{1,2}$ with Hölder-continuous derivatives.
Ito's formula applied to $\partial_{i,j} \varphi^R_s ( X^{t,x,\alpha^R}_s )$ then yields
\begin{multline*} \label{eq:ItoHess}
\partial_{i,j} \varphi^R_t ( x ) = \overline\lambda_T \partial_{i,j} \psi_T (X^{t,x,\alpha^R}_T) + \int_t^T \big\{ f_s + \frac{1}{2}\mathrm{Tr}[\partial_i a^R_s \partial_j \nabla^2 \varphi^R_s] + \lambda_s \partial_{i,j} \psi_s \big\} (X^{t,x,\alpha^R}_s) \d s \\
+ \int_t^T ( \sigma^R_s )^\top \nabla \partial_{i,j} \varphi^R_s (X^{t,x,\alpha^R}_s) \d B_s.
\end{multline*} 
As previously, we subtract $\partial_{i,j} \varphi^R_t ( y )$, we take expectations, and we use \eqref{eq:PriolaWang} to get that 
\[ \lvert \partial_{i,j} \varphi^R_t (x) - \partial_{i,j} \varphi^R_t (y) \rvert \leq C \vert x - y \rvert + \int_t^T \frac{C \lvert x- y \rvert}{1 \wedge \sqrt{s-t}} \big[ 1 + \lVert \partial_i a^R_s \nabla^3 \varphi^R_s \rVert_\infty + \lambda_s \big] \d s, \]
where $\lVert \partial_i a^R_s \nabla^3 \varphi^R_s \rVert_\infty$ is finite because $\nabla^3 \varphi^R$ is continuous and $\partial_i a^R_s \equiv 0$ out of $B(0,R+1)$.
Dividing by $\lvert x - y \rvert$ and taking the supremum over $(x,y)$ and $(i,j)$, this proves that $\sup_{0 \leq t \leq T} \lVert \nabla^3 \varphi^R_t \rVert_\infty$ is finite.
For any $\varepsilon \in (0,1)$, we split the last integral between $t$ and $t + \varepsilon^2$:
\[ \int_t^T \frac{\lambda_s}{\sqrt{t-s}} \d s \leq 2 \varepsilon \sup_{t \leq s \leq t + \varepsilon^2} \vert \lambda_s \vert + \frac{1}{\varepsilon} \int_{t + \varepsilon^2}^T \lambda_s \d s, \]
so that
\[ \forall t \in [0,T], \quad \lVert \nabla^3 \varphi^R_t \big\rVert_\infty \leq C + \int_t^T \frac{C}{1 \wedge \sqrt{s-t}} \big[ \varepsilon^{-1} + \lVert \nabla^3 \varphi^R_s \rVert_\infty + \varepsilon \sup_{0 \leq r \leq T} \vert \lambda_r \rvert \big] \d s. \]
Using the Gronwall lemma, we eventually obtain
\[ \sup_{0 \leq t \leq T} \lVert \nabla^3 \varphi^R_t \big\rVert_\infty \leq C \big[ \varepsilon^{-1} + \varepsilon \sup_{0 \leq t \leq T} \vert \lambda_t \vert \big] \bigg[ 1 + \int_0^T \frac{\d s}{1 \wedge \sqrt{s}} \bigg] \exp \bigg[ \int_0^T \frac{C}{1 \wedge \sqrt{s}} \d s \bigg],\]
where the constant $C$ is independent of $R$ and satisfies the requirements of Lemma \ref{lem:thirdorder}.
We conclude the proof using the convergence \eqref{eq:CVapproxSig} as previously.
\end{proof}

We fix $X_0 \sim \overline\mu_0$, and we shall write $X^0 := X^{0,X_0,\alpha^0}$.
Since $\overline\mu_0 (\d x ) = \overline{Z}^{-1} e^{-\overline\varphi_0(x)} \nu_0 ( \d x)$ from \cite[Theorem 3.6]{ConstrainedSchrodinger}, and $\overline{\varphi}_0$ has linear growth from \eqref{eq:LinGrowth}, \ref{ass:ini1} implies the following useful result.

\begin{lemma}[Initial bound] \label{lem:InitalEntrop}
If any of the bounds \ref{ass:ini2}-\ref{item:inih1}-\ref{item:inih2}-\ref{item:iniT1} is satisfied by $\nu_0$ for some ${\varepsilon}_0 > 0$, then the same property holds for $\overline{\mu}_0$ with $\varepsilon_0 = 0$, for a constant $C > 0$ that satisfies the same requirements as the one in \eqref{eq:LinGrowth}. 
\end{lemma}

The following result is crucial to obtain density estimates on the multiplier.

\begin{proposition} \label{pro:perturbound}
There exist $\delta \in (0,1)$ and $C > 0$ such that  for any $C^1$ function $h : [t,t'] \rightarrow [0,1]$ with $0 \leq t < t'\leq T$, $h_{t} = h_{t'} =0$ and $\lVert h \rVert_{\infty} + \lVert \tfrac{\d}{\d s} h \rVert_{\infty}^2 \leq \delta$,
\[ \int_{t}^{t'} h_s \E [ \lvert \sigma_s \nabla \psi_s ( X^{0}_s ) \rvert^2 ] \lambda_s \d s \leq \int_{t}^{t'} C ( h_s + \lvert \tfrac{\d}{\d s} h_s \rvert^2 ) - 2 \big\{ \tfrac{\d}{\d s} \E [ \psi_s ( X^{0}_s ) ] - \E [ \partial_s \psi_s ( X^{0}_s ) ] \big\} \tfrac{\d}{\d s} h_s \d s.
\]
Moreover, under \ref{ass:ini2}-\ref{item:inih1}, $\delta$ and $C$ only depend on $\lambda$ through $\lambda ([0,T])$, and $C$ only depends on the other coefficients through the uniform norms of $\sigma$, $\sigma^{-1}$, $\partial_s b$, $\partial_s \sigma$, $\partial_s \psi$, $\nabla b$, $\nabla \sigma$, $\nabla c$, $\nabla \psi$, $\nabla \partial_s \psi$, $\nabla^2 \sigma$, $\nabla^2 \psi$ and $\nabla^3 \psi$.
Under \ref{ass:ini2}-\ref{item:inih2},
for every $\varepsilon >0$, there further exists $C_\varepsilon > 0$ satisfying the same requirements except that $C_\varepsilon$ depends on $\nabla^2 b$ and $\nabla^3 \sigma$ but not on $\nabla^3 \psi$, such that the bound holds for the constant $C := C_\varepsilon + \varepsilon \sup_{0 \leq s \leq T} \lvert \lambda_s \rvert$.
\end{proposition} 

In order to prove the above upper bound, we are going to construct a competitor in the stochastic control problem by pushing the marginal flow of the optimal solution along the gradient flow generated by $\psi$. Such construction shares some analogies with earlier results obtained e.g. in \cite[Section 4]{erbar2015large} or \cite{monsaingeon2023dynamical} in the context of the sudy of (an abstract version of) the Schr\"odinger problem.

\begin{proof}
Let $(\mu^0_s)_{0 \leq s \leq T}$ be the flow of marginal laws of $( X^{0}_s )_{0 \leq s \leq T}$.
To prove the bound, we build a competitor $\alpha^h$ for $V^\lambda_\Sigma (\overline\mu_0)$ in \emph{\textbf{Step 1.}}, by perturbing the optimal flow on $[t,t']$ only.
The related controlled process $X^h$ is such that $\E [ \psi_s ( X^h_s ) ] < \E [ \psi_s ( X^0_s ) ]$ for $s \in (t,t')$. 
The related cost $\E \int_0^T \tfrac{1}{2} \vert \alpha^h_s ( X^h_s ) \vert^2 \d s$ is computed in \emph{\textbf{Step 2.}} and \emph{\textbf{Step 3.}}.
The comparison between $X^0$ and $X^h$ is done in \emph{\textbf{Step 4.}}, yielding the desired estimate.
Final computations are done in \emph{\textbf{Step 5.}} to obtain the wanted dependence on the coefficients.

\medskip

\emph{\textbf{Step 1.} Perturbation.}
For $s \in [0,T]$, let $( S_{s,t} )_{t \geq 0}$ be the flow defined by 
\[ \forall x \in \R^d, \forall t \geq 0, \quad
\frac{\d}{\d t} S_{s,t} (x) = - a_s \nabla \psi_s (S_{s,t} (x)), \quad S_{s,0}(x) = x,
\]
where $a_s := \sigma_s \sigma_s^\top$.
Since $\psi_s$ is $C^4$ and $a_s$ is bounded and $C^3$, both with uniformly bounded derivatives, the Cauchy-Lipschitz theory guarantees that $x \mapsto S_{s,t}(x)$ is a $C^3$-diffeomorphism.
We define the perturbed flow as
\[ \mu^h_s := \mu^0_s \circ S^{-1}_{s,h_s}, \quad 0 \leq s \leq T. \]
Writing $h'_s := \tfrac{\d}{\d s} h_s$, Ito's formula yields
\begin{equation*} \label{eq:PushedProcess}
\d S_{s,h_s}(X^0_s) = [ - h'_s a_s \nabla \psi_s ( S_{s,h_s}(X^0_s)) + (\partial_s + L^{\alpha^0}_s ) S_{s,h_s}( X^0_s ) ] \d s + \nabla S_{s,h_s}( X^0_s ) \cdot \sigma_s (X^0_s) \d B_s.  
\end{equation*}
Since $a$ is uniformly elliptic and the coefficients are smooth from \ref{ass:SmoothSett}, it is standard that $\mu^0_s$ has a positive $C^2$ density for $s > 0$.
Since $x \mapsto S_{s,t}(x)$ is a $C^3$-diffeomorphism, $\mu^h_s$ also has a positive $C^2$ density that we still write $x \mapsto \mu^h_s (x)$ with a slight abuse.
We then define 
\begin{align*} 
A^h_s &:= - h'_s a_s \nabla \psi_s, \quad B^h_s := (\partial_s + L^{\alpha^0}_s ) [S_{s,h_s}] \circ S^{-1}_{s,h_s} - b_s, \\
C^h_s &:= - \tfrac{1}{2} ( \mu^h_s )^{-1} \nabla \cdot [ \mu^h_s ( \nabla S_{s,h_s} a_s \nabla S^\top_{s,h_s} \circ S^{-1}_{s,h_s} - a_s ) ], \quad \beta^h_s := A^h_s + B^h_s + C^h_s. 
\end{align*}
We refer to the notation section \ref{subsec:Notations} for the meaning of operators applied to vector fields or matrices. 
Let us now define the feed-back $\alpha^h_s := \sigma^{-1}_s \beta^h_s$.
Since the law of $S_{s,h_s} (X^0_s)$ is precisely $\mu^h_s$, we get that
$\partial_s \mu^h_s = L^{\alpha^h}_s \mu^h_s$ in the sense of distributions.
To obtain an admissible control from $\alpha^h$, we need to show that $\alpha^h$ has a finite $L^2(\d s \otimes \mu^h_s)$-norm by computing $\E [ \lvert \alpha^h_s(S_{s,h_s}(X^0_s)) \rvert^2 ]$.
This computation is done in \emph{\textbf{Step 3.}}, after computing $\E [ \lvert \beta^h_s(S_{s,h_s}(X^0_s)) \rvert^2 ]$ in \emph{\textbf{Step 2.}}.

\medskip

\emph{\textbf{Step 2.} Expansion.}
In the following, we write $O(h_s^i)$ for quantities that are bounded by $C \lvert h_s \rvert^i$ where the constant $C$ is allowed to depend only on the uniform norm of the coefficients and their derivatives (of any order). 
For $(i,j) \in \{0,1\} \times \{0,1,2\}$, we will repeatedly use the (uniform in $x$) expansion
\[ \partial^i_s \nabla^j S_{s,h_s} = \partial^i_s \nabla^j \mathrm{Id} - h_s \nabla^{j} \partial^i_s (a_s \nabla \psi_s) + O(h^2_s ), \]
which stems from a uniform bound on $\nabla^{j+1} \partial_s^{i} ( a \nabla \psi)$ (we will always have $i = 0$ when $j = 2)$. 
Let us start the expansion of $\beta^h_s \circ S_{s,h_s}$: for $x$ in $\R^d$,
\begin{align*}
A^h_s (S_{s,h_s}(x)) &= - h'_s a_s \nabla \psi_s (x) + O (h'_s h_s ), \\
B^h_s (S_{s,h_s}(x)) &= \sigma_s \alpha^0_s (x) - h_s [(\partial_s + L^{\alpha^0}_s) [ a_s \nabla \psi_s ] - \nabla b_s a_s \nabla \psi_s] (x) + O((1 + \lvert x \rvert)h^2_s),
\end{align*} 
using that $b$ is globally Lipschitz.
Similarly,
\[ \nabla S_{s,h_s} a_s \nabla S^\top_{s,h_s} \circ S^{-1}_{s,h_s} - a_s = - h_s [ a_s [ \nabla (a_s \nabla \psi_s) ]^\top + \nabla (a_s \nabla \psi_s) a_s - \nabla a_s \cdot a_s \nabla \psi_s] + O(h^2_s), \]
using that $\nabla^2 ( a \nabla \psi)$ is uniformly bounded. 
Using the bound on $\nabla^3 ( a \nabla \psi)$, we turn this into
\begin{align*}
C^h_s (S_{s,h_s}(x)) =&\phantom{+!} \tfrac{h_s}{2} [ a_s [ \nabla (a_s \nabla \psi_s) ]^\top + \nabla (a_s \nabla \psi_s) a_s - \nabla a_s \cdot a_s \nabla \psi_s] [ (\nabla \log \mu^h_s ) \circ S_{s,h_s} ] \\
&+ \tfrac{h_s}{2} \nabla \cdot [ a_s [ \nabla (a_s \nabla \psi_s) ]^\top + \nabla (a_s \nabla \psi_s) a_s - \nabla a_s \cdot a_s \nabla \psi_s] + O(h^2_s) \\
=&\!\!: \, D^h_s [ (\nabla \log \mu^h_s ) \circ S_{s,h_s} ] + E^h_s + O(h^2_s).
\end{align*}
Gathering everything:
\begin{multline} \label{eq:ExprBeta}
\beta^h_s (S_{s,h_s}(X^0_s)) =\beta^0_s (X^0_s) - h'_s a_s \nabla \psi_s (X^0_s) - h_s [(\partial_s + L^{\alpha^0}_s) [ a_s \nabla \psi_s ] - \nabla b_s a_s \nabla \psi_s] (X^0_s) \\
+ D^h_s ( X^0_s ) \nabla \log \mu^h_s ( S_{s,h_s} (X^0_s)) + E^h_s (X^0_s ) + O((1+\lvert X^0_s \rvert)h^{2}_s + h'_s h_s).
\end{multline}
Since $\mu^h_s = \mu^0_s \circ S^{-1}_{s,h_s}$:
\begin{align*}
\nabla \log \mu^h_s (S_{s,h_s}(X^0_s)) &= \nabla S^{-1}_{s,h_s} (S_{s,h_s}(X^0_s)) \nabla \log \mu^0_s (X^0_s) + \nabla \log \lvert \mathrm{det} \nabla S^{-1}_{s,h_s} \rvert ( S_{s,h_s}(X^0_s) ) \\
&= \nabla \log \mu^0_s (X^0_s) + O((1+\lvert  \nabla \log \mu^0_s (X^0_s) \rvert)h),
\end{align*}
where we used that $\nabla^2( a \nabla \psi)$ is uniformly bounded to get the $O(h)$ term.
Using Lemma \ref{lem:InitalEntrop} and the bound on $\nabla \alpha^0$ given by Proposition \ref{pro:hessbound}, we can apply Lemma \ref{lem:Gamma_2} to $(\mu^0_{s})_{0 \leq s \leq T}$.
This provides a bound on $\E [\lvert \nabla \log \mu^0_s (X^0_s) \rvert]$ involving $\int_{\R^d} \vert \nabla \log \nu_0 \rvert^{2 +\varepsilon_0} \d \nu_0$. 
Bounding $\E [\lvert X^0_s \rvert ]$ is then standard from \ref{ass:coefReg1} and the bound on $\nabla \alpha^0$, so that
$\E [ \lvert\beta^h_s (S_{s,h_s}(X^0_s)) \rvert^2 ]$ is finite. 

\medskip

\emph{\textbf{Step 3.} Integrations by parts.}
Since $\E [\lvert \nabla \log \mu^0_s (X^0_s) \rvert]$ is bounded, we can integrate by parts the density $\mu^0_s$ to obtain that
\[ \E [ D^h_s ( X^0_s ) \nabla \log \mu^0_s ( X^0_s ) ] = - \E [ E^h_s ( X^0_s ) ]. \]
Plugging this cancellation within \eqref{eq:ExprBeta} yields
\begin{multline*}
\E [\lvert \beta^h_s (S_{s,h_s}(X^0_s)) \rvert^2 - \lvert \beta^0_s (X^0_s) \rvert^2] = - 2 h'_s \E [ \sigma_s \alpha^0_s \cdot a_s \nabla \psi_s (X^0_s)] \\
- 2 h_s \E \{ \sigma _s \alpha^0_s \cdot [ \partial_s ( a_s \nabla \psi_s) + \tfrac{1}{2} \mathrm{Tr}[ a_s \nabla^2 ( a_s \nabla \psi_s ) ] ] (X^0_s)\} + C' h_s + O(h_s^2 + (h'_s)^2 ),
\end{multline*}
where $C'$ is a constant that satisfies the requirements of Proposition \ref{pro:perturbound}.
In the following, $C'$ may change from line to line but still satisfying these requirements. 
The quantities $\nabla \partial_s \psi_s$ and $\nabla^3 \psi_s$ appear in the first term of the second line. 
To get rid of them, we use coordinates (and the summing convention of repeated indices) to rewrite this term as 
\[ - 2 h_s \E \{ \sigma_s^{i,j} (\alpha^0_s)^j [ \partial_s a_s^{i,k} \partial_k \psi_s + a_s^{i,k} \partial_k \partial_s \psi_s + \tfrac{1}{2} a_s^{l,m} \partial_{l,m} ( a_s^{i,k} \partial_k \psi_s ) ] (X^0_s) \}. \]
As above, we integrate by parts the $\partial_k$-derivative to get that 
\[ \E [\lvert \beta^h_s (S_{s,h_s}(X^0_s)) \rvert^2 - \lvert \beta^0_s (X^0_s) \rvert^2] = - 2 h'_s \E [ \sigma_s \alpha^0_s \cdot a_s \nabla \psi_s (X^0_s)] + C' h_s + O(h_s^2 + (h'_s)^2 ). \]
We can now write that
\begin{multline*}
\alpha^h_s (S_{s,h_s}(X^0_s)) = \sigma^{-1}_s (X^0_s) \beta^h_s (S_{s,h_s}(X^0_s)) - h_s \beta^h_s (S_{s,h_s}(X^0_s)) \cdot \nabla \sigma^{-1}_s a_s \nabla \psi_s (X^0_s) 
\\ + O(h^2_s \beta^h_s (S_{s,h_s}(X^0_s))), 
\end{multline*} 
using that $\nabla \sigma^{-1}$ is uniformly bounded, and the same computations as above yield
\[ \E [\lvert \alpha^h_s (S_{s,h_s}(X^0_s)) \rvert^2 ] - \E [ \lvert \alpha^0_s (X^0_s) \rvert^2] = - 2 h'_s \E [ \sigma_s \alpha^0_s \cdot \nabla \psi_s (X^0_s)] + C' h_s + O(h_s^2 + (h'_s)^2 ). \]
In particular, we obtained that $\int_0^T \E [ \lvert \alpha^h_s(S_{s,h_s}(X^0_s)) \rvert^2 ] \d s$ is finite.

\medskip

\emph{\textbf{Step 4.} Comparison.}
Using \emph{\textbf{Step 3.}} and \cite[Theorem 2.5]{TrevisanMart}, there exists a weak solution for the SDE \eqref{eq:feedbackSDE} with the feed-back control $\alpha = \alpha^h$ and the initial law $\overline\mu_0$ (\cite[Theorem 2.5]{TrevisanMart} actually shows existence of a solution for the related martingale problem, but this is equivalent, see e.g.
\cite[Proposition 2.1]{daudin2020optimal}). As a consequence, there exists a reference probability system $\Sigma'$ and an adapted process $( X^{h}_s )_{0\leq s\leq T}$ on it satisfying \eqref{eq:feedbackSDE} with $\alpha = \alpha^h$ and $\L ( X^h_0 ) = \overline\mu_0$.
In particular, $( \alpha^h_s ( X^h_s) )_{0 \leq s \leq T}$ can be seen as a specific open-loop control for \eqref{eq:relaxcontrolled}. 
Thus,
\begin{multline*}
V^\lambda_{\Sigma}(\overline\mu_0) = V^\lambda_{\Sigma'}(\overline\mu_0) \leq \E \bigg[ \int_0^T \frac{1}{2} \lvert \alpha^h_s (X^h_s) \rvert^2 + c_s (X^h_s) + \psi_s ( X^h_s) \lambda_s \d s + \overline\lambda_T \psi_T (X^{h}_T) \bigg] \\
= \int_0^T \E \big[ \frac{1}{2} \lvert \alpha^h_s ( S_{s,h_s} (X^0_s) ) \rvert^2 + c_s (S_{s,h_s} (X^0_s)) + \psi_s ( S_{s,h_s} (X^0_s)) \lambda_s \big] \d s + \overline\lambda_T \E [ \psi_T (X^0_T) ] ,
\end{multline*}
using that $X^h_s$ and $S_{s,h(s)} (X^0_s)$ have the same law $\mu^h_s$, and $X^h_T = X^0_T$.
We further expand
\[ c_s (S_{s,h_s} (X^0_s)) = c_s(X^0_s) - h_s \nabla c_s \cdot a_s \nabla \psi_s ( X^0_s ) + O(h^2_s) = C' h_s + O(h^2_s),  \]
\[ \psi_s (S_{s,h_s} (X^0_s)) = \psi_s (X^0_s) - h_s \nabla \psi_s \cdot a_s \nabla \psi_s ( X^0_s ) + O(h^2_s),  \]
using that $\nabla^2 c$ and $\nabla^2 \psi$ are uniformly bounded. Using \emph{\textbf{Step 3.}} to expand $\alpha^h_s$, we obtain $\delta > 0$ that only depends on $\lambda$ through $\lambda ([0,T])$ such that $\lVert h \rVert_{\infty} + \lVert h' \rVert_{\infty}^2 \leq \delta$ implies 
\[ \int_{t}^{t'} h_s \E [ \lvert \sigma^\top_s \nabla \psi_s ( X^0_s ) \rvert^2 ] \lambda_s \d s \leq \int_{t}^{t'} C' ( h_s + \lvert h'_s \rvert^2 ) - 2 h'_s \E [ \sigma_s \alpha^0_s \cdot \nabla \psi_s (X^0_s)] \d s. \]
The first term in the integral on the r.h.s. has the desired shape, but the second one still involves the leading order term $h'_s$, so we must integrate it. 

\medskip

\emph{\textbf{Step 5.} Last integrations by parts.}
Ito's formula gives that
\[ \E [ \sigma_s \alpha^0_s \cdot \nabla \psi_s (X^0_s)] = \frac{\d}{\d s} \E [ \psi_s ( X_s ) ] - \E [ (\partial_s + L^0_s) \psi_s (X^0_s) ], \]
and we integrate by parts in time using that $h_t = h_{t'} =0$:
\[ -\int_t^{t'} h'_s \E [ L^0_s \psi_s (X^0_s) ] \d s = \int_t^{t'} h_s \frac{\d}{\d s} \E [ L^0_s \psi_s (X^0_s) ] \d s = \int_t^{t'} h_s \E [ (\partial_s + L^{\alpha^0_s} ) L^0_s \psi_s (X^0_s) ] \d s. \]
As previously, we integrate by parts the density to get rid of the $\nabla^2 \partial_s \psi_s$ term:
\[ \E \, \mathrm{Tr}[a_s \nabla^2 \partial_s \psi_s (X^0_s) ]  = C'- \E \, [ a^{i,j}_s \partial_i \partial_s \psi_s \partial_j \log \mu^0_s (X^0_s) ].\]
Similarly, we get rid of the $\nabla^2 b_s$ and $\nabla^4 \psi_s$ terms.
The resulting $C'$ involves $\int \vert \nabla \log \nu_0 \vert^{2 + \varepsilon_0} \d \nu_0$ and a $\nabla^3 \psi_s$ term.
To get rid of this last term, we can integrate by parts one more time:
\[ \E \, \mathrm{Tr}\{ a_s \nabla^2 \mathrm{Tr} [ a_s \nabla^2 \psi_s ] \} (X^0_s) \leq C'- \E \, \mathrm{Tr} [ a_s \nabla^2 \psi_s ] \mathrm{Tr}[a_s \nabla^2 \log \mu^0_s ] (X^0_s). \]
We deduce from Proposition \ref{pro:Gamma_3} applied to $(\mu^0_s)_{0 \leq s \leq T}$ that
\[ \sup_{0 \leq s \leq T} \E [ \vert \nabla^2 \log \mu^0_s (X^0_s) \vert ] \leq C [ 1 + \sup_{0 \leq s \leq T} \lVert \nabla^3 \varphi_s \rVert_\infty ], \]
where the constant $C$ now involves $\int \vert \nabla \log \nu_0 \vert^{4 + \varepsilon_0} \d \nu_0$ and $\int \vert \nabla^2 \log \nu_0 \vert^{2 + \varepsilon_0}$ too.
The desired dependence on $\varepsilon \sup_{0 \leq s \leq T} \vert \lambda_s \vert$ eventually results from Lemma \ref{lem:thirdorder}.
\end{proof}

\begin{corollary}[Qualification condition] \label{cor:Qualif}
A sufficient condition for \eqref{eq:smoothQualif} is that
\begin{equation} \label{eq:QuallifDiff}
\forall s \in [0,T], \quad \int_{\R^d} \psi_s \d \overline\mu_s = 0 \Rightarrow \int_{\R^d} \vert \nabla \psi_s \rvert^2 \d \overline\mu_s \neq 0.    
\end{equation}
This condition is also necessary from Remark \ref{rem:Suffqualif}.
\end{corollary}

\begin{proof}
We follow the proof of Proposition \ref{pro:perturbound}, defining the perturbed flow
\[ \mu^h_s := \overline\mu_s \circ S^{-1}_{s,h}, \quad 0 \leq s \leq T, \]
for $h > 0$.
As previously, we can show that for $h$ small enough$, ( \mu^h_s )_{0 \leq s \leq T}$ are the marginal laws of a process of type \eqref{eq:relaxcontrolled}, for a square-integrable control process.
Moreover,
\[ \frac{\d}{\d h} \int_{\R^d} \psi_s \d \mu^h_s = - \int_{\R^d} \vert \sigma^\top_s \nabla \psi_s \rvert^2 \d \mu^h_s \leq 0. \]
Since $\sigma^{-1}$ is bounded from \ref{ass:coefReg1}, \eqref{eq:QuallifDiff} guarantees that the derivative is negative at points where $\int_{\R^d} \psi_s \d \overline\mu_s$ (which is always non-positive) is not already negative. 
As a consequence \eqref{eq:smoothQualif} is satisfied, choosing $\tilde\nu_{[0,T]}$ as being the path-law of a controlled process with marginal laws $( \mu^h_s )_{0 \leq s \leq T}$, for $h$ small enough.
\end{proof}

\subsection{Existence of a density for the Lagrange multiplier} \label{subsec:smoothdensity}

In our smooth setting \ref{ass:SmoothSett}, \cite[Theorem 3.6]{ConstrainedSchrodinger} gives that the measure $\overline\mu_{[0,T]}$ satisfying \eqref{eq:LinDensity} is the path-law of the solution to
\[ \d \overline{X}_t = b_t (  \overline{X}_t ) \d t - a_t \nabla \overline\varphi_t (  \overline{X}_t ) \d t + \sigma_t ( \overline{X}_t ) \d B_t, \quad  \overline{X}_0 \sim \overline{Z}^{-1} e^{-\overline\varphi_0 (x)} \nu_0 (\d x), \]
in a reference system $\Sigma = (\Omega, (\F_t)_{0 \leq t \leq T},\P,(B_t)_{0 \leq t \leq T})$,
where $\overline\varphi$ is the solution to
\[ - \overline\varphi_t + \int_t^T b_s \cdot \nabla \overline\varphi_s - \frac{1}{2} \vert \sigma_s^\top \nabla \overline\varphi_s \vert^2 + c_s + \frac{1}{2} \mathrm{Tr}[a_s \nabla^2 \overline\varphi_s] \d s + \int_{[t,T]} \psi_s \overline{\lambda}(\d s) = 0, \]
in the sense of Definition \ref{def:mildForm}.

Using convolution with a smooth kernel,
let us consider approximations of $\overline{\lambda} - \overline\lambda_T \delta_T$ given by measures that have $C^1$ densities $(\lambda^k)_{k \geq 1}$ w.r.t. the Lebesgue measure.
We similarly regularise $t \mapsto c_t (x)$ into $t \mapsto c^k_t(x)$.
Since convolution in time does not affect $x$-regularity, $c^k$ still satisfies \ref{ass:SmoothSett}, with regularity constants independent of $k$.
$(\lambda^k,c^k)$ further satisfy \ref{ass:SmoothBound}.
Let $(X^k_s)_{0 \leq s \leq T}$ be the optimal process for $V_\Sigma^{\lambda^k} ( \overline\mu_0)$ obtained from the solution $\varphi^k$ of \eqref{eq:approxHJB}, and let us consider the feed-back control $\alpha^k := - \sigma^\top \nabla \varphi^k$. 
We recall that the number of atoms of the measure $\overline\lambda$ is at most countable.

\begin{lemma} \label{lem:boundPhi}
For $t=0$ and every $t \in (0,T]$ with $\overline\lambda ( \{ t \} ) = 0$,
\[ \varphi^k_t \xrightarrow[k \rightarrow +\infty]{} \overline\varphi_t, \qquad \nabla \varphi^k_t \xrightarrow[k \rightarrow +\infty]{} \nabla \overline\varphi_t, \]
uniformly on every compact set of $\R^d$. 
Moreover, the bounds in Proposition \ref{pro:hessbound} extend to $\nabla \overline \varphi$ and $\nabla^2 \overline \varphi$,
\[ \E [ \sup_{0 \leq t \leq T} \vert X^k_t - \overline{X}_t \vert ] \xrightarrow[k \rightarrow + \infty]{} 0, \]
and for every $1 \leq i \leq d$, $x \in \R^d$, and $t \in [0,T]$ with $\overline\lambda ( \{ t \} ) = 0$,
\begin{multline*}
\partial_i \overline\varphi_t (x) = \int_t^T \E \big\{ \partial_i b_s \cdot \nabla \overline\varphi_s - \frac{1}{2} \nabla \overline\varphi_s \cdot \partial_i a_s \nabla \overline\varphi_s + \frac{1}{2} \mathrm{Tr}[ \partial_i a_s \nabla^2 \overline\varphi_s] + \partial_i c_s \big\} (X^{t,x,\overline\alpha}_s) \d s \\ 
+ \int_{[t,T]} \E [ \partial_i \psi^k_s (X^{t,x,\overline\alpha}_s) ] \overline\lambda (\d s) + \overline{\lambda}_T \E [ \partial_i \psi_T (X^{t,x,\overline\alpha}_T) ]. 
\end{multline*}
\end{lemma}

\begin{proof}
For every $t \in [0,T]$, Proposition \ref{pro:hessbound} provides equi-continuity for $( \nabla \varphi^k_t )_{k \geq 1}$, and $\varphi^k$ has linear growth uniformly in $k$ from \eqref{eq:LinGrowth}.
The stability result \cite[Lemma 5.3]{ConstrainedSchrodinger} then provides the convergence of $\overline{\varphi}$ and $\nabla\overline\varphi$, and the bounds are inherited from Proposition \ref{pro:hessbound}.
Using this and our Lipschitz assumptions, a standard coupling argument (a variation of the proof of \cite[Chapter 5, Theorem 5.2]{friedman1975stochastic}) shows that  $\E [ \sup_{0 \leq t \leq T} \vert X^k_t - \overline{X}_t \vert ] \rightarrow 0$.

From \cite[Theorem 2.7]{billingsley2013convergence}, the measures $\lambda^k_s \d s + \overline{\lambda}_T \delta_T ( \d s )$ restricted to $[t,T]$ weakly converges towards the restriction of $\overline{\lambda}$ to $[t,T]$, for every $t$ with $\overline\lambda ( \{ t \} ) = 0$.
The equation on $\nabla \overline{\varphi}_t$ then results from taking the $k \rightarrow + \infty$ limit within \eqref{eq:ItoGrad}.
\end{proof}

Let us now show that $\overline\lambda$ can only have atoms at $0$ and $T$.

\begin{lemma} \label{lem:NoAtom}
$\overline{\lambda}$ has no atom in $(0,T)$.    
\end{lemma}

\begin{proof}
For every $t \in (0,T)$, we can write that $\overline{\lambda} = \overline{\lambda}_t \delta_t + \lambda'$, for $\overline{\lambda}_t \in \R$ and $\lambda' \in \M_+([0,T])$ with $\lambda'(\{ t \}) = 0$.
From the complementary slackness condition \eqref{eq:PAP2Compsl}, we can assume that $\E [ \psi_t ( \overline{X}_t ) ] = 0$. 
Let $(t_0,t'_0)$ be in $(0,T)^2$ with $t_0 < t < t_0'$. 
By induction, we define a sequence $(t_l,t'_l)_{l \geq 0}$ in $(0,T)^2$ such that $t_l < t < t'_l$ and 
\[ t_{l+1} = \inf \big\{ s \in [\tfrac{t_l+t}{2},t], \; \forall r \in [s,t], \, \E [ \psi_s ( \overline{X}_s ) ] \leq \E [ \psi_r ( \overline{X}_r ) ] \big\}, \]
\[t'_{l+1} = \sup \big\{ s' \in [t,\tfrac{t+t'_l}{2}], \; \forall r \in [t,s'], \, \E [ \psi_{s'} ( \overline{X}_{s'} ) ] \leq \E [ \psi_r ( \overline{X}_r ) ] \big\}. \]
The iterations are well-defined because $s \mapsto \E [ \psi_s (\overline{X}_s)]$ is a non-positive continuous function and $\E [ \psi_t (\overline{X}_t)] = 0$.
Let now $h^l : [t_l,t'_l] \rightarrow [0,1]$ be a $C^2$ function such that
\[ h^l_{t_l} = h^l_{t'_l} = 0, \quad \tfrac{\d}{\d s}\big\rvert_{s = t} \, h^l_{s} = 0, \quad \text{and} \quad \forall s \in [t_l,t'_l], \;\; \tfrac{\d^2}{\d s^2} h^l_{s} \leq 0. \] 
Thus, $h^l$ has a global maximum at $t$.
Let us choose $h^l_t = \lVert h^l \rVert_{\infty} \leq 2 \min [ (t - t_l)^2, (t'_l -t)^2 ]$ and such that
$\lVert \tfrac{\d}{\d s} h^l \rVert_\infty^2 \leq 2 \lVert h^l \rVert_{\infty}$.

Let now $\delta, C >0$ be given by Proposition \ref{pro:perturbound} for $\lambda ( \d s) = \lambda^k_s \d s + \overline{\lambda}_T \delta_T ( \d s)$: these constants do not depend on $(k,l)$ because they only depend on $\lambda$ through $\lambda( [0,T] )$, which is unaffected by the regularisation procedure.
From \eqref{eq:asspartialpsi}, $\E [ \partial_s \psi_s ( \overline{X}_s ) ] = 0$ for $s \in [0,T]$, hence
\begin{equation} \label{eq:Vanishpartial}
\int_{t_l}^{t'_l} \E [ \partial_s \psi_s (X^k_s) ] \frac{\d}{\d s} h^l_s \d s \xrightarrow[k \rightarrow + \infty]{} 0. 
\end{equation} 
Integrating by parts,
\begin{multline} \label{eq:cleverslice}
-\int_{t_l}^{t'_l} \frac{\d}{\d s} \E [ \psi_s (X^k_s) ] \frac{\d}{\d s} h^l_s \d s \\
= \int_{t_l}^t \{ \E [ \psi_s (X^k_s) ] - \E [ \psi_{t_l} (X^k_{t_l}) ] \} \frac{\d^2}{\d s^2} h^l_s \d s + \int_t^{t'_l} \{ \E [ \psi_s (X^k_s) ] - \E [ \psi_{t'_l} (X^k_{t'_l}]) \} \frac{\d^2}{\d s^2} h^l_s \d s,
\end{multline}
where we used that $\tfrac{\d}{\d s} \big\rvert_{s=t} h^l_{s} =0$.
Since $\tfrac{\d^2}{\d s^2} h^l_{s} \leq 0$, and $\E [ \psi_s (\overline{X}_s) ] - \E [ \psi_{t_l} (\overline{X}_{t_l}) ] \leq 0$, $\E [ \psi_s (\overline{X}_s) ] - \E [ \psi_{t'_l} (\overline{X}_{t'_l})] \leq 0$ by definition of $(t_l,t'_l)$, the above integrals are non-positive as $k \rightarrow +\infty$.
Plugging \eqref{eq:Vanishpartial}-\eqref{eq:cleverslice} into Proposition \ref{pro:perturbound} yields
\[ \int_{(t_l,t'_l)} h^l_s \, \E [ \lvert \sigma_s^\top \nabla \psi_s ( \overline{X}_s ) \rvert^2 ] \overline{\lambda}(\d s) \leq \limsup_{k \rightarrow +\infty} \int_{t_l}^{t'_l} h^l_s \, \E [ \lvert \sigma_s \nabla \psi_s ( X^{k}_s ) \rvert^2 ] \lambda^k_s \d s \leq \int_{t_l}^{t'_l} C \lvert h^l_s \rvert \d s,
\] 
where we used that $\lVert \lambda^k \rVert_{L^1(0,T)}$ is bounded uniformly in $k$.
Moreover, the constant $C$ does not depend on $l$.
From the decomposition $\overline{\lambda} = \overline{\lambda}_t \delta_t + {\lambda'}$, the l.h.s. is lower-bounded by $h^l_t \overline\lambda_t \E [ \lvert \sigma_t \nabla \psi_t ( \overline{X}_t ) \rvert^2 ]$. 
From the qualification condition \eqref{eq:smoothQualif} and Remark \ref{rem:Suffqualif}, $\E [ \lvert \sigma_t \nabla \psi_t ( \overline{X}_t ) \rvert^2 ] \neq 0$. 
Since $h^l_t = \lVert h^l \rVert_\infty$, dividing by $h^l_t$ and sending $l \rightarrow +\infty$ shows that $\overline\lambda_t = 0$.
We thus proved that $\overline\lambda(\{t\}) = 0$, for any $t$ in $(0,T)$.
\end{proof}

\begin{proposition}[Density] \label{pro:density}
The multiplier has the following decomposition
\[ \overline{\lambda}(\d t) = \overline{\lambda}_0 \delta_0 (\d t) + \overline{\lambda}_t \d t + \overline{\lambda}_T \delta_T (\d t). \]
Moreover, $\lVert \overline\lambda \rVert_{L^\infty(0,T)} \leq C$, for $C > 0$ independent of $\overline{\lambda}$ that only depends on the coefficients through the uniform norms of $\sigma$, $\sigma^{-1}$, $\partial_s b$, $\partial_s \sigma$, $\partial_s \psi$, $\nabla b$, $\nabla \sigma$, $\nabla c$, $\nabla \psi$, $\nabla \partial_s \psi$, $\nabla^2 \sigma$, $\nabla^2 \psi$, and either $\nabla^3 \psi$ or $\nabla^2 b$ and $\nabla^3 \sigma$.
\end{proposition}

We recall that we restricted ourselves to $\overline{\lambda}(\{0\}) = 0$ at the beginning of Section \ref{sec:Bounds}, without loss of generality.

\begin{proof}
Let us define the set   
\[ F := \{ t \in [0,T] \, , \, \E [ \psi_t ( \overline{X}_t ) ] =0 \}. \]
From the complementary slackness condition \eqref{eq:PAP2Compsl}, we only need to study the restriction of $\overline\lambda$ to $F$.
Since $t \mapsto  \E [ \psi_t ( \overline{X}_t ) ]$ is continuous, $F$ is a closed set, and $F$ is a (at most) countable union of disjoint intervals. 
Since $\overline\lambda$ has no atom in $(0,T)$ from Lemma \ref{lem:NoAtom}, we only need to look at intervals $[t_0,t'_0] \subset F$ with $t_0 < t'_0$.
Let us fix such a $[t_0,t'_0]$, and let $\delta_0, C >0$ be given by Proposition \ref{pro:perturbound}, so that
\[ \int_{t_1}^{t_1'} h_s \, \E [ \lvert \sigma^\top_s \nabla \psi^k_s ( X^k_s ) \rvert^2 ] \lambda^k_s \d s \leq \int_{t_1}^{t_1'} C [ h_s + \lvert \tfrac{\d}{\d s} h_s \rvert^2 ] - 2 \tfrac{\d}{\d s} h_s \big\{ \tfrac{\d}{\d s} \E [ \psi_s ( X^{k}_s ) ] - \E [ \partial_s \psi_s ( X^{k}_s ) ] \big\} \d s,
\]
for every $C^2$ function $h : [t_1,t_1'] \rightarrow [0,1]$ with $0 \leq t_1 \leq t_1' \leq T$, $h_{t_1} = h_{t_1'} =0$ and $\lVert h \rVert_{\infty} + \lVert \tfrac{\d}{\d s} h \rVert_{\infty}^2 \leq \delta_0$.

We first allow $C$ to depend on $\nabla^3 \psi$ in the setting of Proposition \ref{pro:perturbound}, so that $(C,\delta_0)$ does not depend on $k$, because $(\Vert \lambda^k \Vert_{L^1 (0,T)})_{ k \geq 1}$ is bounded.
Similarly to \eqref{eq:cleverslice},
\begin{multline*}
-\int_{t_1}^{t_1'} \frac{\d}{\d s} \E [ \psi_s (X^k_s) ] \frac{\d}{\d s} h_s \d s
= \E [ \psi_s (X^k_{t_1}) ] \frac{\d}{\d s} \big\vert_{s = t_-} h_s - \E [ \psi_s (X^k_{t_1'}) ] \frac{\d}{\d s} \big\vert_{s = t_1'} h_s  \\
+ \int_{t_1}^{t_1'} \E [ \psi_s (X^k_s) ] \frac{\d^2}{\d s^2} h_s \d s \xrightarrow[k \rightarrow +\infty]{} 0, 
\end{multline*}
if we assume that $[t_1',t_1'] \subset F$.
From the qualification condition \eqref{eq:smoothQualif} and Remark \ref{rem:Suffqualif}, the continuous function $s \mapsto \E [ \lvert \sigma^\top_s \nabla \psi_s ( \overline{X}_s ) \rvert^2 ]$ is positive on the compact set $F$ ($F$ is a closed subset of the compact $[0,T]$), so that there exists $\eta > 0$ such that
\[ \forall s \in F, \quad 0 < \eta \leq \E [ \lvert \sigma^\top_s \nabla \psi_s (\overline{X}_s ) \rvert^2 ]. \]
Taking the $k \rightarrow +\infty$ limit and using \eqref{eq:Vanishpartial}, as in the proof of Lemma \ref{lem:NoAtom},
\[ \eta \int_{[t_1,t_1']} h_s \overline{\lambda}({\d s} ) \leq  C \int_{t_1}^{t_1'} [ h_s + \lvert \tfrac{\d}{\d s} h_s \rvert^2 ] \d s, \]
for every $h$ as above with $[t_1,t_1'] \subset F$.
For $\delta \in (0, (t'_0 - t_0)/2)$ and $[t,t'] \subset [t_0-\delta,t'_0+\delta]$ with $t' - t < \delta$, we apply this estimate to $t_1 = t - (t'-t)$ and $t_1' = t + (t'-t)$, so that $[t,t'] \subset [t_1,t_1'] \subset [t_0,t'_0]$. 
We further choose $h_s = - \varepsilon_0 ( s - t_1) ( s - t_1')$, for $\varepsilon_0 \in (0,1)$ small enough so that $\lVert h \rVert_{\infty} + \lVert \tfrac{\d}{\d s} h \rVert_{\infty}^2 \leq \delta_0$.
In particular, $h_s \leq \varepsilon_0 [2 (t'- t)]^2$, and $h_s \geq \varepsilon_0 \vert t' - t \vert^2$ for $s \in [t,t']$. 
As a consequence,
\[ \varepsilon_0 \eta \vert t' - t \vert^2 \overline{\lambda} ([t,t']) \leq \eta \int_{[t,t']} h_s \overline{\lambda}( {\d s} ) \leq C \int_{t_1}^{t_1'} [ h_s + \lvert \tfrac{\d}{\d s} h_s \rvert^2 ] \d s \leq 5 \varepsilon_0 C [ 2 ( t' - t ) ]^3. \]
We deduce that 
\[ \forall \, 0 < \delta < \tfrac{t'_0 - t_0}{2}, \; \forall \, [t,t'] \subset [t_0-\delta,t'_0+\delta], \quad t' - t < \delta \Rightarrow \overline{\lambda} ([t,t']) \leq 40 \eta^{-1} C \vert t' - t \vert. \]
The cumulative distribution function of $\overline{\lambda}$ is thus absolutely continuous on every compact subset of $(t_0,t_0')$, with a.s. derivative bounded by $40 \eta^{-1} C $.
This shows that $\overline{\lambda}$ has a density bounded by $40 \eta^{-1} C $ on $(t_0,t_0)$. Since $\overline{\lambda}(\{0\}) = 0$, Lemma \ref{lem:NoAtom} extends the result to  $[t_0,t_0')$, and then to $[t_0,t_0']$ if $t'_0 \neq T$.
Since $(\eta,C)$ does not depend on $[t_0,t_0']$, the result holds on the whole set $F \cap (0,T)$, and thus on $(0,T)$.

To refine the bound, we now allow for $C = C_\varepsilon + \varepsilon \sup_{0 \leq t \leq T} \lambda^k_s$ when applying Proposition \ref{pro:perturbound}.
Since we now know that $\overline{\lambda}$ restricted to $(0,T)$ has a bounded density $\overline{\lambda} \vert_{(0,T)} \in L^\infty (0,T)$, and the continuous function $\lambda^k$ was obtained from $\overline{\lambda} - \overline{\lambda}_T \delta_T$ by convolution, we can bound $\sup_{0 \leq s \leq T} \lambda^k_s$ by $\lVert \overline{\lambda} \vert_{(0,T)} \rVert_{L^\infty(0,T)}$. 
Since there is no more dependence on $k$ in the bounds, we then perform the same computations as above to obtain that
\[ \lVert \overline{\lambda} \vert_{(0,T)} \rVert_{L^\infty(0,T)} \leq 40 \eta^{-1} [ C_\varepsilon + \varepsilon \lVert \overline{\lambda} \vert_{(0,T)} \rVert_{L^\infty(0,T)} ].\]
By choosing $\varepsilon$ small enough beforehand, this yields the desired bound on $\lVert \overline{\lambda} \vert_{(0,T)} \rVert_{L^\infty(0,T)}$. 
\end{proof}

\section{Proof of the main results} \label{sec:proofs}

We now turn to the proofs of the results stated in Section \ref{sec:PAP2results}. 
The proof of Theorem \ref{thm:opti} essentially relies on results from \cite{ConstrainedSchrodinger}, see Section \ref{subsec:optimality} below. 
For Theorem \ref{thm:density}, a smoothing procedure is detailed in Section \ref{subsec:density} to enter the framework of Section \ref{sec:Bounds}. 
Theorem \ref{thm:stability} is proved independently in Section \ref{subsec:stability}.
We recall the convenient notation $a := \sigma \sigma^\top$.

\subsection{Optimality conditions} \label{subsec:optimality}

This section is devoted to the proof of Theorem \ref{thm:opti}.
Let $\overline\mu_{[0,T]}$ be an optimiser for \eqref{eq:mfMinProb} that satisfies \ref{ass:consquali}.
We first rewrite \eqref{eq:mfMinProb} to make it enter the framework of \cite{ConstrainedSchrodinger}.

\begin{lemma} \label{lem:NLRate}
For any $\mu_{[0,T]}$ in $\ps_1(C([0,T],\R^d))$ such that $H( \mu_{[0,T]} \lvert \Gamma( \mu_{[0,T]} ) )$, $H( \mu_{[0,T]} \lvert \Gamma(\overline\mu_{[0,T]}) )$ are finite, we define
\[ \mathcal{F}
(\mu_{[0,T]}) := H( \mu_{[0,T]} \lvert \Gamma( \mu_{[0,T]} ) ) - H( \mu_{[0,T]} \lvert \Gamma(\overline\mu_{[0,T]}) ). \]
$\F$ is differentiable at $\overline\mu_{[0,T]}$, its linear functional derivative
being given by 
\begin{equation} \label{eq:mfDerivF}
\frac{\delta\F}{\delta\mu}(\overline\mu_{[0,T]},x_{[0,T]}) = - \E_{\overline\mu_{[0,T]}} \int_0^T a^{-1}_t ( {\sf X}_t ) \tfrac{\delta b_t}{\delta\mu} ({\sf X}_t, \overline\mu_t,x_t) [\d {\sf X}_t - b_t ( {\sf X}_t, \overline\mu_t) \d t].
\end{equation}
Moreover, there exists a measurable $ \overline{c} : [0,T] \times \R^d \rightarrow \R$ such that
\[ \frac{\delta\F}{\delta\mu}(\overline\mu_{[0,T]},x_{[0,T]}) = \int_0^T \int_{\R^d} \overline{c}_t (x_t) \overline\mu_t ( \d x ) \d t, \]
and $\overline{c}_t$ is globally Lipschitz uniformly in $t$.
\end{lemma}

We recall that $({\sf X}_t)_{0 \leq t \leq T}$ denotes the canonical process on $C([0,T],\R^d)$, see Section \ref{subsec:Notations}.

\begin{proof}
The finiteness of the entropy allows us to write
\[ H( \mu_{[0,T]} \lvert \Gamma( \mu_{[0,T]} )) = H( \mu_{[0,T]} \lvert \Gamma(\overline\mu_{[0,T]}) ) - \int_{C([0,T],\R^d)} \log \frac{\d \Gamma( \mu_{[0,T]} )}{\d  \Gamma(\overline\mu_{[0,T]})} \d \mu_{[0,T]}. \]
Since $\sigma^{-1}$ is bounded and $b$ is globally Lipschitz in $\mu$ from \ref{ass:coefReg1}, the Novikov condition 
\[ \E_{\Gamma( \overline\mu_{[0,T]} )} \, \exp \bigg[ \frac{1}{2} \int_0^T \lvert \sigma^{-1}_t ({\sf X}_t ) [ b_t ({\sf X}_t, \mu_t) - b_t ({\sf X}_t, \overline\mu_t) ] \rvert^2 \d t \bigg] < +\infty,  \] 
is satisfied, and the Girsanov transform \cite[Theorem 2.3]{leonard2012girsanov} gives that
\begin{multline*} 
\frac{\d \Gamma( \mu_{[0,T]} )}{\d \Gamma(\overline\mu_{[0,T]} )} = \exp \bigg[ \int_0^T \sigma^{-1}_t ( {\sf X}_t ) [ b_t ({\sf X}_t, \mu_t) - b_t ({\sf X}_t, \overline\mu_t) ] \d {\sf B}_t \\
- \frac{1}{2} \int_0^T \lvert \sigma^{-1}_t ( {\sf X}_t ) [ b_t ({\sf X}_t, \mu_t) - b_t ({\sf X}_t, \overline\mu_t) ] \rvert^2 \d t \bigg],
\end{multline*}
where $({\sf B}_t)_{0 \leq t \leq T}$ is a Brownian motion under $\Gamma(\overline\mu_{[0,T]})$, and
$\Gamma(\overline\mu_{[0,T]})$-almost surely,
\begin{equation*} \label{eq:PathGammabar}
\d {\sf X}_t =  b_t({\sf X}_t,\overline\mu_t) \d t + \sigma_t ({\sf X}_t) \d {\sf B}_t.
\end{equation*}
Using this to rewrite $\d {\sf B}_t$,
\begin{multline*} \label{eq:mfF}
\F( \mu_{[0,T]} ) = \frac{1}{2} \int_0^T \int_{\R^d} \lvert \sigma^{-1}_t (x) [ b_t (x, \mu_t) - b_t (x, \overline\mu_t) ] \rvert^2 \mu_t ( \d x ) \d t \\
- \E_{\mu_{[0,T]}} \int_0^T a^{-1}_t ( {\sf X}_t ) [b_t ({\sf X}_t, \mu_t) - b_t ({\sf X}_t, \overline\mu_t)] [\d {\sf X}_t - b_t ( {\sf X}_t, \overline\mu_t) \d t], 
\end{multline*}
and it is direct that $\F( \overline\mu_{[0,T]} ) = 0$.
For any $\mu_{[0,T]}$ with $H( \mu_{[0,T]} \lvert \Gamma(\overline\mu_{[0,T]}) ) < + \infty$, the differentiability of $b$ yields 
\begin{multline*}
\frac{\d}{\d \varepsilon}\bigg\vert_{\varepsilon=0} \F( (1-\varepsilon)\overline\mu_{[0,T]} + \varepsilon \mu_{[0,T]}) = \\
- \E_{\overline\mu_{[0,T]}} \int_0^T a^{-1}_t ( {\sf X}_t ) \int_{\R^d} \tfrac{\delta b_t}{\delta\mu} ({\sf X}_t, \overline\mu_t, x) [ \mu_t ( \d x ) - \overline\mu_t ( \d x) ] [\d {\sf X}_t - b_t ( {\sf X}_t, \overline\mu_t) \d t].
\end{multline*}
Since $\int_{\R^d} \tfrac{\delta b_t}{\delta\mu} ({\sf X}_t, \overline\mu_t,x) \overline{\mu}_t( \d x) = 0$ by definition of $\tfrac{\delta b}{\delta \mu}$,
this proves that $\F$ admits a linear functional derivative at $\overline\mu_{[0,T]}$ given by \eqref{eq:mfDerivF}, in the sense of Definition \ref{def:PAP2DiFF}. 
Since $H( \overline\mu_{[0,T]} \lvert \Gamma(\overline\mu_{[0,T]}) )$ is finite, \cite[Theorem 2.1]{leonard2012girsanov} provides an adapted process $(c_t)_{0\leq t\leq T}$ on the canonical space $\Omega = C([0,T],\R^d)$ such that
\[
\E_{\overline\mu_{[0,T]}} \int_0^T \lvert \sigma^\top_t ({\sf X}_t) {c}_t \rvert^2 \d t < +\infty, \]
where 
\begin{equation} \label{eq:defct}
\d {\sf X}_t = b_t ( {\sf X}_t, \overline\mu_t) \d t + a_t({\sf X}_t) {c}_t \d t + \sigma_t ( {\sf X}_t ) \d {\sf B}_t, \quad \overline\mu_{[0,T]}\text{-a.s.},
\end{equation} 
the process $( {\sf B}_t )_{0\leq t\leq T}$ being a Brownian motion under $\overline\mu_{[0,T]}$.
For $x$ in $\R^d$, we then define 
\begin{equation} \label{eq:defcbart}  \overline{c}_t (x) := - \E_{\overline\mu_{[0,T]}} {c}_t \cdot \frac{\delta b_t}{\delta\mu} ({\sf X}_t, \overline\mu_t,x).  
\end{equation} 
From \ref{ass:linDiff}, $\overline{c}_t$ is globally Lipschitz in $x$ uniformly in $t$.
Plugging this into \eqref{eq:mfDerivF} concludes. 
\end{proof}

\begin{proof}[Proof of Theorem \ref{thm:opti}-\ref{item:OptCond}]
Using Lemma \ref{lem:NLRate},
\[ \overline{H}_\Psi = \inf_{\substack{\mu_{[0,T]} \in \ps_1 (C([0,T],\R^d)) \\ \A t \in [0,T], \; \Psi(\mu_t) \leq 0}} H( \mu_{[0,T]} \lvert \Gamma(\overline\mu_{[0,T]}) ) + \mathcal{F}(\mu_{[0,T]}). \]
This setting satisfies the assumptions of Theorem \cite[Theorem 2.12]{ConstrainedSchrodinger} with $E = C([0,T],\R^d)$, $\phi(x_{[0,T]})=\sup_{t \in [0,T]} \vert x_t \vert$, $\mathcal{F}$ as above, the inequality constraints $\Psi_t : \mu_{[0,T]} \mapsto \Psi ( \mu_t )$ with $\mathcal{T} = [0,T]$, and no equality constraint.
Indeed, $\mu_{[0,T]} \mapsto H ( \mu_{[0,T]} \vert \Gamma (  \mu_{[0,T]} ) )$ is lower semi-continuous, so that \cite[Assumptions (A1)-(A2)-(A3)]{ConstrainedSchrodinger} are satisfied.
From Corollary \ref{cor:finiteRate}, the domain of $\F$ is convex and $\F ( \mu_{[0,T]}) < +\infty$ implies $\F ( \mu'_{[0,T]}) < +\infty$ for any $\mu'_{[0,T]}$ that has a bounded density w.r.t. $\mu_{[0,T]}$.
Using Lemma \ref{lem:NLRate}, this implies \cite[Assumptions (A6)-(A8)]{ConstrainedSchrodinger}.
The regularity assumptions \ref{ass:coefReg1}-\ref{ass:linDiff} and the qualification condition \ref{ass:consquali} eventually imply
\cite[Assumption (B3)]{ConstrainedSchrodinger}, and thus \cite[Assumptions (A7)]{ConstrainedSchrodinger}. 
As a consequence, we obtain a Lagrange multiplier $\overline{\lambda} \in \M_+([0,T])$ such that 
\[ \frac{\d \overline{\mu}_{[0,T]}}{\d \Gamma(\overline{\mu}_{[0,T]})}(x_{[0,T]}) = \overline{Z}^{-1} \exp \bigg[ -\int_0^T \overline{c}_s (x_s) \d s - \int_{0}^T \frac{\delta\Psi}{\delta\mu}(\overline\mu_s,x_s) \overline\lambda (\d s) \bigg], \]
together with the complementary slackness condition \eqref{eq:CompSlackmf}.

By convolution with a smooth kernel with compact support, we regularise 
$(t,x) \mapsto b_t ( x, \overline{\mu}_t )$, $\overline{c}_t(x)$, $\sigma_t (x)$, $\tfrac{\delta \Psi}{\delta\mu}(\overline\mu_t,x)$ into $C^\infty$ functions $b^k$, $c^k$, $\sigma^k$, $\psi^k$.
We similarly obtain a weak approximation of $\overline\lambda$ given by a sequence $( \lambda^k )_{k \geq 1}$ of functions in $C^1([0,T],\R_+)$. 
Since we assumed \ref{ass:coefReg1}-\ref{ass:linDiff} and convolution preserves already existing regularity properties, $b^k$, $c^k$, $\sigma^k$, $\psi^k$ satisfy \ref{ass:coefReg1}-\ref{ass:linDiff} for regularity constants independent of $k$.
We thus enter the setting of Section \ref{sec:Bounds}, and let $\varphi^k$ denote the solution of the corresponding HJB equation \eqref{eq:approxHJB} with regularised coefficients.
For every $t \in [0,T]$, Proposition \ref{pro:hessbound} shows that $( \nabla \varphi^k_t )_{k \geq 1}$ and $( \nabla^2 \varphi^k_t )_{k \geq 1}$ are uniformly bounded. 
Hence, from the Arzelà-Ascoli theorem, $( \nabla \varphi^k_t )_{k \geq 1}$ is pre-compact on any compact set.
Using \cite[Lemma 5.4]{ConstrainedSchrodinger}, \cite[Lemma 5.3]{ConstrainedSchrodinger} now shows that for Lebesgue-a.e. $t \in [0,T]$, $\varphi^k_t$ converges towards $\overline\varphi_t$ on any compact set, where $\overline\varphi$ is the mild solution of \eqref{eq:coupled_FP_HJB} in the sense of Definition \ref{def:mildForm}.  
Moreover, $\overline{\mu}_{[0,T]}$ is the law of the solution to
\[ \d \overline{X}_t = b_t ( \overline{X}_t ) \d t - a_t \nabla \varphi_t ( \overline{X}_t ) \d t + \sigma_t ( \overline{X}_t ) \d B_t, \quad \overline{X}_0 \sim \overline{Z}^{-1} e^{-\overline{\varphi}_0(x)} \nu_0(\d x), \]
which we now identify to \eqref{eq:defct}. 
This yields $c_t = -\nabla \overline\varphi_t (\overline{X}_t)$ a.s., giving the desired expression for $\overline{c}_t$ using \eqref{eq:defcbart}.
\end{proof}

From \cite[Remark 4.4]{ConstrainedSchrodinger}, we further have
\begin{equation} \label{eq:LambdaMass}
\overline{\lambda}([0,T]) \leq \tilde\varepsilon \eta^{-1} \int_{\ps_1(C([0,T],\R^d)} \bigg[ \log \frac{\d \overline{\mu}_{[0,T]}}{\d \Gamma(\overline{\mu}_{[0,T]})} + \frac{\delta\F}{\delta\mu} (\overline{\mu}_{[0,T]}) \bigg] \d (\tilde{\mu}_{[0,T]}-\overline{\mu}_{[0,T]}),
\end{equation}
where $\eta := -\sup_{0 \leq t \leq T} \Psi (\overline{\mu}_t) + \int \tfrac{\delta\Psi}{\delta\mu} (\overline\mu_t) \d \tilde\mu_t$.

\begin{proof}[Proof of Theorem \ref{thm:opti}-\ref{item:Control}]
From Lemma \ref{lem:repEntr}, 
\[ \overline{H}_\Psi = \inf_\Sigma \overline{V}_\Psi. \]
From the above proof of Theorem \ref{thm:opti}-\ref{item:OptCond}, this infimum is realised on any reference system $\Sigma$ by the initial law $\overline\mu_0$ and the feed-back control $\overline{\alpha}_t = - a_t \nabla \overline\varphi_t$.
Consequently, $\overline{H}_\Psi = \overline{V}_\Psi = H(\overline\mu_0 \vert \nu_0) + V_\Sigma (\overline\mu_0)$.
\end{proof}

\subsection{Existence of a density for the Lagrange multiplier} \label{subsec:density}

Under \ref{ass:coefReg2} and \ref{ass:ini2}-\ref{item:inih1}, let $\overline\mu_{[0,T]}$ be an optimiser for \eqref{eq:mfMinProb} that satisfies \ref{ass:consquali}.
We are going to smooth Problem \eqref{eq:mfMinProb} to make it enter the framework of Section \ref{sec:Bounds}.
Let $\rho : \R^d \rightarrow [0,1]$ be a $C^\infty$ symmetric function with compact support and $\int_{\R^d} \rho (x) \d x = 1$.
For $k \geq 1$, we define $\rho^k(x) := k^d \rho (k x)$, and the regularised coefficients
\[ b_t^k (x,\mu) :=  [ b_t^k (\cdot, \rho^k \ast \mu) ] \ast \rho^k (x), \quad \sigma^k_t (x) := \sigma^k_t \ast \rho^k (x), \quad \Psi^k ( \mu ) := \Psi ( \rho^k \ast \mu ). \]
The convolution against vectors or matrices is done component-wise.
This mollifying procedure of the measure argument well behaves w.r.t. linear derivatives, as recalled in Lemma \ref{lem:mollif}.
Since the coefficients were already Lipschitz-continuous, $b^k$, $\sigma^k$, $\Psi^k$, $\tfrac{\delta}{\delta\mu} b^k$  and $\tfrac{\delta}{\delta\mu} \Psi^k$ converge towards $b$, $\sigma$, $\Psi$, $\tfrac{\delta}{\delta\mu} b$ and $\tfrac{\delta}{\delta\mu} \Psi$ uniformly on $[0,T] \times \R^d \times \ps_1 ( \R^d)$.
Moreover, since convolution preserves already existing regularity properties, \ref{ass:coefReg2} holds for $b^k$, $\sigma^k$, $\tfrac{\delta}{\delta\mu} b^k$ and $\tfrac{\delta}{\delta\mu} \Psi^k$ with regularity bounds independent of $k$.
For $\mu_{[0,T]}$ in $\ps_1 (C([0,T],\R^d))$, let $\Gamma^k (\mu_{[0,T]})$ denote the path-law of the pathwise unique strong solution to 
\[ \d Y^k_t = b^k_t (Y^k_t,\mu_t) \d t + \sigma^k_t ( Y^k_t ) \d B_t, \quad Y^k_0 \sim \nu_0, \]
where $( B_t )_{0 \leq t \leq T}$ is a Brownian motion. 
From Theorem \ref{thm:opti}, on
the canonical space $\Omega = C([0,T],\R^d)$, there exists an essentially bounded process $( \overline\alpha_t )_{0 \leq t \leq T}$ such that the canonical process satisfies 
\[ \d {\sf X}_t = b_t ( {\sf X}_t, \overline \mu_t) \d t + \sigma_t ( {\sf X}_t ) \overline\alpha_t \d t + \sigma_t ( {\sf X}_t) \d {\sf B}_t, \quad \overline\mu_{[0,T]}\text{-a.s.}, \]
where $( {\sf B}_t )_{0 \leq t \leq T}$ is a Brownian motion under $\overline\mu_{[0,T]}$.
In our Lipschitz setting, $\overline\alpha$ being essentially bounded, the following McKean-Vlasov equation
\[ d X^{k}_t = b^k_t ( X^{k}_t, \mathcal{L}(X^{k}_t)) \d t + \sigma^k_t ( X^{k}_t ) \overline\alpha_t \d t + \sigma^k_t ( X^{k}_t ) \d {\sf B}_t, \quad X^{k}_0 = {\sf X}_0, \]
has a pathwise-unique strong solution, and let $\mu^k_{[0,T]}$ denote its law. 
Using the Girsanov transform,
\begin{equation} \label{eq:CompetEntrop}
H( \mu^k_{[0,T]} \vert \Gamma ( \mu^k_{[0,T]} ) ) = H( \overline\mu_{[0,T]} \vert \Gamma(\overline\mu_{[0,T]} ) ) = H( \overline\mu_0 \vert \nu_0) + \E_{\overline\mu_{[0,T]}} \int_0^T \frac{1}{2} \vert \overline\alpha_t \vert^2 \d t. 
\end{equation} 
Using the uniform convergence of coefficients,
a standard coupling argument yields
\begin{equation} \label{eq:CVCompet}
\E_{\overline\mu_{[0,T]}} \big[ \sup_{0 \leq t \leq T} \vert {\sf X}_t - X^{k}_t \vert \big] \xrightarrow[k \rightarrow +\infty]{} 0. 
\end{equation}
As a consequence, $\Psi^k (\mu^k_t )$ converges towards $\Psi (\overline\mu_t )$, and we can assume that
\begin{equation} \label{eq:AdmCompet} 
\forall k \geq 1, \, \forall t \in [0,T], \quad \Psi^k ( \mu^k_t ) \leq \Psi ( \overline\mu_t ) \leq 0, 
\end{equation}
up to subtracting a small constant to $\Psi^k$ that vanishes as $k \rightarrow + \infty$.
Let now $( \phi_i )_{i \geq 1}$ be a countable set of $C^\infty$ functions $\R^d \rightarrow \R$ that separates measures, with $\lVert \nabla^j \phi_i \rVert_\infty \leq 1$ for any $j \geq 0$. 
We further define 
\[ \forall \mu_{[0,T]} \in \ps_1 ( C([0,T],\R^d) ), \quad {\Phi}( \mu_{[0,T]} ) := \sum_{i \geq 1} 2^{-i-1} \int_0^T \bigg( \int_{\R^d} \phi_i \, \d ( \mu_t - \overline\mu_t ) \bigg)^2 \d t.  \]
This positive term only vanishes when $(\mu_t)_{0 \leq t \leq T} = (\overline\mu_t)_{0 \leq t \leq T}$, hence 
\[ \overline{H}_\Psi = \inf_{\substack{\mu_{[0,T]} \in \ps_1( C ( [0,T], \R^d ) ) \\ \forall t \in [0,T], \; \Psi ( \mu_t ) \leq 0}} H(\mu_{[0,T]} \vert \Gamma(\mu_{[0,T]} )) + \Phi(\mu_{[0,T]}), \]
and the r.h.s. is uniquely realised by $\overline\mu_{[0,T]}$.
Indeed, $\Gamma(\overline\mu_{[0,T]} )$ only depends on $(\overline\mu_t)_{0 \leq t \leq T}$, and there is only one path-measure with given time-marginals and minimal entropy (because $H$ is strictly convex).
We now approximate this problem.

\begin{lemma}[Non-linear approximation] \label{lem:NLApprox}
For large enough $k \geq 1$,
\[ \overline{H}^k_\Psi := \inf_{\substack{\mu_{[0,T]} \in \ps_1( C ( [0,T], \R^d ) ) \\ \forall t \in [0,T], \; \Psi^k ( \mu_t ) \leq 0}} H(\mu_{[0,T]} \vert \Gamma^k(\mu_{[0,T]} )) + \Phi(\mu_{[0,T]}), \] 
is realised by at least one measure $\overline{\mu}^k_{[0,T]}$ at which the qualification constraint \ref{ass:consquali} holds.
Moreover, $( \overline{\mu}^k_{[0,T]} )_{k \geq 1}$ converges towards $\overline{\mu}_{[0,T]}$ in $\ps_1 ( C ( [0,T] , \R^d ))$.
\end{lemma}

\begin{proof}
We first prove existence and convergence, before to show constraint qualification.

\medskip

\emph{\textbf{Step 1.} Convergence of minimisers.}
From \eqref{eq:CompetEntrop}-\eqref{eq:AdmCompet}, the competitor $\mu^k_{[0,T]}$ that we previously built shows that $\overline{H}^k_\Psi$ is finite and bounded uniformly in $k$.
Since $\Phi$ is non-negative continuous and $\mu_{[0,T]} \mapsto H( \mu_{[0,T]} \vert \Gamma^k ( \mu_{[0,T]} ) )$ has compact level sets in $\ps_1 ( C ( [0,T] , \R^d ))$ (see e.g. \cite[Corollary B.4]{chaintronLDP}), existence holds for at least one minimiser $\overline{\mu}^k_{[0,T]}$.  
Since $H( \overline{\mu}^k_{[0,T]} \vert \Gamma^k ( \overline{\mu}^k_{[0,T]} ) )$ is finite, there exists a square-integrable process $(\overline{\alpha}^k_t)_{0 \leq t \leq T}$ on the canonical space such that
\[ \d {\sf X}_t = b^k_t ( {\sf X}_t, \overline\mu^k_t) \d t + \sigma^k_t ( {\sf X}_t ) \overline\alpha^k_t \d t + \sigma^k_t ( {\sf X}_t) \d {\sf B}^k_t, \quad \overline{\mu}^k_{[0,T]}\text{-a.s.}, \] 
where $( {\sf B}^k_t )_{0 \leq t \leq T}$ is a Brownian motion under $\overline\mu^k_{[0,T]}$ (this argument was detailed in the proof of Lemma \ref{lem:NLRate}).
This is a McKean-Vlasov equation whose coefficients are Lipschitz uniformly in $k$.
Since $\E [ \int_0^T \vert \overline\alpha^k_t \vert^2 \d t ]$ and $\sigma^k_t$ are furthermore uniformly bounded, a standard Gronwall argument yields
\[ \sup_{k \geq 1} \E_{\overline\mu^k_{[0,T]}} \big[ \sup_{0 \leq t \leq T} \vert {\sf X}_t \vert \big] < + \infty, \]
hence the Lipschitz function $x \mapsto b^k_t (x, \overline\mu^k_t)$ has linear growth uniformly in $k$.
Using \cite[Lemma B.2 and Remark B.3]{chaintronLDP}, we deduce that $( \overline\mu^k_{[0,T]} )_{k \geq 1}$ is pre-compact in $\ps_1 (C([0,T],\R^d))$.
Up to re-indexing, we can assume that it converges towards some $\overline\mu^\infty_{[0,T]}$.
Using a coupling argument as for \eqref{eq:CVCompet}, we get that $\Gamma^k ( \overline\mu^k_{[0,T]} )$ converges towards $\Gamma ( \overline\mu^\infty_{[0,T]} )$ in $\ps_1 ( C ( [0,T] , \R^d ))$. 
Moreover, by optimality of $\overline\mu^k_{[0,T]}$,
\[ \forall k \geq 1, \quad H ( \overline\mu^k_{[0,T]} \vert \Gamma^k ( \overline\mu^k_{[0,T]} ) ) + \Phi ( \overline\mu^k_{[0,T]} ) \leq H ( \mu^k_{[0,T]} \vert \Gamma^k ( \mu^k_{[0,T]} ) ) + \Phi ( \mu^k_{[0,T]} ). \]
From \eqref{eq:CompetEntrop}-\eqref{eq:CVCompet}, the r.h.s. converges to $\overline{H}_\Psi$ as $k \rightarrow + \infty$.
Since $(\mu_{[0,T]},\nu_{[0,T]}) \mapsto H( \mu_{[0,T]} \vert \nu_{[0,T]} )$ is lower semi-continuous, we get
\[ H ( \overline\mu^\infty_{[0,T]} \vert \Gamma ( \overline\mu^\infty_{[0,T]} ) ) + \Phi ( \overline\mu^\infty_{[0,T]} ) \leq \overline{H}_\Psi, \]
and using the uniform convergence of $\Psi^k$:
\[ \forall t \in [0,T], \quad \Psi^k( \overline\mu^\infty_t ) \leq 0. \]
By uniqueness for the problem defining $\overline{H}_\Psi$, we get $\overline\mu^\infty_{[0,T]} = \overline\mu_{[0,T]}$.
By uniqueness of limit points, the full sequence $( \overline\mu^k_{[0,T]} )_{k \geq 1}$ thus converges towards $\overline\mu_{[0,T]}$.

\medskip

\emph{\textbf{Step 2.} Constraint qualification.}
From \ref{ass:consquali}, we have existence of $\tilde\mu_{[0,T]}$ with finite $H(\tilde{\mu}_{[0,T]} \vert \Gamma(\overline\mu_{[0,T]}))$ such that  
\[ \forall t \in [0,T], \quad \Psi (\overline\mu_t) + \tilde\varepsilon \int_{\R^d} \frac{\delta \Psi}{\delta\mu} (\overline\mu_t) \d [ \tilde\mu_t - \overline\mu_t ] < 0. \]
Reasoning as we did to build $\mu^k_{[0,T]}$, we can produce $\tilde\mu^k_{[0,T]}$ with $\sup_{k \geq 1}  H(\tilde{\mu}^k_{[0,T]} \vert \Gamma^k (\overline\mu^k_{[0,T]})) < +\infty$ such that $( \tilde\mu^k_{[0,T]} )_{k \geq 1}$ converges towards $\tilde\mu_{[0,T]}$ in $\ps_1 (C([0,T],\R^d))$. 
The uniform convergence of $\Psi^k$ and $\frac{\delta}{\delta\mu} \Psi^k$ then implies that 
\begin{equation} \label{eq:unifqualif}
\sup_{k \geq 1}  H(\tilde{\mu}^k_{[0,T]} \vert \Gamma^k (\overline\mu^k_{[0,T]})) < +\infty, \qquad \sup_{k \geq k_0} \sup_{t \in [0,T]} \Psi^k (\overline\mu^k_t) + \tilde\varepsilon \int_{\R^d} \frac{\delta \Psi^k}{\delta\mu} (\overline\mu^k_t) \d \tilde\mu^k_t < 0,
\end{equation}
for a large enough $k_0$, hence constraint qualification holds at $\overline\mu^k_{[0,T]}$.
\end{proof}

We now set $\psi^k_t (x) := \Psi^k ( \overline{\mu}^k_t ) + \tfrac{\delta\Psi^k}{\delta\mu}(\overline\mu^k_t,x)$ and $\overline\psi_t (x) := \Psi ( \overline{\mu}_t ) + \tfrac{\delta\Psi}{\delta\mu}(\overline\mu_t,x)$. 
We then linearise the problem defining $\overline{H}^k_\Psi$ at $\overline\mu^k_{[0,T]}$ using Lemma \ref{lem:NLRate} and \cite[Lemma 4.3]{ConstrainedSchrodinger}.
We obtain that $\overline\mu^k_{[0,T]}$ is the unique minimiser of the strictly convex problem
\begin{equation} \label{eq:linprob}
\inf_{\substack{\mu_{[0,T]} \in \ps_1( C ( [0,T], \R^d ) ) \\ \forall t \in [0,T], \; \int_{\R^d} {\psi}^k_t \d \mu_t \leq 0}} H ( \mu_{[0,T]} \vert \Gamma^k ( \overline\mu^k_{[0,T]} ) ) + \int_0^T \int_{\R^d} c^k_t \d \mu_t \d t, 
\end{equation} 
where
\[ c^k_t (x) := \overline{c}^k_t (x) + \Phi^k_t(x), \qquad \Phi^k_t(x) := \sum_{i \geq 1} 2^{-i} \bigg( \int_{\R^d} \phi_i \, \d (\overline\mu^k_t - \overline\mu_t) \bigg) \phi_i (x). \]
The $\overline{c}^k_t$ term is produced using Lemma \ref{lem:NLRate} applied to $\overline\mu^k_{[0,T]}$.
The second term results from computing the linear derivative of $\Phi$.
Noticing that $\Psi^k (\overline\mu^k_t) = \int_{\R^d} \psi^k_t \d \overline\mu^k_t$, differentiating yields that $\int_{\R^d} \partial_t \psi^k_t \d \overline\mu^k_t = 0$, as required by \eqref{eq:asspartialpsi}.
Since our regularised coefficients satisfy \ref{ass:SmoothSett},
Problem \eqref{eq:linprob} now enters the scope of Section \ref{sec:Bounds} for the drift coefficient $b^k_t(x) := b^k_t (x,\overline\mu^k_t)$.
Let $(\overline{\lambda}^k, \overline{\varphi}^k)$ be related to $\overline{\mu}^k_{[0,T]}$ as in \eqref{eq:LinDensity}: the path-density reads
\begin{equation} \label{eq:densityk}
\frac{\d \overline\mu^k_{[0,T]}}{\d \Gamma^k ( \overline\mu^k_{[0,T]})} ( x_{[0,T]} ) = \overline{Z}_k^{-1} \exp \bigg[ -\int_0^T c^k_t (x_t) \d t - \int_{0}^T \psi^k_t (x_t) \overline\lambda^k (\d t) \bigg], 
\end{equation}
and $\overline{\varphi}^k$ is the solution to
\[ -\overline\varphi^k_t + \int_t^T b_s^k \cdot \nabla \overline\varphi^k_s - \frac{1}{2} \big\lvert ( \sigma^k_s )^\top \nabla\overline\varphi^k_s \big\rvert^2 + c^k_s  + \frac{1}{2} \mathrm{Tr}[ \sigma^k_s ( \sigma^k_s )^\top \nabla^2 \overline\varphi^k_s] \d s
+ \int_{[t,T]} \psi^k_s \, \overline{\lambda}^k ( \d s) = 0, \]
in the sense of Definition \ref{def:mildForm}. 
Moreover, reasoning as in the proof of Theorem \ref{thm:opti}-\ref{item:OptCond},
\begin{equation} \label{eq:NLclosure}
\overline{c}^k_s (x) = \int_{\R^d} \frac{\delta b^k_s}{\delta\mu} ( y, \overline\mu^k_s, x ) \cdot \nabla \overline{\varphi}^k_s \overline{\mu}^k_s ( \d y).
\end{equation}
Let us further define $\alpha^k_t := -( \sigma^k_t )^\top \nabla \overline\varphi^k_t$, $a^k_t := \sigma^k_t ( \sigma^k_t )^\top$ and
\begin{equation} \label{eq:InfGen}
\forall \phi \in C^2 ( \R^d), \quad L^{\alpha^k}_t \phi := b^k_t \cdot \nabla \phi + \sigma^k_t \alpha^k_t \cdot \nabla \phi + \frac{1}{2} \mathrm{Tr}[ a^k_t \nabla^2 \phi ],
\end{equation} 
the operator $L^{\alpha_k}_t$ being the infinitesimal generator of the process with path-law $\overline\mu^k_{[0,T]}$.

\begin{proof}[Proof of Theorem \ref{thm:density}]
Since $\overline{\mu}^k_{[0,T]}$ converges to $\overline{\mu}_{[0,T]}$ in $\ps_1 ( C ( [0,T], \R^d) )$, $\Phi^k_t (x)$ converges to $0$ uniformly in $(t,x)$.
Similarly, using our Lipschitz assumptions, $\psi^k$ uniformly converges towards $\overline\psi$.
Since convolution preserves already existing regularity properties, the regularity properties w.r.t. $x$ assumed by Theorem \ref{thm:density} hold for $b^k$, $\sigma^k$, $\tfrac{\delta}{\delta\mu} b^k$ and $\tfrac{\delta}{\delta\mu} \Psi^k$, with regularity bounds independent of $k$.
To take the $k \rightarrow +\infty$ limit within Proposition \ref{pro:density}, we still have to check that the time-derivatives of the coefficients are bounded uniformly in $k$. 
This is clear for $\sigma^k_t$ because $x$-convolution does not affect time-regularity. 
Moreover, from e.g. \cite[Proposition 1.3]{daudin2023optimal}, 
\[ \frac{\d}{\d t} b^k_t (x,\overline\mu^k_t) = \partial_t b^k_t ( x,\overline{\mu}^k_t ) + \int_{\R^d} L_{t}^{\alpha^k} \big[ y \mapsto \frac{\delta b_t}{\delta\mu} (x, \mu_t, y) \big] \d \overline\mu_t,\]
\[ \partial_t {\psi}^k_t (x) = \int_{\R^d} \frac{\delta\Psi}{\delta\mu} (\overline\mu^k_t) + L_{t}^{\alpha^k} \big[ y \mapsto \frac{\delta^2 \Psi}{\delta\mu} (\mu_t, x,y) \big] \d \overline\mu^k_t. \]
From Lemma \ref{lem:boundPhi}, $\alpha^k$ is bounded and Lipschitz uniformly in $k$.
Using the bounds on $\varphi^k$, we similarly deduce that  $\overline\mu^k_{0}$ satisfies Lemma \ref{lem:InitalEntrop} with bounds independent of $k$.
Applying Lemma \ref{lem:Gamma_2} to $(\overline\mu^k_{t})_{0 \leq t \leq T}$
thus provides a bound on $\E [\lvert \nabla \log \overline\mu^k_t (X^0_t) \rvert]$ that is uniform in $k$. 
Integrating by parts, we get rid of second order derivatives in $y$, so that $\partial_t b^k_t$ and $\partial_t \psi^k_t$ are bounded uniformly in $k$ using \ref{ass:coefReg2}.
From Proposition \ref{pro:density}, we can now decompose
\[ \overline{\lambda}^k (\d t) = \overline{\lambda}^k_0 \delta_0 ( 
\d t) + \overline\lambda^k_t \d t + \overline{\lambda}^k_T \delta_T(\d t), \]
for some bounded measurable $t \mapsto \overline{\lambda}^k$, whose $L^\infty$-norm only depends on $\overline\lambda^k ([0,T])$ and the regularity constants allowed by Theorem \ref{thm:density}.

In the proof of Lemma \ref{lem:NLApprox}, qualification was obtained using $\tilde\mu^k_{[0,T]}$ converging to $\tilde\mu_{[0,T]}$ with bounded $H( \tilde\mu^k_{[0,T]} \vert \Gamma^k ( \overline\mu^k_{[0,T]} ) )$ uniformly in $k$.
Using \eqref{eq:LambdaMass} and \eqref{eq:unifqualif}, there exists $\eta >0$ independent of $k$ such that
\begin{equation*} 
\overline{\lambda}^k ([0,T]) \leq \tilde\varepsilon \eta^{-1} \int_{\ps_1(C([0,T],\R^d)} \bigg[ \log \frac{\d \overline{\mu}^k_{[0,T]}}{\d \Gamma^k (\overline{\mu}^k_{[0,T]})} + \frac{\delta\F^k}{\delta\mu} (\overline{\mu}^k_{[0,T]}) \bigg] \d \tilde{\mu}^k_{[0,T]}.
\end{equation*}
Since $\tilde{\mu}^k_{[0,T]}$ converges in $\ps_1(C([0,T],\R^d))$ and $\tfrac{\delta\F^k}{\delta\mu}$ has linear growth uniformly in $k$ from Lemma \ref{lem:NLRate} and \eqref{eq:NLclosure}, $\int \tfrac{\delta\F^k}{\delta\mu} (\overline{\mu}^k_{[0,T]}) \d \tilde\mu^k$ is bounded uniformly in $k$.
From \cite[Lemma 1.4-(b)]{nutz2021introduction},
\begin{equation*} 
\int_{\ps_1(C([0,T],\R^d)} \log \frac{\d \overline{\mu}^k_{[0,T]}}{\d \Gamma^k (\overline{\mu}^k_{[0,T]})} \d \tilde{\mu}^k_{[0,T]} \leq H( \tilde{\mu}^k_{[0,T]} \vert \Gamma^k (\overline{\mu}^k_{[0,T]}) ),
\end{equation*}
which is bounded uniformly in $k$.
This implies a bound on $( \overline\lambda^k ( [0,T] ) )_{k \geq 1}$.
From (a variant of) the Prokhorov theorem, up to re-indexing, we can thus assume that $\overline\lambda^k$ converges towards some $\overline\lambda$ in $\M_+ ( [0,T] )$.
In particular, we assume that $( \overline{\lambda}^k_0 )_{k \geq 1}$ and $( \overline{\lambda}^k_T )_{k \geq 1}$ respectively converge towards $\overline{\lambda}(\{0\})$ and $\overline{\lambda}(\{T\})$. 
This shows that
\[ \overline{\lambda}(\d t) = \overline{\lambda}_0 \delta_0 ( 
\d t) + \lambda_t \d t + \overline{\lambda}_T \delta_T(\d t), \]
for some bounded measurable function $t \mapsto \overline\lambda_t$, which is the weak-$\star$ limit of $( \overline\lambda^k )_{k \geq 1}$ (for the $\sigma (L^\infty,L^1)$-topology).

From Lemma \ref{lem:boundPhi}, $\nabla\overline{\varphi}^k_t$ and $\nabla^2 \overline{\varphi}^k_t$ are bounded uniformly in $(t,k)$.
Moreover, for every $1 \leq i \leq d$ and $0<t \leq T$,
\begin{multline} \label{eq:gradC0}
\partial_i \overline\varphi^k_t (x) = \overline{\lambda}^k_T \E [ \partial_i \psi^k_T (X^{t,x,\alpha^k}_T) ] + \int_t^T \E \big\{ \partial_i b^k_s \cdot \nabla \overline\varphi^k_s - \frac{1}{2} \nabla \overline\varphi^k_s \cdot \partial_i a^k_s \nabla \overline\varphi^k_s + \frac{1}{2} \mathrm{Tr}[ \partial_i a^k_s \nabla \overline\varphi^k_s] \\
+ \partial_i c^k_s + \overline\lambda^k_s \partial_i \frac{\delta\Psi^k}{\delta\mu} (\overline\mu^k_s) \big\} (X^{t,x,\alpha^k}_s) \d s. 
\end{multline}
Since $( \overline\lambda^k )_{k \geq 1}$ is now bounded in $L^\infty (0,T)$, we deduce that $( \nabla \overline\varphi^k )_{k \geq 1}$ is equi-continuous on every compact set of $(0,T] \times \R^d$.
Up to-reindexing, the Arzelà-Ascoli allows us to assume that $\overline\varphi^k$ converges towards $\overline{\varphi} \in C^1 ((0,T] \times \R^d)$ uniformly on every compact set.
Similarly, we can define $\overline{\varphi}_0$ so that $\nabla\overline{\varphi}^k_0$ uniformly converges towards $\nabla\overline{\varphi}_0$ on every compact set of $\R^d$.
Going back to \eqref{eq:NLclosure}, $\overline{c}^k$ uniformly converges towards $\overline{c}$ defined by
\[ \overline{c}_t(x) := \int_{\R^d} \frac{\delta b_s}{\delta\mu} ( y, \overline\mu_s, x ) \cdot \nabla \overline{\varphi}_s ( y ) \overline{\mu}_s ( \d y). \]
From the stability result \cite[Proposition 5.2]{ConstrainedSchrodinger}, we get that $\overline{\varphi}$ is a solution of \eqref{eq:coupled_FP_HJB}, so that $(\overline{\mu}_{[0,T]}, \overline{\lambda}, \overline{\varphi})$ satisfies Theorem \ref{thm:opti}-\ref{item:OptCond}.

Taking the $k \rightarrow +\infty$ limit in Lemma \ref{lem:boundPhi} applied to $\overline{\varphi}^k$, we get that $\nabla \overline\varphi_t$ is bounded and Lipschitz uniformly in $t$. 
Under \ref{ass:ini2}-\ref{item:inih2} and \ref{ass:coefReg3}-\ref{item:bC2}, we can further take the limit in Lemma \ref{lem:thirdorder} and get that $\nabla^2 \overline\varphi_t$ is globally Lipschitz-continuous.
Similarly, the analogous of \eqref{eq:gradC0} yields that $t \mapsto \nabla^2 \overline\varphi_t (x) $ is continuous on $(0,T]$.
The bound \eqref{eq:gamm3} results from Proposition \ref{pro:Gamma_3}. 
\end{proof}

As a consequence we get the following result, $( \overline{X}_{t})_{0 \leq t \leq T}$ denoting the optimally controlled process for the feed-back $\overline{\alpha} := -\sigma^\top \nabla \overline\varphi$.

\begin{corollary}[First order regularity] \label{cor:timeReg1} 
\begin{enumerate}[label=(\roman*),ref=(\roman*)] Under \ref{ass:coefReg3}-\ref{item:PsiC3}, 
\item\label{item:gradC0} $t \mapsto \varphi_t (x)$ is continuous on $(0,T]$. Moreover, $( \nabla \overline\varphi_t (\overline{X}_t) )_{0 < t \leq T}$ is almost surely continuous, and for $1 \leq i \leq d$,  
\begin{multline*} 
\d \partial_i \overline\varphi_t ( \overline{X}_t ) = \{ - \partial_i b_t \cdot \nabla \overline\varphi_t + \tfrac{1}{2} \nabla \overline\varphi_t \cdot \partial_i a_t \nabla \overline\varphi_t -  \tfrac{1}{2}\mathrm{Tr}[ \partial_i a_t \nabla^2 \overline\varphi_t] \\
-\partial_i \overline{c}_t -\overline\lambda_t \partial_i \tfrac{\delta\Psi}{\delta\mu}(\overline\mu_t) \} (\overline{X}_t) \d t - \sigma_t \nabla \partial_i \overline\varphi_t (\overline{X}_t) \d B_t. 
\end{multline*}
\item\label{item:PsiC1} $t \mapsto \Psi(\overline{\mu}_t)$ is $C^1$ on $[0,T]$, with bounded derivative given by
\[ \frac{\d}{\d t} \Psi(\overline{\mu}_t) = \E \bigg[ L^{\overline{\alpha}}_t \frac{\delta \Psi}{\delta\mu}(\overline\mu_t, \overline{X}_t) \bigg], \quad t \in (0,T]. \]
\end{enumerate}    
\end{corollary}

\begin{proof}
We follow the regularisation procedure in the above proof of Theorem \ref{thm:density}. 
Since $\overline{\lambda}$ has a bounded density,
the continuity of $t \mapsto \overline\varphi_t(x)$ can be read directly on \eqref{eq:coupled_FP_HJB}.
To get \ref{item:gradC0}, we simply take the $k \rightarrow + \infty$ limit within \eqref{eq:gradC0} initialised at $x = \overline{X}_t$.

For \ref{item:PsiC1}, we differentiate $\Psi^k ( \overline\mu^k_s )$ to get that for every $0 < t < t' < T$,
\[ \Psi^k_{t'} ( \overline\mu^k_{t'} ) - \Psi^k_{t} ( \overline\mu^k_{t} ) = \E \int_t^{t'} L^{\alpha^k}_s \frac{\delta\Psi^k}{\delta\mu} (\overline\mu^k_s, X^{\alpha^k}_s) \d s.  \]
Taking the $k \rightarrow +\infty$ limit and using that $( \nabla \overline\varphi_s ( \overline{X}_s) )_{0 < s \leq T}$ is continuous on $(0,T]$, $t \mapsto \Psi (\overline\mu_t)$ is $C^1$ on $(0,T]$ with derivative given by \ref{item:PsiC1}. 
From \ref{item:gradC0}, we see that $s \mapsto \E [ L^{\overline\alpha}_s \tfrac{\delta\Psi}{\delta\mu} (\overline\mu_s, \overline{X}_s) ]$ has a finite limit as $s \rightarrow 0$.
By continuous prolongation, $t \mapsto \Psi(\overline{\mu}_t)$ is thus differentiable on the full $[0,T]$ with continuous derivative.
\end{proof}

\begin{rem}[Linearising before regularising]
In the proof of Theorem \ref{thm:density}, we regularised the non-linear problem before to linearise it and to take the limit in the optimality conditions.
An alternative (and easier)
way is to linearise before to regularise.
In particular, this avoids using Lemma \ref{lem:mollif}.
However, this approach requires additional derivatives w.r.t. $\mu$ to bound the time-derivatives $\partial_s b^k_s$ and $\partial_s \psi^k_s$ uniformly in $k$.
\end{rem}

\subsection{Quantitative stability} \label{subsec:stability}

We can apply Theorem \ref{thm:opti} to $\mu^\varepsilon_{[0,T]}$ and the constraints $\Psi - \varepsilon$, because $\mu^0_{[0,T]}$ provides constraint qualification using Remark \ref{rem:Suffconv}.
Let us fix a reference probability system $\Sigma = ( \Omega,(\mathcal{F}_t)_{0\leq t\leq T},\P,(B_t)_{0\leq t\leq T} )$.
From Theorem \ref{thm:opti}, $\mu^\varepsilon_{[0,T]}$ is the law of the solution to the SDE
\[ \d X^\varepsilon_t = b_t ( X^\varepsilon_t ) \d t + \sigma_t \alpha^\varepsilon_t ( X^\varepsilon_t ) \d t + \sigma_t (X^\varepsilon_t ) \d B_t, \]
where $\alpha^\varepsilon_t = -\sigma^\top_t \nabla \varphi^\varepsilon_t$ is the optimal feed-back control. 
In the following, $C$ is a positive constant that may change from line to line, but always independent of $\varepsilon$.
From \eqref{eq:LambdaMass}, we thus have $\lambda^\varepsilon ([0,T]) \leq C$.
Using Theorem \ref{thm:density}, we further get 
\begin{equation} \label{eq:BoundsUnifeps}
\sup_{0 \leq t \leq T} \lVert \nabla \varphi^\varepsilon_t \rVert_\infty + \lVert \nabla^2 \varphi^\varepsilon_t \rVert_\infty \leq C, \quad \sup_{0 \leq t \leq T} \lVert \alpha^\varepsilon_t \rVert_\infty + \lVert \nabla \alpha^\varepsilon_t \rVert_\infty \leq C, 
\end{equation}
and a straightforward consequence is that $\E [ \sup_{0 \leq t \leq T} \vert X^\varepsilon_t \vert^2 ] \leq C$. 

\begin{lemma}[$L^2$-stability] \label{lem:L2stab}
\begin{enumerate}[label=(\roman*),ref=(\roman*)] $\phantom{a}$
\item\label{item:entropBound} $H(\mu^0_{[0,T]} \vert \mu^\varepsilon_{[0,T]} ) = H(\mu^0_0 \vert \mu^\varepsilon_0) + \E \int_0^T \tfrac{1}{2} \lvert \alpha^\varepsilon_t ( 
X^0_t) - \alpha^0_t (  X^0_t ) \rvert^2 \d t \leq C \varepsilon$.
\item\label{item:W1Bound} $\sup_{0 \leq t \leq T} W_1 (\mu^\varepsilon_t, \mu^0_t) \leq C \varepsilon^{1/2}$.
\item\label{item:L2Bound} Let $r^\varepsilon_t (x) := f_t (x) g_t ( x ) h_t(\alpha^\varepsilon_t (x))$, for measurable $f,g,h : [0,T] \times \R^d \rightarrow \R$.
If $f$ is bounded, $g$ is globally Lipschitz in $x$ and $h$ is locally-Lipschitz (uniformly in $t$), 
\[ \E \int_0^T \lvert r^\varepsilon_t (X^\varepsilon_t) - r^0_t (X^0_t ) \rvert \d t \leq C \varepsilon^{1/4}. \] 
\end{enumerate}    
\end{lemma}

\begin{proof}
First, \ref{item:entropBound} is a direct application of the stability result proved \cite[Proposition 2.18 and Remark 2.19]{ConstrainedSchrodinger}, the relative entropy being computed using the Girsanov transform. 
Let us bound $W_1(\mu^\varepsilon_0,\mu^0_0)$. Since $\varphi^\varepsilon_0$ has linear growth independently of $\varepsilon$, Lemma \ref{lem:InitalEntrop} implies that the bound \ref{ass:ini2}-\ref{item:iniT1} holds for $\overline{\mu}^\varepsilon$ with a constant independent of $\varepsilon$. 
This Gaussian bound is known to imply a transport-entropy inequality: $C > 0$ independent of $\varepsilon$ exists such that 
\[ \forall \mu' \in \ps_1 ( \R^d), \quad W^2_1( \mu', \overline\mu_0^\varepsilon ) \leq C H ( \mu' \vert \overline\mu_0^\varepsilon).
\]
Together with \ref{item:entropBound}, this implies $W_1(\mu^\varepsilon_0,\mu^0_0) \leq C \varepsilon^{1/2}$.
Using Ito's formula, our Lipschitz assumptions on the coefficients and the uniform bound on $\alpha^\varepsilon_t$:
\[ \E [ \vert X^\varepsilon_t - X^0_t \vert ] \leq \E [ \vert X^\varepsilon_0 - X^0_0 \vert ] + C \int_0^t \E [ \vert X^\varepsilon_s - X^0_s \vert ] \d s + C \int_0^t \E [ \vert \alpha^\varepsilon_s - \alpha^0_s \vert ] \d s. \]
Using \ref{item:entropBound} and Gronwall's Lemma, we establish \ref{item:W1Bound}.
To prove \ref{item:L2Bound}, we decompose:
\begin{multline*}
\E \int_0^T \lvert r^\varepsilon_t (X^\varepsilon_t) - r^0_t (X^0_t ) \rvert \d t \leq \E \int_0^T \lvert g_t (X^\varepsilon_t) [ f_t (X^\varepsilon_t) h_t ( \alpha^\varepsilon_t ( X^\varepsilon_t) ) - f_t (X^0_t) h_t ( \alpha^\varepsilon_t ( X^0_t) ) ] \rvert \d t \\
+ \E \int_0^T \lvert f_t (X^0_t) g_t (X^\varepsilon_t) [ h_t (\alpha^\varepsilon_t (X^0_t)) - h_t (\alpha^0_t (X^0_t)) ] \rvert + \lvert f_t (X^0_t) h_t ( \alpha^0_t ( X^0_t) )[ g_t (X^\varepsilon_t) - g_t (X^0_t ) ] \rvert \d t. 
\end{multline*}
Using the Cauchy-Schwarz inequality, the first term is bounded by 
\[ C \bigg[ \E \int_0^T \vert g_t (X^\varepsilon_t) \vert^2 \d t \bigg]^{1/2} \sup_{0 \leq t \leq T} \lVert (h_t \circ \alpha^\varepsilon_t) f_t \rVert_\infty \lVert \mu^\varepsilon_t - \mu^0_t \rVert_{\mathrm{TV}}^{1/2}. \]
This term is bounded by $C \varepsilon^{1/4}$ using \eqref{eq:BoundsUnifeps}, \ref{item:entropBound} and Pinsker's inequality.
Similarly, we bound the second term by
\[ C \bigg[ \E \int_0^T \vert g_t (X^\varepsilon_t) \vert^2 \d t \bigg]^{1/2} \bigg[ \E \int_0^T \vert \alpha^\varepsilon_t (X^0_t) - \alpha^0_t (X^0_t) \vert^2 \d t \bigg]^{1/2}, \]
using that $h_t$ is locally Lispchitz and $\alpha^\varepsilon$ is bounded. 
This quantity is smaller than $C \varepsilon^{1/2}$ using \ref{item:entropBound} and the fact that $\E [ \sup_{0 \leq t \leq T} \vert X^\varepsilon_t \vert^2 ] \leq C$. 
Since $g_t$ is Lipschitz and $(h_t \circ \alpha^0_t)f_t$ is bounded, the third term is bounded by $C \sup_{0 \leq t \leq T} W_1 (\mu^\varepsilon_t, \mu^0_t)$, which we bound using \ref{item:W1Bound}.
\end{proof}

In the following, we use the notation $\psi^\varepsilon_t := \Psi (\mu^\varepsilon_t) + \varepsilon + \tfrac{\delta\Psi}{\delta\mu}(\mu^\varepsilon_t)$.
We recall that $\tfrac{\delta\Psi}{\delta\mu}(\mu^\varepsilon_t) : x \mapsto \tfrac{\delta\Psi}{\delta\mu}(\mu^\varepsilon_t,x)$ is defined in Definition \ref{def:PAP2DiFF}.
The Lipschitz assumption \ref{ass:coefReg3}-\ref{item:PsiC4} and Lemma \ref{lem:L2stab}-\ref{item:L2Bound} imply
that $\Psi$ is Lipschitz in $W_1$ (see e.g. \cite[Section 2.2.1]{cardaliaguet2019master}). 
From Lemma \ref{lem:L2stab}-\ref{item:W1Bound}, we get that 
\begin{equation} \label{eq:StabPsi}
\sup_{0 \leq s \leq T} \lvert \Psi (\mu^\varepsilon_s) - \Psi (\mu^0_s) \rvert  \leq C \varepsilon^{1/2}.
\end{equation}
Similarly,
\begin{equation} \label{eq:StabPartisalPsi}
\max_{0 \leq i \leq 4} \sup_{0 \leq s \leq T} \lVert \nabla^i \psi^\varepsilon_s - \nabla^i \psi^0_s \rVert_\infty \leq C \varepsilon^{1/2}.
\end{equation}
As a consequence, we can readily adapt the proof of Lemma \ref{lem:L2stab}-\ref{item:L2Bound} to include the derivatives of $\psi^\varepsilon_t$.

\begin{lemma} \label{lem:L2stabImp}
Lemma \ref{lem:L2stab}-\ref{item:L2Bound} still holds if \[ r^\varepsilon_t (x) = f_t(x) \psi^\varepsilon_t(x) h_t (\alpha^\varepsilon_t(x)) \quad \text{or} \quad r^\varepsilon_t (x) = f_t(x) g_t(x) h_t (\alpha^\varepsilon_t(x)) k^\varepsilon_t(x), \] 
where $k^\varepsilon_t$ is any of the four derivatives of $\psi^\varepsilon_t$.
\end{lemma}

To make $\lambda^\varepsilon_t$ explicit, we differentiate twice $\Psi ( \mu^\varepsilon_t )$.
We recall the definition \eqref{eq:InfGen} of the infinitesimal generator $L_t^{\alpha^\varepsilon}$.

\begin{lemma}[Second order regularity] \label{lem:timeReg2} 
Under \ref{ass:coefReg3} there exists (a linear combination of terms) $r^\varepsilon$ as in Lemma \ref{lem:L2stabImp} such that for every $0 \leq t < t' \leq T$,
\begin{multline*}
\E [ L^{\alpha^\varepsilon}_{t'} \psi^\varepsilon_{t'}({X}^\varepsilon_{t'}) ] - \E [ L^{\alpha^\varepsilon}_{t} \psi^\varepsilon_t ({X}^\varepsilon_{t}) ]
= \mathbbm{1}_{t = 0}{\lambda}^\varepsilon_0 \E [ \vert \sigma_0^\top \nabla \psi^\varepsilon_0 \vert^2 ( X^\varepsilon_0 ) ]  
+ \mathbbm{1}_{t' = T} \lambda^\varepsilon_T \E [ \vert \sigma_T^\top \nabla \psi^\varepsilon_T \vert^2 ( X^\varepsilon_T ) ] \\
+\int_t^{t'} \E [ r^\varepsilon_s ( X^\varepsilon_s ) - s^\varepsilon_s ( X^\varepsilon_s ) + \lambda^\varepsilon_s \vert \sigma_s^\top \nabla \psi^\varepsilon_s \vert^2 ( X^\varepsilon_s ) ] \d s,    
\end{multline*}
where
\[ s^\varepsilon_s = a^{i,j}_s \nabla \partial_i \psi^\varepsilon_s \cdot a_s \nabla \partial_j \varphi^\varepsilon_s + \partial_i \psi^\varepsilon_s \nabla a_t^{i,j} \cdot a_s \nabla \partial_j \varphi^\varepsilon_s - \tfrac{1}{2} a^{i,j}_s \partial_j a^{k,l}_s \partial_i \psi^\varepsilon_s \partial^2_{k,l} \varphi^\varepsilon_s. \]
As a consequence of Corollary \ref{cor:timeReg1}-\ref{item:PsiC1}, $\E [ L^{\alpha^\varepsilon}_{t} \psi^\varepsilon_{t}({X}^\varepsilon_{t}) ] = \frac{\d}{\d t}  \Psi ( \mu^\varepsilon_{t} )$ when $t >0$, hence $t \mapsto \Psi({\mu}^\varepsilon_t)$ is a.e. twice differentiable, and $C^2$ on intervals where $\lambda^\varepsilon_t =0$ identically.
\end{lemma}

\begin{proof}
Using Lemma \ref{lem:boundPhi}, we can compute on a regularised version of the multiplier before to take the limit (and then recover the atoms at $t \in \{0,T \}$). 
We thus assume that $\lambda^\varepsilon$ is $C^1$, and we use Ito's formula to write that
\begin{equation*}
\E [ L^{\alpha^\varepsilon}_{t'} \psi^\varepsilon_{t'}({X}^\varepsilon_{t'}) ] - \E [ L^{\alpha^\varepsilon}_{t} \psi^\varepsilon_t ({X}^\varepsilon_{t}) ] = \int_{t}^{t'} \E [ (\partial_s + L^{\alpha^\varepsilon}_{s} ) L^{\alpha^\varepsilon}_{s} \psi^\varepsilon_s (X^\varepsilon_s) ]  \d s. 
\end{equation*}
In the following, $r^\varepsilon_s$ is a remainder of the desired shape which may change from line one to another.
Decomposing $L^{\alpha^\varepsilon}_s = L^0_s + \sigma_s \alpha^\varepsilon_s \cdot \nabla$,
\begin{align*}
(\partial_s + L^{\alpha^\varepsilon}_{s} ) L^{\alpha^\varepsilon}_{s} \psi^\varepsilon_s &= ( \partial_s + L^{\alpha^\varepsilon}_s ) \sigma_s \alpha^\varepsilon_s \cdot \nabla \psi^\varepsilon_s + ( \partial_s + L^{\alpha^\varepsilon}_s ) L^0_s \psi^\varepsilon_s \\
&= ( \partial_s + L^{\alpha^\varepsilon}_s ) \sigma_s \alpha^\varepsilon_s \cdot \nabla \psi^\varepsilon_s + r^\varepsilon_s.
\end{align*}
We now use coordinates: 
\begin{align*}
( \partial_s + L^{\alpha^\varepsilon}_s ) \sigma_s \alpha^\varepsilon_s \cdot &\nabla \psi^\varepsilon_s = ( \partial_s + L^{\alpha^\varepsilon}_s ) ( \sigma_s \alpha^\varepsilon_s )^i \partial_i \psi^\varepsilon_s \\
&= \partial_i \psi^\varepsilon_s ( \partial_s + L^{\alpha^\varepsilon}_s ) ( \sigma_s \alpha^\varepsilon_s )^i + ( \sigma_s \alpha^\varepsilon_s )^i (\partial_s + L^{\alpha^\varepsilon}_s ) \partial_i \psi^\varepsilon_s + \nabla \partial_i \psi^\varepsilon_s \cdot a_s \nabla ( \sigma_s \alpha^\varepsilon_s )^i \\
&= (\partial_s + L^{\alpha^\varepsilon}_s ) ( \sigma_s \alpha^\varepsilon_s )^i + r^\varepsilon_s.
\end{align*}
Since $( \sigma_s \alpha^\varepsilon_s )^i = - ( a_s \nabla \varphi^\varepsilon_s )^i = -a^{i,j}_s \partial_j \varphi^\varepsilon_s$, we further get
\begin{align*}
(\partial_s + L^{\alpha^\varepsilon}_s ) ( \sigma_s \alpha^\varepsilon_s )^i &= -\partial_j \varphi^\varepsilon_s ( \partial_s + L_s^{\alpha^\varepsilon} ) a_s^{i,j} - a_s^{i,j} ( \partial_s + L_s^{\alpha^\varepsilon} ) \partial_j \varphi^\varepsilon_s - \nabla a_s^{i,j} \cdot a_s \nabla \partial_j \varphi^\varepsilon_s \\
&= - a_s^{i,j} ( \partial_s + L_s^{\alpha^\varepsilon} ) \partial_j \varphi^\varepsilon_s + r^\varepsilon_s.
\end{align*}
We eventually differentiate she HJB equation satisfied by $\varphi^\varepsilon$:
\[ (\partial_s + L^{\alpha^\varepsilon}_s) \partial_j \varphi^\varepsilon_s = - \partial_j b_s \cdot \nabla \varphi^\varepsilon_s + \frac{1}{2} \nabla \varphi^\varepsilon_s \partial_j a_s \nabla \varphi^\varepsilon_s - \frac{1}{2}\mathrm{Tr}[ \partial_j a_s \nabla^2 \varphi^\varepsilon_s] - \lambda^\varepsilon_s \partial_j \psi^\varepsilon_s. \]
Gathering everything, we get $r^\varepsilon$
such that
\begin{multline*}
(\partial_s + L^{\alpha^\varepsilon}_{s} ) L^{\alpha^\varepsilon}_{s} \psi^\varepsilon_s = r^\varepsilon_s - \nabla \partial_i \psi^\varepsilon_s \cdot a_s \nabla ( a_s \nabla \varphi^\varepsilon_s )^i - \partial_i \psi^\varepsilon_s \nabla a_s^{i,j} \cdot a_s \nabla \partial_j \varphi^\varepsilon_s \\
+ \frac{1}{2} a^{i,j}_s \partial_i \psi^\varepsilon_s \mathrm{Tr}[ \partial_j a_s \cdot \nabla^2 \varphi^\varepsilon_s] + \lambda^\varepsilon_s a^{i,j}_s \partial_i \psi^\varepsilon_s \partial_j \psi^\varepsilon_s.
\end{multline*}  
We eventually obtain the desired $s^\varepsilon$ after a slight modification of $r^\varepsilon$.
\end{proof}

We will repeatedly use the following result. 

\begin{lemma} \label{lem:locMax}
For any global maximum or minimum $t_0 \in (0,T)$ of $s \mapsto \Psi (\mu^0_{s})$, we have 
\[ \bigg\vert \frac{\d}{\d s} \big\vert_{s = t_0} \Psi ( \mu^\varepsilon_s ) \bigg\vert \leq C \varepsilon^{1/4}, \]
for $C > 0$ that does not depend on $\varepsilon$ nor $t_0$.
\end{lemma}

The result still holds if we exchange the roles of $\Psi (\mu^\varepsilon_s)$ and $\Psi (\mu^0_s)$.

\begin{proof}
We prove the result for a maximum, the case of a minimum being very similar.
For any $\beta, \gamma >0$,
\[ \forall s \in [0,T] \setminus [t_0 - \beta, t_0 + \beta], \quad \Psi( \mu^0_{t_0} ) \geq \Psi( \mu^0_s ) - \gamma ( s - t_0 )^2 + \gamma \beta^2. \]
Using \eqref{eq:StabPsi}, this implies
\[ \forall s \in [0,T] \setminus [t_0 - \beta, t_0 + \beta], \quad \Psi( \mu^\varepsilon_{t_0} ) \geq \Psi( \mu^\varepsilon_s ) - \gamma ( s - t_0 )^2 + \gamma \beta^2 - 2 C \varepsilon^{1/2}. \]
If $\gamma\beta^2 \geq 2 C \varepsilon^{1/2}$, this proves that $s \mapsto \Psi( \mu^\varepsilon_s ) - \gamma ( s - t_0 )^2$ has a global maximum at some point $t'_0 \in [t_0-\beta,t_0 + \beta]$. 
As a consequence $ \tfrac{\d}{\d s} \vert_{s = t'_0} \Psi( \mu^\varepsilon_s ) = 2 \gamma ( t'_0 - t_0)$.
From Lemma \ref{lem:timeReg2}, 
$\tfrac{\d^2}{\d s^2} \Psi( \mu^\varepsilon_s )$ is bounded on $[0,T]$ by some $K >0$ independent of $\varepsilon$, hence
\[ \bigg\vert \frac{\d}{\d s} \big\vert_{s = t_0} \Psi ( \mu^\varepsilon_s ) \bigg\vert \leq 2 \beta \gamma + K \beta,  \]
for any $\beta, \gamma > 0$ such that $\gamma \beta^2 \geq 2 C \varepsilon^{1/2}$.
Optimising over $(\beta,\gamma)$ concludes.
\end{proof}

By continuity of $t \mapsto \Psi (\mu^\varepsilon_t )$, the sets $F_\varepsilon := \{ t \in [0,T] \; , \; \Psi (\mu^\varepsilon_t ) = \varepsilon \}$ are closed. 
Moreover, $F_\varepsilon$ is a countable union of singletons and non-trivial intervals, and we know that $F_\varepsilon$ contains the support of $\lambda^\varepsilon$ from the complementary slackness condition \eqref{eq:CompSlackmf} with $\Psi - \varepsilon$ in place of $\Psi$.
From \ref{ass:consquali} and Remark \ref{rem:Suffqualif}, the continuous function $s \mapsto \E [ \lvert \sigma^\top_s \nabla \psi^0_s ( \overline{X}_s ) \rvert^2 ]$ is positive on the closed set $F_0$, hence lower-bounded by a positive number.
Using this and Lemma \ref{lem:L2stab}-\ref{item:W1Bound}, there exists $\eta >0$ such that for every small enough $\varepsilon$,
\begin{equation} \label{eq:Qualifeps}
\forall t \in F_\varepsilon, \quad \E [ \lvert \sigma^\top_t \nabla \psi^\varepsilon_t \rvert^2 (X^\varepsilon_t ) ] \geq \eta > 0.   
\end{equation}
We now prove the main stability estimate on $\lambda^\varepsilon$.

\begin{proposition} \label{pro:L1estiamte}
There exists $C >0$ independent of $\varepsilon$ such that $\lVert \lambda^\varepsilon - \lambda^0 \rVert_{L^1(0,T)} \leq C \varepsilon^{1/4}$.
\end{proposition}

\begin{proof} We decompose the norm using $F_\varepsilon$ and $F_0$.

\medskip

\emph{\textbf{Step 1.} On $F_\varepsilon \cap F_0$.} We first write that
\begin{multline*}
\vert \E [ \lambda^\varepsilon_t \lvert \sigma^\top_t \nabla \psi^\varepsilon_t (X^\varepsilon_t) \rvert^2 - \lambda^0_t \lvert \sigma^\top_t \nabla \psi^\varepsilon_t (X^\varepsilon_t) \rvert^2 ] \vert \leq \vert \E [ \lambda^\varepsilon_t \lvert \sigma^\top_t \nabla \psi^\varepsilon_t (X^\varepsilon_t) \rvert^2 - \lambda^0_t \lvert \sigma^\top_t \nabla \psi^0_t (X^0_t) \rvert^2 ] \vert \\
+ \lambda^0_t \vert \E [ \lvert \sigma^\top_t \nabla \psi^\varepsilon_t (X^\varepsilon_t) \rvert^2 - \lvert \sigma^\top_t \nabla \psi^0_t (X^0_t) \rvert^2 ] \vert.
\end{multline*}
Moreover,
\begin{multline*}
\vert \E [ \lvert \sigma^\top_t \nabla \psi^\varepsilon_t (X^\varepsilon_t) \rvert^2 - \lvert \sigma^\top_t \nabla \psi^0_t (X^0_t) \rvert^2 ] \vert \leq \vert \E [ \lvert \sigma^\top_t \nabla \psi^\varepsilon_t (X^\varepsilon_t) \rvert^2 - \lvert \sigma^\top_t \nabla \psi^\varepsilon_t (X^0_t) \rvert^2 ] \vert \\
+ \vert \E [ \lvert \sigma^\top_t \nabla \psi^\varepsilon_t (X^0_t) \rvert^2 - \lvert \sigma^\top_t \nabla \psi^0_t (X^0_t) \rvert^2 ] \vert \leq C \varepsilon^{1/2}, 
\end{multline*}
combining Lemma \ref{lem:L2stab}-\ref{item:W1Bound} and \eqref{eq:StabPartisalPsi} for the second inequality.
Using \eqref{eq:Qualifeps} and the bound on $\lambda^0_t$ given by Theorem \ref{thm:density}, this implies
\[ \eta \lvert \lambda^\varepsilon_t - \lambda^0_t \rvert \leq \lvert \E [ \lambda^\varepsilon_t \lvert \sigma^\top_t \nabla \psi^\varepsilon_t (X^\varepsilon_t) \rvert^2 - \lambda^0_t \lvert \sigma^\top_t \nabla \psi^0_t (X^0_t) \rvert^2 ] \rvert + C \varepsilon^{1/2}. \]
By continuity of $\Psi$,
$F_\varepsilon \cap F_0$ is a (at most) countable union of closed disjoint intervals.
$\Psi (\mu^\varepsilon_s)$ and and $\Psi (\mu^0_s)$ are identically $0$ on $F_\varepsilon \cap F_0$, hence their second derivative is a.e. $0$ on $F_\varepsilon \cap F_0$. From Lemmata \ref{cor:timeReg1} and \ref{lem:timeReg2}, 
\[ 0 = \frac{\d^2}{\d s^2} \Psi (\mu^\varepsilon_s) = \E [ r^\varepsilon_s (X^\varepsilon_s) - s^\varepsilon_s (X^\varepsilon_s) + \lambda^\varepsilon_s \vert \sigma_s^\top \nabla \psi^\varepsilon_s \vert^2 (X^\varepsilon_s) ] \quad \text{for a.e. } s \in F_0 \cap F_\varepsilon, \]
hence
\begin{equation} \label{eq:IntermStep1}
\eta \lvert \lambda^\varepsilon_t - \lambda^0_t \rvert \leq \vert \E [ r^\varepsilon_t (X^\varepsilon_t) - r^0_t (X^0_t) ] \vert + \vert \E [ s^\varepsilon_t (X^\varepsilon_t) - s^0_t (X^0_t) ] \vert.
\end{equation}
The term $s^\varepsilon_s$ can be written as a sum of terms of shape
$m^\varepsilon_s \cdot \nabla \partial_j \varphi^\varepsilon_s$, where $m^\varepsilon$ is uniformly bounded with bounded $\partial_j$-derivative.
However, $\E[ \nabla \partial_j \varphi^\varepsilon_s ( X^\varepsilon_s ) - \nabla \partial_j \varphi^0_s ( X^0_s ) ]$ cannot be controlled using Lemma \ref{lem:L2stabImp}. 
To circumvent this, we integrate by parts: 
\begin{align*}
\E &[ m^0_s \cdot \nabla \partial_j \varphi^0_s (X^0_s) ] = - \E [ \partial_j m^0_s \cdot \nabla \varphi^0_s (X^0_s) ] - \E [ (m^0_s \cdot \nabla \varphi^0_s) \partial_j \log \mu^0_s (X^0_s) ] \\
&= - \E [ \partial_j m^0_s \cdot \nabla \varphi^0_s (X^0_s) ] - \E [ (m^0_s \cdot \nabla \varphi^\varepsilon_s) \partial_j \log \mu^0_s (X^0_s) ] + \E [ m^0_s \cdot ( \nabla \varphi^\varepsilon_s - \nabla \varphi^0_s ) \partial_j \log \mu^0_s (X^0_s) ] \\
&= \E [ m^0_s \cdot \nabla \partial_j \varphi^\varepsilon_s (X^0_s) ] + \E [ \partial_j m^0_s \cdot (\nabla \varphi^\varepsilon_s - \nabla \varphi^0_s ) (X^0_s) ] + \E [ m^0_s \cdot ( \nabla \varphi^\varepsilon_s - \nabla \varphi^0_s ) \partial_j \log \mu^0_s (X^0_s) ].
\end{align*} 
Using the Cauchy-Schwarz inequality and Lemmata \ref{lem:InitalEntrop}-\ref{lem:Gamma_2} to bound $\E [ \vert \nabla \log \mu^0_s (X^0_s) \rvert^2 ]$, 
\begin{multline*}
\int_0^T \lvert \E [ m^0_s \cdot ( \nabla \varphi^\varepsilon_s - \nabla \varphi^0_s ) \partial_j \log \mu^0_s (X^0_s) ] \rvert \d t \\
\leq \lVert m^0 \rVert_\infty \bigg\{ \int^T_0 \E [ \vert \nabla \log \mu^0_s (X^0_s) \rvert^2 ] \d s \bigg\}^{1/2} \bigg\{ \int^T_0 \E [ \vert \nabla \varphi^\varepsilon_s (X^0_s) - \nabla \varphi^0_s (X^0_s) \rvert^2 ] \d s \bigg\}^{1/2} \leq C \varepsilon^{1/2},
\end{multline*} 
Lemma \ref{lem:L2stab}-\ref{item:entropBound} giving the last bound.
We similarly show that
\[ \int_0^T \lvert \E [ \partial_j m^0_s \cdot (\nabla \varphi^\varepsilon_s - \nabla \varphi^0_s ) (X^0_s) ] \vert \d s \leq C \varepsilon^{1/2}. \]
We can now use Lemma \ref{lem:L2stabImp} to write 
\[ \int_0^T \vert \E [ m^\varepsilon_s \cdot \nabla \partial_j \varphi^\varepsilon_s (X^\varepsilon_s) - m^0_s \cdot \nabla \partial_j \varphi^\varepsilon_s (X^0_s) ] \vert \d s \leq C \varepsilon^{1/4}. \]
Gathering everything, we showed that
\[ \int_0^T \vert \E [ m^\varepsilon_s \cdot \nabla \partial_j \varphi^\varepsilon_s (X^\varepsilon_s) - m^0_s \cdot \nabla \partial_j \varphi^0_s (X^0_s) ] \vert \d s \leq C \varepsilon^{1/4}, \]
hence the same holds when replacing $m^\varepsilon_t$ by $s^\varepsilon_t$.
Going back to \eqref{eq:IntermStep1}, Lemma \ref{lem:L2stabImp} now gives that
$\int_{F_\varepsilon \cap F_0} \lvert \lambda^\varepsilon_s - \lambda^0_s \rvert \d s \leq C \varepsilon^{1/4}$.

\medskip

\emph{\textbf{Step 2.} On $F^c_\varepsilon \cap F_0$.} Using the complementary slackness condition
\begin{equation} \label{eq:DecompLambdaF}
\int_{( F_\varepsilon \cap F_0 )^c} \lvert \lambda^\varepsilon_s - \lambda^0_s \rvert \d s = \int_{F^c_\varepsilon \cap F_0} \lambda^0_s \d s + \int_{F_\varepsilon \cap F^c_0} \lambda^\varepsilon_s \d s. \end{equation} 
Since $t \mapsto \Psi(\mu^\varepsilon_t)$ and $t \mapsto \Psi(\mu^0_t)$ are continuous, $[0,T] \setminus (F_\varepsilon \cap F_0)$ is a (at most) countable union of disjoint intervals.
Let $I = (t,t')$ be one them that is not trivial. 

If $0 < t < t' < T$, then $\Psi ( \mu^\varepsilon_t ) = \Psi ( \mu^\varepsilon_{t'} ) = \varepsilon$ and $\Psi ( \mu^0_t ) = \Psi ( \mu^0_{t'} ) = 0$ by definition of $(F_\varepsilon \cap F_0)^c$. 
From Lemma \ref{cor:timeReg1}, $s \mapsto \tfrac{\d}{\d s} \Psi ( \mu^\varepsilon_s )$ and $s \mapsto \tfrac{\d}{\d s} \Psi ( \mu^0_s )$ are $C^1$ at $t$ and $t'$, hence they vanish there ($t$ and $t'$ are local maxima). We then set $t_0 := t$ and $t'_0 := t'$.
Otherwise, if $t'= 0$ or $t = T$,
we consider $t_0, t'_0 \in I$ such that $0 < t_0 < t'_0 < T$ and $\Psi (\mu^0_{t_0}) = \Psi (\mu^0_{t'_0}) =0$. 

Let us apply Lemma \ref{lem:locMax} at $t_0$ and $t'_0$.
Using Lemma \ref{lem:timeReg2} on $[t_0,t'_0]$, reasoning as in \emph{\textbf{Step 1.}} yields
\[ \int_{t_0}^{t'_0} \lambda^0_s \d s - \int_{t_0}^{t'_0} \lambda^\varepsilon_s \d s \leq C \varepsilon^{1/4}, \]
hence 
\[ \int_{t_0}^{t'_0} \lambda^0_s \d s \leq C \varepsilon^{1/4} + \int_I \lambda^\varepsilon_s \d s. \]
Since this holds for every $[t_0,t'_0] \subset I$ with $\Psi (\mu^0_{t_0}) = \Psi (\mu^0_{t'_0}) = 0$, we deduce that 
\[ \int_I \lambda^0_s \d s \leq C \varepsilon^{1/4} + \int_I \lambda^\varepsilon_s \d s, \]
recalling the complementary slackness condition \eqref{eq:CompSlackmf}.
To bound $\int_{F^c_\varepsilon \cap F_0} \lambda^0_s \d s$, it is thus sufficient to control $\int_{F_\varepsilon \cap F^c_0} \lambda^\varepsilon_s \d s$. 
This is the content of the next step. 

\medskip

\emph{\textbf{Step 3.} On $F_\varepsilon \cap F^c_0$.} 
We now use the assumption that there are a finite number of non-trivial intervals where $\Psi(\mu^0_s) = 0$ identically.
Let $[a_k,b_k]$ be these intervals in an increasing order, with $a_k < b_k$ and $1 \leq k \leq m$. 
We further set $b_0 := 0$ and $a_{m+1} := T$ so that
\[ \int_{F_\varepsilon \cap F^c_0} \lambda^\varepsilon_s \d s = \int_{F_0^c} \lambda^\varepsilon_s \d s = \sum_{k=0}^{m} \int_{b_k}^{a_{k+1}} \lambda^\varepsilon_s \d s. \]
For $0 \leq k \leq m$, let us consider $b_k < t_k < t'_k < a_{k+1}$ with $\Psi (\mu^\varepsilon_{t_k}) = \Psi (\mu^\varepsilon_{t'_k}) = \varepsilon$. 
Indeed, if such a $(t_k,t'_k)$ does not exist, $\int_{b_k}^{a_{k+1}} \lambda^\varepsilon_s \d s = 0$ from the complementary slackness condition \eqref{eq:CompSlackmf} with $\Psi -\varepsilon$ in place of $\Psi$.
We then write
\[ \int_{t_k}^{t'_k} \lambda^\varepsilon_s \d s = \int_{t_k}^{t'_k} \lambda^\varepsilon_s \d s - \int_{t_k}^{t'_k} \lambda^0_s \d s, \]
and we apply Lemma \ref{lem:locMax} at $t_k$ and $t'_k$, as in \emph{\textbf{Step 2.}}.
Using Lemma \ref{lem:timeReg2} on $[t_k,t'_k]$, reasoning as in \emph{\textbf{Step 1.}} shows that the r.h.s. is bounded by $C \varepsilon^{1/4}$.
Since this holds for every $[t_k,t'_k] \subset [b_k,a_{k+1}]$ with $\Psi (\mu^\varepsilon_{t_k}) = \Psi (\mu^\varepsilon_{t'_k}) = \varepsilon$, we deduce that 
\[ \int_{b_k}^{a_{k+1}} \lambda^\varepsilon_s \d s \leq C \varepsilon^{1/4}, \]
recalling the complementary slackness condition.
Gluing the steps together in \eqref{eq:DecompLambdaF} concludes. 
\end{proof}

\begin{rem}[Infinite number of intervals] \label{rem:InfiniteNumber}
Without assuming that the number of $[a_k,b_k]$ is finite, we can still write that
\[ \int_{F_\varepsilon \cap F^c_0} \lambda^\varepsilon_s \d s = \sum_{k \geq 0} \int_{I_k} \lambda^\varepsilon_s \d s, \] 
the sum being possibly countable with no more required ordering on the $I_k$. 
For each $k$, let us consider a maximal $[t_k,t'_k] \subset I_k$ such that $\Psi (\mu^\varepsilon_{t_k}) = \Psi (\mu^\varepsilon_{t'_k}) = \varepsilon$. 
Using Lemma \ref{lem:timeReg2} and \emph{\textbf{Step 1.}}, the problem boils down to controlling
\begin{equation} \label{eq:sumInfinInterv}
\sum_{k \geq 0} \frac{\d}{\d s}\big\rvert_{s = t_k} \Psi ( \mu^0_{s} ) - \frac{\d}{\d s}\big\rvert_{s = t'_k} \Psi ( \mu^0_{s} ).
\end{equation} 
From Lemma \ref{lem:timeReg2}, $s \mapsto \Psi (\mu^0_s)$ is $C^2$ on each $I_k$.
Let us assume that $\tfrac{\d^2}{\d s^2} \Psi ( \mu^0_s )$ has a uniform continuity modulus $m_\Psi$ on the union of these intervals: for any $\eta > 0$,
\[ \forall t,t' \in \cup_k I_k, \qquad \vert t - t' \rvert \leq m_\Psi ( \eta ) \Rightarrow \big\vert \tfrac{\d^2}{\d s^2} \vert_{s = t} \Psi ( \mu^0_s ) - \tfrac{\d^2}{\d s^2} \vert_{s = t'} \Psi ( \mu^0_s ) \big\vert \leq \eta. \]
We can then split the sum \eqref{eq:sumInfinInterv} between intervals of size smaller and bigger than $m_\Psi(\eta)$.
On intervals $I_k$ smaller than $m_\Psi(\eta)$, $\vert \tfrac{\d^2}{\d s^2} \Psi( \mu^0_s ) \vert$ is smaller than $\eta$ because $ \tfrac{\d^2}{\d s^2} \Psi( \mu^0_s ) $ vanishes within $I_k$; this is indeed a consequence of Rolle's theorem because $\tfrac{\d}{\d s} \Psi( \mu^0_s )$ vanishes at the boundary of $I_k$ (except possibly for the first and the last $I_k)$.  
On intervals bigger than $m_\Psi (\eta)$, we apply the error bound of \emph{\textbf{Step 3}}.
Since there are at most $\lfloor T m^{-1}_\Psi (\eta) \rfloor +1$ such intervals within $[0,T]$, we can bound \eqref{eq:sumInfinInterv} by
\begin{equation} \label{eq:sumInfinInterv2}
\frac{C}{m_\Psi (\eta)} \varepsilon^{1/4} + \sum_{a_{k+1} - b_k \leq m_\Psi(\eta)} \eta [ t'_{k} - t_k ] \leq C [ m^{-1}_\Psi (\eta) \varepsilon^{1/4} + \eta ] .
\end{equation} 
Such a continuity modulus $m_\Psi$ can be obtained from Lemma \ref{lem:timeReg2} by differentiating once again to get third-order time-regularity.
Strengthening a bit the regularity assumptions on coefficients, we could thus bound $\tfrac{\d^3}{\d s^3} \Psi ( \mu^0_s )$ and optimise over $\eta$ in \eqref{eq:sumInfinInterv2}, to get that \eqref{eq:sumInfinInterv} is of order $\varepsilon^{1/8}$.
\end{rem}

\begin{proof}[Proof of Theorem \ref{thm:stability}]
Let us first handle the convergence of the atom at $0$. 
If $\Psi( \mu^0_0 ) < 0$, then the same holds for $\Psi( \mu^\varepsilon_0 ) < 0$ for $\varepsilon$ small enough, hence $\lambda^\varepsilon_0 = \lambda^0_0 = 0$ and there is nothing to do.
We thus assume $\Psi( \mu^0_0 ) = 0$. 
Reasoning as in \emph{\textbf{Step 1.}} in the proof of Proposition \ref{pro:L1estiamte}, 
\[ \eta \lvert \lambda^\varepsilon_0 - \lambda^0_0 \rvert \leq \lvert \E [ \lambda^\varepsilon_0 \lvert \sigma^\top_0 \nabla \psi^\varepsilon_0 (X^\varepsilon_0) \rvert^2 - \lambda^0_0 \lvert \sigma^\top_0 \nabla \psi^0_0 (X^0_0) \rvert^2 ] \rvert + C \varepsilon^{1/2}. \]
The first term on the r.h.s can be computed using Lemma \ref{lem:timeReg2}, so that
\begin{multline} \label{eq:atom0}
\eta \lvert \lambda^\varepsilon_0 - \lambda^0_0 \rvert \leq \vert \E [ L^{\alpha^\varepsilon}_0 \psi^\varepsilon_0 ( X^\varepsilon_0 ) - L^{\alpha^0}_0 \psi^0_0 ( X^0_0 ) ] \vert + \bigg\vert \frac{\d}{\d t} \big\vert_{t =t'} \Psi(\mu^\varepsilon_t) \bigg\vert +\int_0^{t'} \vert \E [ r^\varepsilon_s ( X^\varepsilon_s ) - r^0_s ( X^0_s ) ] \vert \d s \\ + C \varepsilon^{1/2} + \int_0^{t'} \vert \E [ s^\varepsilon_s ( X^\varepsilon_s ) - s^0_s ( X^0_s ) ] \vert 
+ \vert \E [ \lambda^\varepsilon_s \vert \sigma_s^\top \nabla \psi^\varepsilon_s \vert^2 ( X^\varepsilon_s ) - \lambda^0_s \vert \sigma_s^\top \nabla \psi^0_s \vert^2 ( X^0_s ) ] \vert \d s, 
\end{multline}
for any $t' \in (0,T)$ with $\tfrac{\d}{\d t} \vert_{t =t'} \Psi(\mu^0_t) = 0$.
The $\vert \E [ r^\varepsilon_s ( X^\varepsilon_s ) - r^0_s ( X^0_s ) ] \vert$ term can be bounded using Lemma \ref{lem:boundPhi}-\ref{item:L2Bound}.
We deal with the $\vert \E [ s^\varepsilon_s ( X^\varepsilon_s ) - s^0_s ( X^0_s ) ] \vert$ as previously in \emph{\textbf{Step 1.}}. 
For the remaining terms, we distinguish between two cases.
\begin{enumerate}
\item Let us assume that $t \mapsto \Psi( \mu^0_t )$ has a global minimum at $t' \in (0,T)$, so that $\tfrac{\d}{\d t} \vert_{t =t'} \Psi(\mu^0_t) = 0$.
Using Lemma \ref{lem:locMax}, we get that $\vert \tfrac{\d}{\d t} \vert_{t =t'} \Psi(\mu^\varepsilon_t) \vert \leq C \varepsilon^{1/4}$.
It remains to control the difference $\E [ L^{\alpha^\varepsilon}_{0} \psi^\varepsilon_{0}({X}^\varepsilon_{0}) ] - \E [ L^{\alpha^0}_{0} \psi^0_{0}({X}^0_{0}) ]$. Since
\[ L^{\alpha^\varepsilon}_0 \psi^\varepsilon_0 = L^0_0 \psi^\varepsilon_0 + a_0 \nabla \varphi^\varepsilon_0 \cdot \nabla \psi^\varepsilon_0, \qquad X^\varepsilon_0 \sim Z^{-1}_\varepsilon e^{-\varphi^\varepsilon_0} \d \nu_0, \]
we notice that
\begin{align*}
\E [ L^{\alpha^\varepsilon}_{0} \psi^\varepsilon_{0}({X}^\varepsilon_{0}) ] = \E [ L^{0}_{0} \psi^\varepsilon_{0}({X}^\varepsilon_{0}) ] + Z_\varepsilon^{-1} \int_{\R^d} - a_0 \nabla [ e^{- \varphi^\varepsilon_0} ] \cdot \nabla \psi^\varepsilon_0 \, \d \nu_0. 
\end{align*}
Integrating by parts:
\[ Z_\varepsilon^{-1} \int_{\R^d} - a_0 \nabla [ e^{- \varphi^\varepsilon_0} ] \cdot \nabla \psi^\varepsilon_0 \, \d \nu_0 = \E [ \nabla \cdot ( a_0 \nabla \psi^\varepsilon_0 )({X}^\varepsilon_{0}) ] + \E [ a_0 \nabla \psi^\varepsilon_0 \cdot \nabla \log \nu_0({X}^\varepsilon_{0}) ]. \]
The first term on the r.h.s generates the difference $\E [ \nabla \cdot ( a_0 \nabla \psi^\varepsilon_0 ) ( X^\varepsilon_0 ) - \nabla \cdot ( a_0 \nabla \psi^0_0 ) ( X^0_0 ) ]$ which enters the scope of Lemma \ref{lem:L2stabImp}.
To handle the difference of the second second terms, we first write that
\[ \E [ \vert \nabla \log \nu_0 (X^\varepsilon_0) - \nabla \log \nu_0 (X^0_0) \vert \mathbbm{1}_{\vert \nabla \log \nu_0 (X^\varepsilon_0) \vert + \vert \nabla \log \nu_0 (X^0_0) \vert \leq M} ] \leq C M \varepsilon^{1/2}, \]
using Lemma \ref{lem:L2stab}-\ref{item:entropBound} and Pinsker's inequality. We then use that
\begin{multline*}
\E [ \vert \nabla \log \nu_0 (X^\varepsilon_0) - \nabla \log \nu_0 (X^0_0) \vert \mathbbm{1}_{\vert \nabla \log \nu_0 (X^\varepsilon_0) \vert + \vert \nabla \log \nu_0 (X^0_0) \vert > M} ] \\
\leq M^{-1} \{ \E [ \vert \nabla \log \nu_0 (X^\varepsilon_0) \vert^2 + \vert \nabla \log \nu_0 (X^0_0) \vert^2  ] \}^{3/2}, 
\end{multline*}
using the Cauchy-Schwarz and the Markov inequalities.
We further bound the expectation using \ref{ass:ini2} and Lemma \ref{lem:InitalEntrop}. 
Optimising in $M$, we get that
\[ \E [ \vert \nabla \log \nu_0 (X^\varepsilon_0) - \nabla \log \nu_0 (X^0_0) \vert ] \leq C \varepsilon^{1/4}. \]
Gathering everything, we deduce that 
\[ \vert \E [ L^{\alpha^\varepsilon}_0 \psi^\varepsilon_0 ( X^\varepsilon_0 ) - L^{\alpha^0}_0 \psi^0_0 ( X^0_0 ) ] \vert \leq C \varepsilon^{1/4}. \]
Going back to \eqref{eq:atom0} and
using Lemma \ref{lem:L2stabImp}, we eventually obtain that $\vert \lambda^\varepsilon_0 - \lambda^0_0 \vert \leq C \varepsilon^{1/4}$.
\item If such a minimiser $t'$ does not exist in $(0,T)$, $t \mapsto \Psi( \mu^0_t )$ is necessarily (strictly) decreasing hence $\Psi(\mu^0_T) < 0$, and $\Psi(\mu^\varepsilon_T) < 0$ for small enough $\varepsilon$. 
We can then set $t' = T$, and \eqref{eq:atom0} still holds using Lemma \ref{lem:timeReg2} because
\[ \frac{\d}{\d t} \big\vert_{t =T} \Psi(\mu^\varepsilon_t) = \lambda^\varepsilon_T \E [ \vert \sigma_T^\top \nabla \psi^\varepsilon_T \vert^2 ( X^\varepsilon_T ) ] = 0, \] 
from the complementary slackness.
Reasoning as above then yields $\vert \lambda^\varepsilon_0 - \lambda^0_0 \vert \leq C \varepsilon^{1/4}$.
\end{enumerate}
Since we now control $\vert \lambda^\varepsilon_0 - \lambda^0_0 \vert$ and $\lVert \lambda^\varepsilon - \lambda^0 \rVert_{L^1(0,T)}$, Corollary \ref{cor:timeReg1}-\ref{item:PsiC1} and Lemma \ref{lem:timeReg2} yield
\[ \forall t \in [0,T), \quad \bigg\lvert \frac{\d}{\d t} \Psi (\mu^\varepsilon_t) - \frac{\d}{\d t} \Psi (\mu^0_t) \bigg\rvert \leq C \varepsilon^{1/4}, \]
for $C$ independent of $(t,\varepsilon)$.
From Corollary \ref{cor:timeReg1}-\ref{item:PsiC1}, $t \mapsto \Psi (\mu^\varepsilon_t)$ is $C^1$ hence this identity still holds for $t = T$. 
Since $\tfrac{\d}{\d t} \big\vert_{t =T} \Psi(\mu^\varepsilon_t) = \lambda^\varepsilon_T \E [ \vert \sigma_T^\top \nabla \psi^\varepsilon_T \vert^2 ( X^\varepsilon_T ) ]$, we obtain using \eqref{eq:Qualifeps} that $\vert \lambda^\varepsilon_T - \lambda^\varepsilon_0 \vert \leq C \varepsilon^{1/4}$.

Since $\lambda^\varepsilon - \lambda^0$ is now fully estimated, we can introduce $u^\varepsilon := \varphi^\varepsilon - \varphi^0$, which satisfies
\[ - u^\varepsilon_t + \int_t^T b^\varepsilon_s \cdot \nabla u^\varepsilon_s + \frac{1}{2} \mathrm{Tr}[a_s \nabla^2 u^\varepsilon_s] \d s + \int_{[t,T]} \psi^\varepsilon_s \lambda^\varepsilon ( \d s) - \int_{[t,T]} \psi^0_s \lambda^0 ( \d s) = 0, \]
in the sense of Definition \ref{def:mildForm},
where $b^\varepsilon_t := b_t - \tfrac{1}{2} a_t \nabla [ \varphi^\varepsilon_t + \varphi^0_t ]$.
As for Proposition \ref{pro:hessbound}, we can use a regularisation procedure to obtain the Feynman-Kac formula
\[ u^\varepsilon_t (x) = \E \bigg[ \int_{[t,T]} \psi^\varepsilon_s ( Y^{\varepsilon,t,x}_s ) \lambda^\varepsilon ( \d s ) - \int_{[t,T]} \psi^0_s ( Y^{\varepsilon,t,x}_s ) \lambda^0 ( \d s ) \bigg], \]
using the solution to the following SDE,
\[ \d Y^{\varepsilon,t,x}_s = b^\varepsilon_s ( Y^{\varepsilon,t,x}_s ) \d s + \sigma_s ( Y^{\varepsilon,t,x}_s ) \d B_s, \quad Y^{\varepsilon,t,x}_t = x. \]
We can then estimate $\vert u^\varepsilon_t (x) - u^\varepsilon_t (y) \vert$ using the gradient estimate from \cite{priola2006gradient} as in Proposition \ref{pro:hessbound}.
Indeed, $\nabla b^\varepsilon_t$ is bounded uniformly in $(\varepsilon, t)$ using \eqref{eq:BoundsUnifeps}, hence we get uniform in $\varepsilon$ estimates.
Using \eqref{eq:StabPartisalPsi} and our estimate of $\lambda^\varepsilon - \lambda^0$, this yields $\sup_{(t,x) \in [0,T] \times \R^d} \vert \nabla u^\varepsilon_t (x) \vert \leq C \varepsilon^{1/4}$ as desired.
\end{proof}

\appendix

\section{Time-reversal and density estimates} \label{app:timerev}

In this section, we prove bounds on the density $\mu_t$ of the solution to
\[ \d X_t = b_t ( X_t ) \d t + \sigma_t ( X_t ) \d B_t.  \]
Lemmata \ref{lem:Gamma_1} and \ref{lem:Gamma_2} can be rephrased in terms of entropy and Fisher information: they are non-quantitative versions of results from \cite{bakry2014analysis}.
Contrary to the classical semi-group approach, our estimates rely on semi-martingale decompositions in the spirit of \cite{fontbona2016trajectorial}. 
They provide probabilistic counterparts to PDE results of \cite[Chapter 7.4]{bogachev2022fokker}.
To the best of our knowledge, the third order estimate in Proposition \ref{pro:Gamma_3} is new.

We work under global Lipschitz assumptions which are not optimal, see \cite{haussmann1986time,fontbona2016trajectorial} for sharper discussions.
From \cite[Theorem 3.1]{haussmann1986time}, \ref{ass:coefReg1} is enough for existence of a density $\mu_t \in H^1(\R^d)$.
In the forthcoming proofs, we assume that coefficients are smooth so that $\mu_t$ is the smooth positive solution of $\partial_t \mu_t = L^\star_t \mu_t$.
To get the estimates with the desired regularity, it is classically sufficient to regularise coefficients and to take the limit.  

We define $\overleftarrow{X}_t := X_{T-t}$.
From \cite[Theorem 2.1]{haussmann1986time}, $( \overleftarrow{X}_t )_{0 \leq t \leq T}$ is a time-inhomogeneous Markov process with generator
\[  \overleftarrow{L}_t := \overleftarrow{b}^j_t \partial_j + \frac{1}{2} \overleftarrow{a}^{j,k}_t \partial^2_{j,k}, \]
where
\[ \overleftarrow{b}^j_t := - b^j_{T-t} + \mu_{T-t}^{-1} \partial_k ( a_{T-t}^{j,k} \mu_{T-t} ), \qquad \overleftarrow{a}^{j,k}_t := a^{j,k}_{T-t}. \]
Let $\overleftarrow{\mu}_t := \mu_{T-t}$ be the density of the law of $\overleftarrow{X}_t$ at time $t$.
We then set $\eta_t := \log \overleftarrow{\mu}_t$.
After simplifications, the Fokker-Planck equation $\partial_t \overleftarrow{\mu}_t = \overleftarrow{L}^\star_t \overleftarrow{\mu}_t$ yields
\begin{equation}\label{eq:FKeta}
 (\partial_t + \overleftarrow{L}_t) \eta_t = \partial_j b^j_{T-t} - \frac{1}{2} \partial_{j,k} \overleftarrow{a}^{j,k}_t + \frac{1}{2} \overleftarrow{a}^{j,k}_t \partial_j \eta_t \partial_k \eta_t. 
\end{equation} 
We then define $Y^i_t := \partial_i \eta_t (\overleftarrow{X}_t)$, for $1 \leq i \leq d$. 
In all what follows, we underly evaluation at $\overleftarrow{X}_t$, writing e.g. $\d \eta_t$ instead of $\d \eta_t (\overleftarrow{X}_t)$.

\begin{lemma} \label{lem:Gamma_1}
Under \ref{ass:coefReg1}, if $\mu_0$ satisfies \ref{ass:ini1} then 
\[ \sup_{t \in [0,T]} \int_{\R^d} \log \mu_t \, \d \mu_t +  \int_0^T \int_{\R^d} \lvert \nabla \log \mu_t \rvert^2 \d \mu_t \d t \leq C, \]
where $C>0$ only depends on coefficients through their uniform norm or the one or their derivatives.
\end{lemma}

\begin{proof} 
Using Ito's formula and \eqref{eq:FKeta}, 
\begin{equation} \label{eq:diffeta}
\d \eta_t = [  \partial_j b^j_{T-t} + \tfrac{1}{2} \overleftarrow{a}^{j,k}_t Y^j_t Y^k_t - \tfrac{1}{2} \partial_{j,k} \overleftarrow{a}^{j,k}_t ] \d t + Y^j_t \overleftarrow{\sigma}^{j,k}_t \d B^k_t. 
\end{equation}
Let $\rho (x):= Z^{-1} \exp [-\sqrt{x^2 +1}]$ be a probability density w.r.t. Lebesgue. 
A similar computation yields
\begin{multline*}
\d \log \tfrac{\d \overleftarrow{\mu}_t}{\d \rho} = \tfrac{1}{2} \overleftarrow{a}^{j,k}_t Y^j_t Y^k_t \d t + \partial_j \log \tfrac{\d \overleftarrow{\mu}_t}{\d \rho} \overleftarrow{\sigma}^{j,k}_t \d B^k_t \\
+ [  \partial_j b^j_{T-t} - \tfrac{1}{2} \partial_{j,k} \overleftarrow{a}^{j,k}_t + b^j_{T-t} \partial_j \log \rho - \tfrac{1}{2} \partial_k \overleftarrow{a}^{j,k}_t \partial_j \log \rho - \overleftarrow{a}^{j,k}_t Y^k_t \partial_j \log \rho
+ \tfrac{1}{2} \overleftarrow{a}^{j,k}_t \partial_{j,k} \log \rho] \d t.  
\end{multline*} 
We then integrate in time and we take expectations. 
We integrate by parts to get rid of second order derivatives on $\overleftarrow{a}^{j,k}_t$:
\[ \E[- \tfrac{1}{2} \partial_{j,k} \overleftarrow{a}^{j,k}_t ] = \E[ \tfrac{1}{2} \partial_j \overleftarrow{a}^{j,k}_t Y^k_t ] \geq - \tfrac{\varepsilon}{2} \E [ Y^k_t Y^k_t] -\tfrac{\varepsilon^{-1}}{2} \E [ \partial_j \overleftarrow{a}^{j,k}_t \partial_j \overleftarrow{a}^{j,k}_t ], \]
for $\varepsilon$ small enough.
Using \ref{ass:ini1} on $\mu_0$, $\int \log \frac{\d \overleftarrow{\mu}_T}{\d \rho} \d \overleftarrow{\mu}_T =  \int \log \frac{\d {\mu}_0}{\d \rho} \d {\mu}_0$ is finite.
Similarly,
\begin{equation} \label{eq:Young}
\E [ - \overleftarrow{a}^{j,k}_t  Y^k_t \partial_j \rho ] \geq - \tfrac{\varepsilon}{2} \E [ Y^k_t Y^k_t ] - \tfrac{\varepsilon^{-1}}{2} \E [ \overleftarrow{a}^{j,k}_t  \partial_j \rho \overleftarrow{a}^{j,k}_t \partial_j \rho ]. 
\end{equation} 
Using \ref{ass:coefReg1}-$(iii)$, $\overleftarrow{a}^{j,k}_t Y^j_t Y^k_t \geq C^{-1} Y^k_tY^k_t$ for some $C > 0$, and we choose $\varepsilon$ such that $C^{-1} > 2 \varepsilon$: we thus obtain a bound on $\E [ \overleftarrow{a}^{j,k}_t Y^j_t Y^k_t]$.
This in turn gives a bound on $\int \log \frac{\d \overleftarrow{\mu}_0}{\d \rho} \d \overleftarrow{\mu}_0 =  \int \log \frac{\d {\mu}_T}{\d \rho} \d {\mu}_T$ follows, and then on $\int \log \mu_T \d \mu_T$.
Integrating and taking expectation in \eqref{eq:diffeta}, the bound on $\int \log \mu_t \d \mu_t = \E [ \eta_{T-t}]$ follows by a Gronwall argument.
\end{proof}

We now differentiate \eqref{eq:FKeta}.
After simplifications:
\begin{multline} \label{eq:DiffFK}
(\partial_t + \overleftarrow{L}_t) \partial_i \eta_t = \partial_i b^j_{T-t} \partial_{j} \eta_t -\partial_{i,k} \overleftarrow{a}^{j,k}_t \partial_{j} \eta_t - \tfrac{1}{2} \partial_i \overleftarrow{a}^{j,k}_t \partial_{k} \eta_t \partial_j \eta_t
- \tfrac{1}{2} \partial_i \overleftarrow{a}^{j,k}_t \partial_{j,k} \eta_t \\
+
\partial_{i,j} b^j_{T-t} - \tfrac{1}{2} \partial_{i,j,k} \overleftarrow{a}^{j,k}_t. 
\end{multline}

Let us define $Z^{j,k}_t := \partial_{j,k} \eta_t (\overleftarrow{X}_t)$.

\begin{lemma} \label{lem:Gamma_2}
Under \ref{ass:coefReg1} and \ref{ass:coefReg2}-(i), if $\mu_0$ satisfies \ref{ass:ini1} then 
\[ \sup_{t \in [0,T]} \int_{\R^d} \vert \nabla \log \mu_t \vert^2 \, \d \mu_t +  \int_0^T \int_{\R^d} \lvert \nabla^2 \log \mu_t \rvert^2 \d \mu_t \d t \leq C, \]
where $C>0$ only depends on coefficients through their uniform norm or the one or their derivatives. 
\end{lemma}

\begin{proof}
Using Ito's formula and \eqref{eq:DiffFK}, 
\begin{multline*}
\d Y^i_t =  [ \partial_i b^j_{T-t} Y^j_t + \partial_{i,j} b^j_{T-t} - \tfrac{1}{2} \partial_{i,j,k} \overleftarrow{a}^{j,k}_t - \partial_{i,k} \overleftarrow{a}^{j,k}_t Y^j_t ] \d t + Z^{i,j}_t \overleftarrow{\sigma}^{j,k}_t \d B^k_t \\
- \tfrac{1}{2} [ \partial_i \overleftarrow{a}^{j,k}_t Z^{j,k}_t + \partial_i \overleftarrow{a}^{j,k}_t Y^{j}_t Y^{k}_t ] \d t.
\end{multline*}
On the other hand
\begin{multline*}
\d \overleftarrow{\sigma}^{j,i}_t = [ \partial_t \overleftarrow{\sigma}^{j,i}_t - b^k_{T-t} \partial_k \overleftarrow{\sigma}^{j,i}_t + \partial_l \overleftarrow{a}^{k,l}_t \partial_k \overleftarrow{\sigma}^{j,i}_t + \overleftarrow{a}^{k,l}_t Y^l_t \partial_k \overleftarrow{\sigma}^{j,i}_t + \tfrac{1}{2} \overleftarrow{a}^{k,l}_t \partial_{k,l} \overleftarrow{\sigma}^{i,j}_t ] \d t \\ + \partial_k \overleftarrow{\sigma}^{j,i}_t \overleftarrow{\sigma}_{k,l} \d B_l. 
\end{multline*} 
We then define $\overline{Y}^i_t := \overleftarrow{\sigma}^{j,i}_t Y^j_t$, so that
\begin{align*}
\d \overline{Y}^i_t &= Y^j_t \d \overleftarrow{\sigma}^{j,i}_t + \overleftarrow{\sigma}^{j,i}_t \d Y^j_t + \d [ \overleftarrow{\sigma}^{j,i} , Y^j ]_t \\
&= [ \partial_t \overleftarrow{\sigma}^{j,i}_t Y^j_t + \textbf{B}^i_t + {\textbf{Y}}^i_t + \overline{\textbf{Y}}^i_t + \textbf{Z}^i_t ] \d t + [ Y^j_t \partial_k \overleftarrow{\sigma}^{j,i}_t \overleftarrow{\sigma}^{k,l}_t + \overleftarrow{\sigma}^{j,i}_t Z^{j,k}_t \overleftarrow{\sigma}^{k,l}_t ] \d B^l_t,    
\end{align*}
with 
\[ \textbf{B}^i_t := - b^k_{T-t} \partial_k \overleftarrow{\sigma}^{j,i}_t Y^j + \overleftarrow{\sigma}_t^{j,i} [ \partial_j b^k_{T-t} Y^k_t + \partial_{j,k} b^k_{T-t} ], \]
\[ \textbf{Z}^i_t := \partial_k \overleftarrow{\sigma}^{j,i}_t \overleftarrow{\sigma}_{k,l} Z^{j,m} \overleftarrow{\sigma}_t^{m,l} - \tfrac{1}{2} \overleftarrow{\sigma}_t^{j,i} \partial_j \overleftarrow{a}^{k,l}_t Z^{k,l}_t,  \quad \overline{\textbf{Y}}^i_t := Y^j_t \overleftarrow{a}^{k,l}_t Y^l_t \partial_k \overleftarrow{\sigma}^{j,i}_t - \tfrac{1}{2} \overleftarrow{\sigma}^{j,i}_t  \partial_j \overleftarrow{a}^{k,l}_t Y^{k}_t Y^{l}_t, \]
\[ {\textbf{Y}}^i_t := Y^j_t [ \partial_l \overleftarrow{a}^{k,l}_t \partial_k \overleftarrow{\sigma}^{j,i}_t + \tfrac{1}{2} \overleftarrow{a}^{k,l}_t \partial_{k,l} \overleftarrow{\sigma}^{i,j}_t ] - \tfrac{1}{2} \overleftarrow{\sigma}^{j,i}_t \partial_{j,k,l} \overleftarrow{a}^{k,l}_t - \overleftarrow{\sigma}^{j,i}_t \partial_{j,k} \overleftarrow{a}^{k,l}_t Y^l_t. \]
We now want to compute
\[ d [ \overline{Y}^i_t \overline{Y}^i_t ] = 2 \overline{Y}^i_t \d \overline{Y}^i_t + \sum_{i,l = 1}^d [ Y^j_t \partial_k \overleftarrow{\sigma}^{j,i}_t \overleftarrow{\sigma}^{k,l}_t + \overleftarrow{\sigma}^{j,i}_t Z^{j,k}_t \overleftarrow{\sigma}^{k,l}_t ]^2. \]
We first observe that
\[ \overline{Y}^i_t \overline{\textbf{Y}}^i_t = \partial_k \overleftarrow{\sigma}^{j,i}_t \overleftarrow{\sigma}^{m,i}_t \overleftarrow{a}^{k,l}_t Y^j_t Y^l_t Y^m_t - \tfrac{1}{2} \overleftarrow{\sigma}^{m,i}_t \overleftarrow{\sigma}^{j,i}_t \partial_j \overleftarrow{a}^{k,l}_t Y^{k}_t Y^{l}_t Y^m_t. \]
We then write that $\overleftarrow{\sigma}^{m,i}_t \overleftarrow{\sigma}^{j,i}_t \partial_j \overleftarrow{a}^{k,l}_t = \overleftarrow{a}^{j,m}_t ( \partial_j \overleftarrow{\sigma}^{k,i}_t \overleftarrow{\sigma}^{l,i}_t + \overleftarrow{\sigma}^{k,i}_t \partial_j \overleftarrow{\sigma}^{l,i}_t)$, and we re-arrange terms to get that $\overline{Y}^i_t \overline{\textbf{Y}}^i_t = 0$.
Using \ref{ass:coefReg1}-$(iii)$, we get 
\[ \sum_{i,l = 1}^d [ Y^j_t \partial_k \overleftarrow{\sigma}^{j,i}_t \overleftarrow{\sigma}^{k,l}_t + \overleftarrow{\sigma}^{j,i}_t Z^{j,k}_t \overleftarrow{\sigma}^{k,l}_t ]^2 \geq - C Y^j_t Y^j_t + D Z^{j,l}_t Z^{j,l}_t, \]
for some $C, D > 0$. 
As for proving Lemma \ref{lem:Gamma_1}, we take expectation and we integrate by parts to lower derivative orders:
\begin{multline*} 
\E [- \overline{{Y}}^i_t \overleftarrow{\sigma}_t^{j,i} [ \partial_j b^k_{T-t} Y^k_t + \partial_{j,k} b^k_{T-t} ] ] = \E [ \overline{{Y}}^i_t \partial_{k} \overleftarrow{\sigma}^{j,i}_t \partial_{j} b^{k}_{T-t} + \partial_k \overleftarrow{\sigma}^{m,i}_t Y^m_t \overleftarrow{\sigma}^{j,i}_t \partial_{j} \overleftarrow{b}^{k}_{T-t} \\
+ \overleftarrow{\sigma}^{m,i}_t Z^{j,k}_t \overleftarrow{\sigma}^{j,i}_t \partial_{j} \overleftarrow{b}^{k}_{T-t}],
\end{multline*}
and
\begin{multline*} 
\E [- \overline{{Y}}^i_t \overleftarrow{\sigma}^{j,i}_t \partial_{j,k,l} \overleftarrow{a}^{k,l}_t ] = \E [ \overline{{Y}}^i_t \overleftarrow{\sigma}^{j,i}_t \partial_{j,k} \overleftarrow{a}^{k,l}_t Y^l_t + \overline{{Y}}^i_t \partial_{l} \overleftarrow{\sigma}^{j,i}_t \partial_{j,k} \overleftarrow{a}^{k,l}_t + \partial_l \overleftarrow{\sigma}^{m,i}_t Y^m_t \overleftarrow{\sigma}^{j,i}_t \partial_{j,k} \overleftarrow{a}^{k,l}_t \\
+ \overleftarrow{\sigma}^{m,i}_t Z^{m,l}_t \overleftarrow{\sigma}^{j,i}_t \partial_{j,k} \overleftarrow{a}^{k,l}_t ]. 
\end{multline*}
To handle product terms, we repeatedly use $\overline{Y}^i_t Z^{j,k}_t \geq -\tfrac{\varepsilon}{2} (Z^{j,k}_t)^2 + \tfrac{\varepsilon^{-1}}{2} (\overline{Y}^i_t)^2$ as in \eqref{eq:Young}, for $\varepsilon$ small enough compared to $D$.
Using bounds \ref{ass:coefReg1} and \ref{ass:coefReg2}-$(i)$ on coefficients, we eventually get
\begin{equation} \label{eq:GronwallY}
\tfrac{\d}{\d t} \E [ \overline{Y}^i_t \overline{Y}^i_t ] \geq - \alpha ( 1 + \E [ Y^i_t Y^i_t ] ) + \beta \E [ Z^{j,l}_t Z^{j,l}] \geq - \gamma ( 1 + \E [ \overline{Y}^i_t \overline{Y}^i_t ] ) + \beta \E [ Z^{j,l}_t Z^{j,l}], 
\end{equation} 
for some $\alpha, \beta, \gamma > 0$.
Integrating in time, we conclude as in the proof of Lemma \ref{lem:Gamma_1}. 
\end{proof}

\begin{lemma} \label{lem:Gamma_2.5}
Under \ref{ass:coefReg1}, \ref{ass:coefReg2}-(i), if $\mu_0$ satisfies \ref{ass:ini1} and $\int \vert \nabla \log \mu_0 \rvert^4 \d \mu_0 < +\infty$ then 
\begin{itemize}
\item [(i)] $\sup_{t \in [0,T]} \E [ \vert Y_t \rvert^4 ] = \sup_{t \in [0,T]} \int \vert \nabla \log \mu_t \rvert^4 \d \mu_t \leq C$,
\item[(ii)] $\E [ \int_0^T  Y^j_t Y^j_t Z^{i,l}_t Z^{i,l}_t ] \leq C$,
\end{itemize}
where $C>0$ only depends on coefficients through their uniform norm or the one or their derivatives. 
\end{lemma}

\begin{proof}
We apply the same scheme as before writing that
\[ \d [ \overline{Y}^j_t \overline{Y}^j_t \overline{Y}^l_t \overline{Y}^l_t ] = 2 \overline{Y}^j_t \overline{Y}^j_t \d [ \overline{Y}^l_t \overline{Y}^l_t ] + \d [ \overline{Y}^j \overline{Y}^j , \overline{Y}^l \overline{Y}^l ]_t \geq 2 \overline{Y}^j_t \overline{Y}^j_t \d [ \overline{Y}^l_t \overline{Y}^l_t ],  \]
and we follow the same computations that yield to \eqref{eq:GronwallY} in the proof of Lemma \ref{lem:Gamma_2}. 
In particular, we lower bound product terms by writing that $(\overline{Y}^j_t)^2 Z^{i,l}_t \geq - \tfrac{\varepsilon^{-1}}{2} (\overline{Y}^j_t)^4 + \tfrac{\varepsilon}{2} (Z^{i,l}_t)^2$ for small enough $\varepsilon > 0$.
We end up to
\begin{equation*}
\tfrac{\d}{\d t} \E [ \overline{Y}^j_t \overline{Y}^j_t \overline{Y}^l_t \overline{Y}^l_t ] \geq - \alpha \E [ ( 1 + Y^i_t Y^i_t )^2 ]  + \beta \E [ Y^j_t Y^j_t Z^{i,l}_t Z^{i,l}], 
\end{equation*} 
for some $\alpha, \beta > 0$.
A Gronwall argument now gives the bound on $\E [ \overline{Y}^j_t \overline{Y}^j_t \overline{Y}^l_t \overline{Y}^l_t ]$.
The bound on $\E [ Y^j_t Y^j_t Z^{i,l}_t Z^{i,l}_t ]$ a posteriori follows.
\end{proof}

We now differentiate \eqref{eq:DiffFK}.
After computations, we get that $(\partial_t + \overleftarrow{L}_t) \partial_{i,l} \eta_t$ equals
\begin{multline*}\label{eq:DDiffFKeta}
- (\partial_l \overleftarrow{L}_t) \partial_i \eta_t + \partial_{i,j,l} b^j_{T-t} + \partial_{i,l} b^j_{T-t} \partial_{j} \eta_t + \partial_{i} b^j_{T-t} \partial_{l,j} \eta_t - \tfrac{1}{2} \partial_{i,j,k,l} \overleftarrow{a}^{j,k}_t - \partial_{i,k} \overleftarrow{a}^{j,k}_t \partial_{j,l} \eta_t \\
-\partial_{i,k,l} \overleftarrow{a}^{j,k}_t \partial_{j} \eta_t
- \tfrac{1}{2} \partial_{i,l} \overleftarrow{a}^{j,k}_t \partial_{k} \eta_t \partial_j \eta_t
-\partial_i \overleftarrow{a}^{j,k}_t \partial_{j,l} \eta_t \partial_k \eta_t -\tfrac{1}{2} \partial_i \overleftarrow{a}^{j,k}_t \partial_{j,k,l} \eta_t - \tfrac{1}{2} \partial_{i,l} \overleftarrow{a}^{j,k}_t \partial_{j,k} \eta_t, 
\end{multline*} 
where
\begin{multline*}
- (\partial_l \overleftarrow{L}_t) \partial_i \eta_t = \partial_l b^k_{T-t} \partial_{i,k} \eta_t - \partial_{j,l} \overleftarrow{a}^{j,k}_t \partial_{i,k} \eta_t - \partial_l \overleftarrow{a}^{j,k}_t \partial_j \eta_t \partial_{i,k} \eta_t - \tfrac{1}{2} \partial_l \overleftarrow{a}^{j,k}_t \partial_{i,j,k} \eta_t \\
- \overleftarrow{a}^{j,k}_t  \partial_{j,l} \eta_t \partial_{i,k} \eta_t.
\end{multline*} 
We define $\Gamma^{i,j,k}_t := \partial_{i,j,k} \eta_t (\overleftarrow{X}_t)$.

\begin{proposition} \label{pro:Gamma_3}
Under \ref{ass:ini2} and \ref{ass:coefReg2}, for every $t \in [0,T]$,
\[ \sup_{s \in [0,t]} \int_{\R^d} \vert \nabla^2 \log \mu_s \vert^2 \, \d \mu_s +  \int_0^t \int_{\R^d} \lvert \nabla^3 \log \mu_s \rvert^2 \d \mu_s \d s \leq C \big[ 1 + \sup_{s \in [0,t]} \lVert \nabla^2 b_s \rVert_\infty^2 \big], \]
where $C>0$ does not depend on 
$\lVert \nabla^2 b_s \rVert_\infty$, and $C$ only depends on coefficients through their uniform norm or the one or their derivatives.
\end{proposition}

\begin{proof}
Using Ito's formula and the above expression for $(\partial_t + \overleftarrow{L}_t) \partial_{i,l} \eta_t$,
\[ \d Z^{i,l}_t = [ \partial_{i,j,l} b^j_{T-t} -\tfrac{1}{2} \partial_{i,j,k,l} a^{j,k}_t + \mathbf{Y}^{i,l}_t + \mathbf{Z}^{i,l}_t + \mathbf{\Gamma}^{i,l}_t - \overleftarrow{a}^{j,k}_t Z^{j,l}_t Z^{k,i}_t ] \d t + \Gamma^{i,j,l}_t \overleftarrow{\sigma}^{j,k}_t \d B^k_t, \]
where
\begin{multline*}
\mathbf{Y}^{i,l}_t := \partial_{i,l} b^j_{T-t} Y^j_t-\tfrac{1}{2} \partial_{i,l} \overleftarrow{a}^{j,k}_t Y^j_t Y^k_t - \partial_{i,k,l} \overleftarrow{a}^{j,k}_t Y^j, \quad \mathbf{\Gamma}^{i,l}_t = -\tfrac{1}{2} \partial_i \overleftarrow{a}^{j,k}_t \Gamma^{j,k,l}_t - \tfrac{1}{2} \partial_l \overleftarrow{a}^{j,k}_t \Gamma^{i,j,k}_t, \\
\mathbf{Z}^{i,l}_t := \partial_{i} b^j_{T-t} Z^{j,l}_t + \partial_l b^k_{T-t} Z^{i,k}_t -\partial_{j,l} \overleftarrow{a}^{j,k}_t Z^{i,k}_t - \partial_l \overleftarrow{a}^{j,k}_t Y^j_t Z^{i,k}_t - \partial_i \overleftarrow{a}^{j,k}_t Z^{j,l}_t Y^k_t \\
- \partial_{i,k} \overleftarrow{a}^{j,k}_t Z^{j,l}_t - \tfrac{1}{2} \partial_{i,l} \overleftarrow{a}^{j,k}_t Z^{j,k}_t. 
\end{multline*}
Therefore,
\begin{multline*} 
\d Z^{i,l}_t Z^{i,l}_t = 2 Z^{i,l}_t [ \partial_{i,j,l} b^j_{T-t} -\tfrac{1}{2} \partial_{i,j,k,l} a^{j,k}_t + \mathbf{Y}^{i,l}_t + \mathbf{Z}^{i,l}_t + \mathbf{\Gamma}^{i,l}_t - \overleftarrow{a}^{j,k}_t Z^{j,l}_t Z^{k,i}_t ] \d t \\
+ \Gamma^{i,j,l}_t \overleftarrow{\sigma}^{j,k}_t \Gamma^{i,l,m} \overleftarrow{\sigma}^{m,k} \d t + 2 \Gamma^{i,j,l}_t \overleftarrow{\sigma}^{j,k}_t \d B^k_t.
\end{multline*}
Using \ref{ass:coefReg1}-$(iii)$, $\Gamma^{i,j,l}_t \overleftarrow{\sigma}^{j,k}_t \Gamma^{i,l,m} \overleftarrow{\sigma}^{m,k} \geq C \Gamma^{i,j,l}_t \Gamma^{i,j,l}_t$
for some $C>0$. Then using \ref{ass:coefReg2},
\[ 2 Z^{i,l}_t [ \mathbf{Y}^{i,l}_t +\mathbf{Z}^{i,l}_t ] \geq - D [ 1 + Y^j_t Y^j_t] Z^{i,l}_t Z^{i,l}_t, \quad \text{and} \quad 2 Z^{i,l}_t \mathbf{\Gamma}^{i,l}_t \geq - \varepsilon \Gamma^{i,j,k}_t \Gamma^{i,j,k}_t - \varepsilon^{-1} Z^{i,l}_t Z^{i,l}_t, \]
for $D >0$ and $\varepsilon > 0$ small enough.
We now take expectation.
To lower derivative orders, we integrate by parts:
\begin{multline*}
\E \{ Z^{i,l}_t [ \partial_{i,j,l} b^j_{T-t} - \tfrac{1}{2} \partial_{i,j,k,l} a^{j,k}_t ] \} = -\E \{ [\Gamma^{i,j,l}_t + Y^j_t Z^{i,l}_t ][ \partial_{i,l} b^j_{T-t} - \tfrac{1}{2} \partial_{i,k,l} \overleftarrow{a}^{j,k}_t ] \} \\
\geq - \E [ \varepsilon \Gamma^{i,j,l}_t \Gamma^{i,j,l}_t + \varepsilon^{-1} D \lVert \nabla^2 b_{T-t} \rVert^2_\infty + \varepsilon^{-1} D [ 1 + Y^j_t Y^j_t] Z^{i,l}_t Z^{i,l}_t ].  
\end{multline*}
for some $D > 0$ that does not depend on $\lVert \nabla^2 b_{T-t} \rVert_\infty$.
Similarly, integrating $\E [ -\overleftarrow{a}^{j,k}_t Z^{j,l}_t Z^{i,k}_t Z^{i,l}_t ]$ by parts gives 
\begin{multline*}
\E [ \partial_j \overleftarrow{a}^{j,k}_t Y^{l}_t Z^{i,k}_t Z^{i,l}_t + \overleftarrow{a}^{j,k}_t Y^{l}_t \Gamma^{i,j,k}_t Z^{i,l}_t + \overleftarrow{a}^{j,k}_t Y^{l}_t Z^{i,k}_t \Gamma^{i,j,l}_t + \overleftarrow{a}^{j,k}_t Y^j_t Y^{l}_t Z^{i,k}_t Z^{i,l}_t ]    \\
\geq - \E [ \varepsilon \Gamma^{i,j,l}_t \Gamma^{i,j,l}_t + \varepsilon^{-1} D [ 1 + Y^j_t Y^j_t] Z^{i,l}_t Z^{i,l}_t ].
\end{multline*}
Choosing $\varepsilon$ small enough compared to $C$, we obtain that
\[ \tfrac{\d}{\d t } \E [ Z^{i,l}_t Z^{i,l}_t ] \geq -\alpha \sup_{s \in [t,T]} \lVert \nabla^2 b_{T-s} \rVert^2_\infty - \alpha \E [ [ 1 + Y^j_t Y^j_t] Z^{i,l}_t Z^{i,l}_t ] + \beta \E [ \Gamma^{i,j,l}_t \Gamma^{i,j,l}_t ], \]
for some $\alpha, \beta > 0$. Using Lemma \ref{lem:Gamma_2.5}, a Gronwall-type argument concludes as previously.
\end{proof}

\section{Relative entropy and regularity on measures} \label{app:RepEntr}

We recall that the notion of reference system $\Sigma = (\Omega,(\F_t)_{t \leq 0 \leq T}, \P, (B_t)_{0 \leq t \leq T})$ is defined in Section \ref{subsec:Notations}.
For such a $\Sigma$ and any square-integrable progressively measurable process $\alpha = (\alpha_t)_{0\leq t\leq T}$ on $\Sigma$, we say that $X^{\alpha}_{[0,T]} := ( X^{\alpha}_t )_{0 \leq t \leq T}$ is a solution of the McKean-Vlasov SDE
\begin{equation} \label{eq:PAP2controlled_process}
\d X^{\alpha}_s = b_{s} ( X^{\alpha}, \mathcal{L}( X^{\alpha} ) ) \d s + \sigma_{s} (X^{\alpha}, \mathcal{L}( X^{\alpha}) ) \alpha_s \d s + \sigma_{s} (X^{\alpha}, \mathcal{L}( X^{\alpha}) )\d B_s,
\end{equation} 
if $( X^{\alpha}_t )_{0 \leq t \leq T}$ is $(\F_t)_{t \leq 0 \leq T}$-adapted and the integrated version of \eqref{eq:PAP2controlled_process} holds $\P$-a.s.
The following result is \cite[Lemma B.1]{chaintronLDP}.

\begin{lemma} \label{lem:repEntr}
Under \ref{ass:ini1}-\ref{ass:coefReg1}, for every measure $\mu_{[0,T]}$ in $\ps_1 (C([0,T],\R^d) )$, 
\[ H ( \mu_{[0,T]} \vert \Gamma ( \mu_{[0,T]} ) ) = \inf_\Sigma \inf_{\mathcal{L}(X^\alpha_{[0,T]}) = \mu_{[0,T]}} H(\L(X^\alpha_0) \vert \nu_0) + \E \int_0^T \frac{1}{2} \vert \alpha_t \rvert^2 \d t. \]    
where we minimise over $(X^{\alpha}_{[0,T]},\alpha)$ satisfying \eqref{eq:PAP2controlled_process} in the reference system $\Sigma$, with the convention that an infimum over an empty set equals $+\infty$. 
\end{lemma}

This results allows us to extend \cite[Lemma 5.2]{backhoff2020mean} to our setting.

\begin{corollary}[Range of finiteness] \label{cor:finiteRate}
Under \ref{ass:ini1}-\ref{ass:coefReg1}, let us fix any $\nu_{[0,T]} \in \ps_1 (C ( [0,T] , \R^d ) )$. 
Let $W^\nu_{[0,T]}$ denote the path-law of the solution to
\[ \d Y_t = b_t (Y_t, \nu_t) \d t + \sigma_t ( Y_t ) \d B_t, \qquad Y_0 \sim \nu_0. \]
Then for every $\mu_{[0,T]} \in \ps_1(C([0,T],\R^d))$, 
\[ H( \mu_{[0,T]} \vert \Gamma ( \mu_{[0,T]} )) < + \infty \quad \Longleftrightarrow \quad H( \mu_{[0,T]} \vert W^{\nu}_{[0,T]} ) < + \infty. \]
\end{corollary}

\begin{proof}
If $H( \mu_{[0,T]} \vert \Gamma ( \mu_{[0,T]} )) < + \infty$, \cite[Theorem 2.1]{leonard2012girsanov} provides an adapted square-integrable process $(c_t)_{0\leq t\leq T}$ on the canonical space $\Omega = C([0,T],\R^d)$ such that
\begin{equation*} 
\d {\sf X}_t = b_t ( {\sf X}_t, \mu_t) \d t + \sigma_t({\sf X}_t) {c}_t \d t + \sigma_t ( {\sf X}_t ) \d {\sf B}_t, \quad \mu_{[0,T]}\text{-a.s.},
\end{equation*} 
the process $( {\sf B}_t )_{0\leq t\leq T}$ being a Brownian motion under $\mu_{[0,T]}$.
We then apply Lemma \ref{lem:repEntr} with $x \mapsto b_t (x,\nu_t)$ instead of $b$ and $\alpha_t := c_t + \sigma^{-1}_t ( {\sf X}_t ) [ b_t ( {\sf X}_t,\mu_t) - b_t ( {\sf X}_t,\nu_t) ] $.
Since $b$ is Lipschitz-continuous in $\mu$ independently of $x$, this ensures that $H( \mu_{[0,T]} \vert W^{\nu}_{[0,T]}  )< + \infty$.

Reciprocally if $H( \mu_{[0,T]} \vert W^{\nu}_{[0,T]} ) < + \infty$, we similarly obtain a square-integrable $(c_t)_{0\leq t\leq T}$ such that
\begin{equation*} 
\d {\sf X}_t = b_t({\sf X}_t, \nu_t) \d t + \sigma_t({\sf X}_t) {c}_t \d t + \sigma_t ( {\sf X}_t ) \d {\sf B}_t, \quad \mu_{[0,T]}\text{-a.s.},
\end{equation*} 
where $( {\sf B}_t )_{0\leq t\leq T}$ is a Brownian motion under $\mu_{[0,T]}$.
Setting $\alpha_t := c_t + \sigma^{-1}_t ( {\sf X}_t) [ b_t({\sf X}_t, \nu_t) - b_t ( {\sf X}_t , \mu_t ) ]$, $H( \mu_{[0,T]} \vert \Gamma ( \mu_{[0,T]} )) < + \infty$ then results from Lemma \ref{lem:repEntr}.
\end{proof}

The following result is used to regularise linear derivatives of functions defined on $\ps_1 ( \R^d )$.
It is a slight variation of \cite[Lemma 4.2]{daudin2023optimalrate}.

\begin{lemma}[Mollification] \label{lem:mollif}
Let $\F : \ps_1 ( \R^d ) \rightarrow \R$ be Lispchitz. Let $\rho : \R^d \rightarrow [0,1]$ be a $C^\infty$ symmetric function with compact support and $\int_{\R^d} \rho(x) \d x = 1$. For $k \geq 1$, we set
\[ \forall x \in \R^d, \; \rho^k(x) := k^d \rho (k x), \quad \forall \mu \in \ps_1 ( \R^d ), \; \F^k(\mu) := \F (\rho^k \ast \mu). \]
Then $\F^k$ is Lipschitz-continuous uniformly in $k$ and
\[ \sup_{\mu \in \ps_1 (\R^d)} \vert \F^k ( \mu ) - \F(\mu) \vert \leq C k^{-1}, \]
for $C >0$ independent of $k$. Moreover, if $\F$ has a jointly continuous linear derivative $\frac{\delta\F}{\delta\mu}$, then so has $\F^k$ and
\[ \forall (x,\mu) 
\in \R^d \times \ps_1 (\R^d), \quad \frac{\delta\F^k}{\delta\mu}(\mu,x) = \bigg[ \frac{\delta\F^k}{\delta\mu}(\rho^k \ast \mu,\cdot) \ast \rho^k \bigg] (x). \]
As a consequence, $\frac{\delta\F^k}{\delta\mu}(\mu) : \R^d \rightarrow \R$ is $C^\infty$.
\end{lemma}

\section{PDE notion of solution results} \label{app:PDE}

The following notion of solution is extracted from \cite[Section 3.3]{ConstrainedSchrodinger}.

\begin{assumptionV}[On the reference SDE]\label{ass-c:SDE} 
The functions $b : [0,T] \times \R^d \rightarrow \R^d$, $\sigma : [0,T] \times \R^d \rightarrow \R^{d \times d}$, $\psi$, $c : [0,T] \times \R^d \rightarrow \R$ are measurable and locally bounded, and in addition: 
    \begin{enumerate}[label=(\roman*),ref=\roman*]
        \item\label{it:ass-c-SDE:1} uniformly in $t \in [0,T]$, $x \mapsto b_t(x)$ and $x \mapsto \sigma_t(x)$ are Lipschitz continuous;
        \item\label{it:ass-c-SDE:2} there exists $M_\sigma \geq 0$ such that $|\sigma_t(x)| \leq M_\sigma$ for all $(t,x) \in [0,T] \times \R^d$, and $t \mapsto \sigma_t (x)$ is locally Hölder-continuous;
        \item\label{it:ass-c-HJB:2} uniformly in $t \in [0,T]$, $c_t$ and $\psi_t$ are Lipschitz-continuous, and $\sigma_t^{-1}$ is bounded. 
    \end{enumerate}
\end{assumptionV}

We now fix a positive Radon measure $\lambda \in \M_+ ([0,T])$, and we want to make sense of the equation
\begin{equation} \label{eq:HJBLim} 
- \varphi_t + \int_t^T \left(b_s \cdot \nabla \varphi_s - \frac{1}{2} \vert \sigma_s^\top \nabla \varphi_s \vert^2 + \frac{1}{2} \mathrm{Tr}[a_s \nabla^2 \varphi_s] + c_s\right) \d s
+ \int_{[t,T]} \psi_s \lambda ( \d s) = 0.
\end{equation}
If $\lambda(\{t \}) \neq 0$, an arbitrary choice has been made when considering integrals over $[t,T]$ rather than $(t,T]$.
However, the set of atoms of $\lambda$ is at most countable; hence the choice of the interval does not matter if we only require equality Lebesgue-a.e. 
For approximation and stability purposes, we introduce a specific notion of solution, which relies on the following heat semi-group.

Under the assumptions made on $\sigma$ in~\ref{ass-c:SDE}, a consequence of \cite[Theorem 2.1]{rubio2011existence} is that for any $0 < s \leq T$ and any continuous $\varphi : \R^d \rightarrow \R$ with linear growth, the parabolic equation
\[
\begin{cases}
\partial_t \varphi_t + \tfrac{1}{2} \mathrm{Tr} \big[ a_t \nabla^2 \varphi_t \big] = 0, \quad 0 \leq t \leq s, \\
\varphi_s = \varphi,
\end{cases}
\]
has a unique solution $\varphi \in C([0,s] \times \R^d) \cap C^{1,2}((0,s) \times \R^d)$ with linear growth. 
From this, we define the evolution system $(S_{t,s})_{0\leq t\leq s\leq T}$ by
\[ S_{t,s} [ \varphi ] (x) = \varphi_t (x), \]
for any continuous $\varphi : \R^d \rightarrow \R$ with linear growth, $S_{t,s} [ \varphi ]$ being a $C^2$ function with linear growth as soon as $t < s$.

\begin{definition} \label{def:mildForm}
We say that a measurable $\varphi : [0,T] \times \R^d \rightarrow \R$ is a \emph{mild} solution of \eqref{eq:HJBLim}
if for Lebesgue-a.e. $t \in [0,T]$, $x \mapsto \varphi_t (x)$ is $C^1$, $(t,x) \mapsto \nabla \varphi_t (x)$ is bounded measurable, and for a.e. $t \in [0,T]$,
\begin{equation*} 
\varphi_t = \int_t^T S_{t,s} \big[ b_s \cdot \nabla \varphi_s - \tfrac{1}{2} \lvert \sigma_s^\top \nabla \varphi_s \rvert^2 + c_s \big] \d s
+ \int_{[t,T]} S_{t,s} [ \psi_s ] \lambda (\d s).
\end{equation*} 
This implies that $x \mapsto \varphi_t (x)$ is $C^2$ for Lebesgue-a.e. $t$.
\end{definition}

Lebesgue-almost sure uniqueness always holds for \eqref{eq:HJBLim} in the sense of Definition \ref{def:mildForm} because the difference of two solutions solves a classical linear parabolic equation without source term.

\section*{Acknowledgements}
\addcontentsline{toc}{section}{Acknowledgement}

The authors thank Julien Reygner for fruitful discussions throughout this work.
L.-P.C. also thanks Samuel Daudin and Simon Martin for advices on some technical issues. 
The work of G.C. is partially supported by the Agence Nationale de la Recherche through the grants ANR-20-CE40-0014 (SPOT) and ANR-23-CE40-0003 (Conviviality). 

\printbibliography
\addcontentsline{toc}{section}{References}

\end{document}